\documentclass[11pt, twoside]{article}

\usepackage{enumerate}
\usepackage{graphics}
\usepackage{graphicx}
\usepackage{amssymb}
\usepackage{amsmath}
\usepackage{amsthm}

\usepackage[title,titletoc]{appendix}

\usepackage{color}
\usepackage{mathrsfs}

\usepackage{indentfirst}

\usepackage{txfonts}

\usepackage{anysize}

\allowdisplaybreaks

\newtheorem{theorem}{Theorem}[section]
\newtheorem{lemma}[theorem]{Lemma}
\newtheorem{corollary}[theorem]{Corollary}
\newtheorem{proposition}[theorem]{Proposition}
\theoremstyle{definition}
\newcounter{assum}

\newtheorem{assumption}[assum]{Assumption}
\newtheorem{remark}[theorem]{Remark}

\newtheorem{definition}[theorem]{Definition}
\renewcommand{\appendix}{\par
	\setcounter{section}{0}%
	\setcounter{subsection}{0}%
	\setcounter{subsubsection}{0}%
	\gdef\thesection{\@Alph\c@section}%
	\gdef\thesubsection{\@Alph\c@section.\@arabic\c@subsection}%
	\gdef\theHsection{\@Alph\c@section.}%
	\gd\theHsubsection{\@Alph\c@section.\@arabic\c@subsection}%
	\csname appendixmore\endcsname
}

\newtheorem{letterthm}{Theorem}

\numberwithin{equation}{section}

\begin{document}

\title{\bf\Large Fractional Heat Semigroup Characterization of Distances from
Functions in Lipschitz Spaces to Their Subspaces
\footnotetext{\hspace{-0.35cm} 2020 {\it
Mathematics Subject Classification}. Primary 46E35;
Secondary 26A16, 35K08,
42C40, 42E35.
\endgraf {\it Key words and phrases.} fractional heat semigroup,
distance,
Daubechies wavelet,
difference, Lipschitz Space, Besov--Triebel–-Lizorkin space.
\endgraf This project is supported by the National Key Research
and Development Program of China
(Grant No. 2020YFA0712900), the National Natural Science
Foundation of China
(Grant Nos. 12431006 and 12371093),
and the Fundamental Research Funds for the Central
Universities (Grant Nos. 2253200028 and 2233300008).
The first author is also supported by
NSERC of Canada Discovery grant RGPIN-2020-03909.}}
\author{Feng Dai, Eero Saksman\footnote{Corresponding author, E-mail: \texttt{eero.saksman@helsinki.fi}/{\color{red}\today}/Final version.},\ \
Dachun Yang,
Wen Yuan and Yangyang Zhang}
\date{}
\maketitle

\vspace{-0.7cm}

\begin{center}
\begin{minipage}{13cm}
{\small {\bf Abstract}\quad
Let $\Lambda_s$ denote
the inhomogeneous Lipschitz space of order $s\in(0,\infty)$ on $\mathbb{R}^n$. This article characterizes the distance
$d(f, V)_{\Lambda_s}: =
\inf_{g\in V} \|f-g\|_{\Lambda_s}$ from a function $f\in \Lambda_s$ to 
a non-dense subspace $V\subset \Lambda_s$ via the fractional semigroup
$\{T_{\alpha, t}: =e^{-t (-\Delta)^{\alpha/2}}: t\in (0, \infty)\}$ for 
any $\alpha\in(0,\infty)$.
 Given an integer $ r >s/\alpha$, a uniformly bounded continuous function
$f$ on $\mathbb{R}^n$
belongs to the space $\Lambda_s$
if and only if there exists a constant
$\lambda\in(0,\infty)$ such that
	\begin{align*}		 \left|(-\Delta)^{\frac {\alpha r}2} 
(T_{\alpha, t^\alpha } f)(x) \right|\leq \lambda t^{s -r\alpha }\ 
\ \text{for any $x\in\mathbb{R}^n$ and $t\in (0, 1]$}.
	\end{align*}
The least such constant
 is denoted by $\lambda_{ \alpha, r, s}(f)$.
 For each $f\in \Lambda_s$ and $0<\varepsilon<
\lambda_{\alpha,r, s}(f)$, let
$$ D_{\alpha, r}(s,f,\varepsilon):=\left\{ (x,t)\in
\mathbb{R}^n\times (0,1]:\
\left| (-\Delta)^{\frac {\alpha r}2} (T_{\alpha, t^\alpha} f)(x)
\right|> \varepsilon t^{s -r \alpha }\right\}$$
 be the set of ``bad'' points. To quantify its size,
we introduce a class of extended nonnegative
\emph{admissible set functions}
$\nu$ on the Borel $\sigma$-algebra
$\mathcal{B}(\mathbb{R}^n\times [0, 1])$
and define, for any admissible function $\nu$, the
\emph{critical index}
$ \varepsilon_{\alpha, r, s,\nu}(f):=\inf\{\varepsilon\in(0,\infty):\
\nu(D_{\alpha, r}(s,f,\varepsilon))<\infty\}.$
Our result shows that, for a broad class of subspaces $V\subset \Lambda_s$,
including intersections of $\Lambda_s$ with Sobolev, Besov, Triebel--Lizorkin, and
Besov-type spaces, there exists an admissible function $\nu$ depending on $V$
such that
$\varepsilon_{\alpha, r, s,\nu}(f)\sim \mathrm{dist}(f, V)_{\Lambda_s}.$
}
\end{minipage}
\end{center}

\vspace{0.2cm}

\tableofcontents

\vspace{0.2cm}

\section{Introduction}

Let $\mathop\mathrm{\,BMO\,}$ denote the space of all locally
integrable functions $f$ on $\mathbb{R}^n$ with bounded mean
oscillation
$$
\|f\|_{\mathop\mathrm{\,BMO\,}}:=\sup _B \frac{1}{|B|}
\int_B\left|f(x)-f_B\right|\, d x<\infty,
$$
where the supremum is taken over all Euclidean balls $B$ in
$\mathbb{R}^n$ and
$
f_B:=\frac{1}{|B|} \int_B f(x)\, d x$.
As is well known, $L^{\infty}\subsetneqq \mathrm{BMO}$.
In 1978, Garnett and Jones \cite{gj78} characterized
 the distance
$$ d \left(f,
L^{\infty}\right)_{\mathop\mathrm{\,BMO\,}}:=\inf _{g \in
L^{\infty}}\|f-g\|_{\mathop\mathrm{\,BMO\,}}\ \text{for}\ \
f\in \mathop\mathrm{\,BMO\,}$$
by the infimum of the constant $\delta\in(0,\infty)$ in the
 John--Nirenberg inequality (see \cite{JN}),
\begin{equation}\label{1-1-b}
\sup _B \frac{\mid\left\{x \in
B:\left|f(x)-f_B\right|>\lambda\} |\right.}{|B|} \leq e^{-
\frac{\lambda}{\delta}} \quad \text { for large} \ \lambda,
\end{equation}
proving that
\begin{equation*} d \left(f,
L^{\infty}\right)_{\mathop\mathrm{\,BMO\,}}\sim_n
\delta(f):=\inf\left \{\delta\in(0,\infty): \eqref{1-1-b} \
\text {holds}\right\}. \end{equation*}
Here and throughout the article, $|E|$ denotes the Lebesgue
measure of $E\subset \mathbb{R}^n$, the symbol $A\sim_{p,q,r}
D$ means that there exist positive constants $C_1, C_2$
depending only on the subscripts $p,q,r,\ldots$ such that $C_1
A\leq D\leq C_2 A$.
This remarkable result has several important applications.
First, it shows that, for any $f\in \mathop\mathrm{\,BMO\,}$,
	$$\delta(f)\sim_n\inf
\left\{\sum_{j=1}^n\left\|g_j\right\|_{\infty}: \ g_0,
g_1,\ldots, g_n \in L^{\infty}\ \text{and}\ f=g_0+\sum_{j=1}^
n R_j g_j\right\},$$
	which complements the Fefferman--Stein characterization
that $f\in\mathop\mathrm{\,BMO\,}$ if and only if there exist
$g_0, g_1,\ldots, g_n\in L^\infty$ such that
	$f=g_0+\sum_{j=1}^n R_j g_j$ ( \cite[pp.\,
374--375]{gj78}). Here $R_j$ is the $j$-th Riesz transform.
Second, it can be used to generalize the classical
Helson--Szeg\"o theorem to higher dimensions, involving
weights $w$ such that the Riesz transforms are bounded on
$L^2(wdx)$ (\cite[Corollary 1.2]{gj78}).
 Third, using it, Jones \cite{j} derived
 sharp estimates for solutions to the corona problem for
$H^\infty$ of the upper half-plane or unit disk, and in
\cite{j80} he established the
factorization
of $A_p$ weights. Lastly, Bourgain \cite{b} applied it to
study embedding $L^1$ in $L^1 / H^1$, proving that $L^1$ is
isomorphic to a subspace of $L^1 / H^1$. For more details, we
refer to \cite{cdlsy} and the references therein.

In \cite{ss},
Saksman and Soler i Gibert studied a similar problem for
distances in the nonhomogeneous Lipschitz space
$\Lambda_s=(\Lambda_s, \|\cdot\|_{\Lambda_s})$ of
smooth order $s\in(0,\infty)$ on $\mathbb{R}^n$, seeking to
determine the distance (up to a positive constant
multiple) in $\Lambda_s$ for all $f\in \Lambda_s$ and
$s\in(0,1]$:
\begin{align*}
 d \left(f, \mathrm{J}_s(\mathop\mathrm{\,bmo\,})
\right)_{\Lambda_s}:=&
\inf_{g\in \mathrm{J}_s(\mathop\mathrm{\,bmo\,})}
\|f-g\|_{\Lambda_s}.
\end{align*}
Here, $\mathrm{J}_s(\mathop\mathrm{\,bmo\,})$ is
the image of the space $\mathop\mathrm{\,bmo\,}$ under
the Bessel potential $\mathrm{J}_s:=(I-\Delta)^{-\frac s2}$,
and $\mathop\mathrm{b m o}$
is a nonhomogeneous version of the space
$\mathop\mathrm{\,BMO\,}$ consisting of all functions $f\in
\mathop\mathrm{\,BMO\,}$ with
$$ \|f\|_{\mathop\mathrm{\,bmo\,}}
:=\|f\|_{\mathop\mathrm{\,BMO\,}}
+\sup_{k\in\mathbb{Z}^n}
\left[ \int_{k+[0,1)^n} |f(x)|^2\, dx \right]^{\frac 12}
<\infty.
$$

 To state the results in \cite{ss},
we begin by introducing some necessary symbols.
For any $k\in\mathbb N$ and $h\in \mathbb{R}^n$,
the $k$-th order \emph{symmetric
difference operator} $\Delta_h^k$ is defined
recursively as follows: for any
$f:\mathbb{R}^n\to\mathbb{R}$ and $x\in\mathbb{R}^n$,
\begin{equation}\label{diff-eq}
\Delta_h(f)(x) :=
f\left(x+\frac h2\right)-f\left(x-\frac h2\right),\
\ \Delta^{i+1}_h:
= \Delta^i_h(\Delta_h),\ i\in\mathbb{N}.
\end{equation}
Given $s\in(0,\infty)$,
the \emph{Lipschitz space}
$\Lambda_s$ on $\mathbb{R}^n$
is the Banach space of all continuous
functions $f:\mathbb{R}^n\to\mathbb{R}$ such that
\begin{align*}
\|f\|_{\Lambda_s}:=
\|f\|_{L^\infty}+
\sup_{y\in(0,1]}\sup_{x\in\mathbb{R}^n}
\frac {\Delta_r f(x,y)}{y^s}<\infty,
\end{align*}
where $r=\lfloor s\rfloor+1$, $\lfloor s\rfloor$
denotes the \emph{largest integer not greater than} $s$,
and
\begin{align*}
\Delta_k f(x,y)=\sup_{h\in\mathbb{R}^n, |h|=y}
|\Delta_h^kf(x)|,\ x\in \mathbb{R}^n,\
\ y\in(0,\infty),\ k\in\mathbb{N}.
\end{align*}
It is well known that replacing
$r=\lfloor s\rfloor+1$ with any positive
integer $r>s$ in the definition of
$\|\cdot\|_{\Lambda_s}$ results in an
equivalent norm.
Below we will fix $s\in(0,\infty)$ and an integer $r>s$. Note
that by the definition, $f\in L^\infty$ belongs to the space
$\Lambda_s$ if and only if there exists a constant
$L\in(0,\infty)$ such that
$$ \Delta_r f(x, y)\leq L y^s\ \text{for all}\
x\in\mathbb{R}^n\ \text{and}\ \
y\in(0,1].$$

Given $f\in\Lambda_s$, we define $\varepsilon_{r,s}(f)$ to be
the infimum
of all $\varepsilon\in(0,\infty)$ such that
\begin{equation}
	\sup_{I\in\mathcal D} \frac 1 {|I|} \int_0^{\ell(I)}
 \left|\left\{ x\in I:\ \Delta_r f(x,y)>\varepsilon
 y^s\right\}\right|\frac {dy} y<\infty.\label{1-1-2}
\end{equation}
Here and throughout the article,
$\mathcal D$ denotes the collection of all dyadic
cubes $I$ in $\mathbb{R}^n$ of edge length
$\ell(I)\leq 1$.
For any $f\in\Lambda_s$ and $\varepsilon\in(0,\infty)$, we
also define the set \begin{align*}
S_r(s,f, \varepsilon):&=\left\{ (x, y)\in\mathbb{R}_+^{n+1}
:=\mathbb{R}^n\times(0,\infty):
\Delta_r f(x, y)>\varepsilon y^s,\ x\in\mathbb{R}^n \right\},
\end{align*}
and the measure
$$
d\gamma_{r, s,f,\varepsilon}(x,y)
:= \frac {\mathbf{1}_{S_r(s,f, \varepsilon)}
(x,y)\mathbf{1}_{[0, 1]}(y)\,dx\,dy}y,\
(x,y)\in\mathbb{R}_+^{n+1}.
$$
It is easily seen that
\eqref{1-1-2} holds if and only if
$\gamma_{r, s,f,\varepsilon}$ is a Carleson measure on
$\mathbb{R}^{n+1}_+$.
This means that $\varepsilon_{r,s}(f)$ is the \emph{critical
index} such that $\gamma_{r, s, f,\varepsilon}$ is a Carleson
measure whenever
$\varepsilon\in(\varepsilon_{r,s}(f),\infty)$.
Remarkably, Saksman and Soler i Gibert
\cite{ss} proved that, for any $s\in(0,1]$,
the critical index $\varepsilon_{r,s}(f)$ characterizes
the distance $ d (f,
\mathrm{J}_s(\mathop\mathrm{\,bmo\,}))_{\Lambda_s}$
for any $f\in\Lambda_s$, showing the following theorem.

\begin{letterthm}(\cite{ss})\label{fwfw}
If $s\in(0,1)$ and $r = 1$ or if $s\in(0,1]$
and $r = 2$, then,
for any $f\in\Lambda_s$,
\begin{equation*}
 d (f,\mathrm{J}_s(\mathop\mathrm{\,bmo\,}))_{\Lambda_s}\sim_{s,n,r}
\varepsilon_{r,s}(f):= \inf\left\{ \varepsilon\in(0,\infty):\
\text{\eqref{1-1-2} holds} \right\}.
\end{equation*}
\end{letterthm}

Several remarks are in order. First, we point out that Theorem
A was previously established by Nicolau
and Soler i Gibert in
\cite{ns} for $n=s=1$, however with a proof
that can not be extended to higher-dimensional cases.
Second, Theorem A yields the following characterization
of the closure of
$\mathrm{J}_s(\mathop\mathrm{\,bmo\,})$
under the topology of $\Lambda_s$:
$$ f\in\overline{\mathrm{J}_s
(\mathop\mathrm{\,bmo\,})}^{\Lambda_s}\iff
f\in\Lambda_s\ \text{and}\ \varepsilon_{r,s}(f)=0.$$
Here and throughout the article, we denote
the closure of a set $E\subset \Lambda^s$ under the topology
of $\Lambda_s$ by $ \overline{E}^{\Lambda_s}$.
Clearly, for any $E\subset \Lambda_s$,
$$\overline{E}^{\Lambda_s}
=\left\{ f\in\Lambda_s:\
 d (f, E)_{\Lambda_s}
=0\right\}\ \text{with}\ d \left(f,
E\right)_{\Lambda_s}:=
\inf_{g\in E} \|f-g\|_{\Lambda_s}.$$
In the case when $E$ does not lie entirely
in the space $\Lambda_s$, we will still define with a slight abuse of the
symbol
$$d\left(f, E\right)_{\Lambda_s}:=	
 d \left(f, E\cap
\Lambda_s\right)_{\Lambda_s}\ \text{and}\ \
\overline{E}^{\Lambda_s}
:=\overline{E\cap \Lambda_s}^{\Lambda_s}. $$

In a recent article \cite{dsy}, we significantly extended
Theorem A by using higher-order finite differences to
characterize the distances $ d (f, V)_{\Lambda_s}$ from $f\in
\Lambda_s$ to various non-dense subspaces $V\cap \Lambda_s$ of
 $\Lambda_s$
for the full range of $s\in(0,\infty)$.
The spaces $V$ considered in \cite{dsy} include the space
$\mathrm{J}_s(\mathop\mathrm{\,bmo\,})$ for all
$s\in(0,\infty)$,
the Besov spaces $B_{p,q}^s$ and the Triebel--Lizorkin spaces
$F_{p,q}^s$ for all $s\in(0,\infty)$ and $p, q\in (0,
\infty]$, the Besov-type spaces,
and the Triebel--Lizorkin-type spaces.

For reader's convenience, below we summarize the main results
of \cite{dsy} for the
aforementioned specific well-known
function spaces. Assume that $s\in(0,\infty)$ and $r>s$ is a
given integer. Let $p \in(0,\infty)$ and $q \in(0,\infty]$.
Let $\mathcal{C}_0$ denote the space of all
continuous functions $f:\mathbb{R}^n\to\mathbb{R}$
such that $\lim_{|x|\to\infty} f(x)=0$. The following results
 were proved in \cite{dsy}:
\begin{itemize}
\item For the space
$\mathrm{J}_s(\mathop\mathrm{\,bmo\,})$ and $f\in\Lambda_s$,
 the conclusion of Theorem A holds for the full range of
 $s\in(0,\infty)$.

\item For the Sobolev space $W^{1,p}$
and
$f\in \mathcal{C}_0\cap
\Lambda_1$,
\begin{align*}
 d (f, W^{1,p})_{\Lambda_s}
\sim_{s,p,n}
\inf \left\{ \varepsilon\in(0,\infty):\
\left\| \int_{\left\{y\in (0,1]:y^{-1}
\Delta_2 f(\cdot, y)>\varepsilon \right\}}
\frac {dy} y\right\|_{L^{\frac p2}}
<\infty \right\}.
\end{align*}

\item For the Besov space $B_{p,q}^s$
 and
$f\in \mathcal{C}_0\cap
\Lambda_s$,
\begin{align*}
&	
 d
(f, B_{p, q}^s)_{\Lambda_s}
\sim_{r, s,p,q, n}
\inf \left\{ \varepsilon\in(0,\infty):\
B_{p,q,\varepsilon}^s(f) <\infty\right\},
\end{align*}
where
\begin{equation*}
B_{p,q,\varepsilon}^s(f):= \left\|
\left\{ \int_{2^{-j-1}}^{2^{-j}} \left|
\left\{ x\in{\mathbb R}^n:\
y^{-s} \Delta_r f(x, y)>\varepsilon
\right\} \right|\,\frac {dy} y \right\}_{j=0}^\infty
 \right\|_{\ell^{\frac qp}({\mathbb Z}_+)}.
\end{equation*}
\item For the Triebel--Lizorkin space
$F_{p,q}^s$ and
$f\in \mathcal{C}_0\cap
\Lambda_s$,
\begin{align*}
 d (f, F_{p, q}^s)_{\Lambda_s}
\sim_{r, s,p,q, n}
\inf \left\{ \varepsilon\in(0,\infty):\
F_{p, q,\varepsilon}^s(f) <\infty \right\},
\end{align*}
where
$$ F_{p, q,\varepsilon}^s(f):=
\left\| \int_{\left\{y\in [0,1]:\
y^{-s} \Delta_r f(x, y)>\varepsilon \right\}}
\frac {dy} y\right\|_{L^{\frac pq}}.$$
\end{itemize}

This article is highly motivated by the work in \cite{ss} and
serves as a continuation of the article \cite{dsy}. Our main
purpose is to
characterize the distance
 $\mathrm{dist}
(f, V)_{\Lambda_s}: =
\inf_{g\in V} \|f-g\|_{\Lambda_s}$ from $f\in \Lambda_s$ to a
non-dense subspace $V\subset \Lambda_s$ in terms of the
fractional heat semigroup
\begin{align}\label{fdsmf}
T_{\alpha, t}: =e^{- t (-\Delta)^{\alpha/2}}, \ \ t\in (0,
\infty),\ \ \alpha\in(0,\infty),
\end{align}
where $\Delta=\sum_{j=1}^n \partial_j^2$ with
$\partial_j:=\frac{\partial}{\partial x_j}$ for any
$j\in\{1,\ldots,n\}$
denotes the \emph{Laplacian}
on $\mathbb{R}^n$. To proceed, let us recall briefly some
symbols and definitions.

For any $\alpha\in(0,\infty)$,
 let $(-\Delta)^{\alpha/2}$ denote the fractional power of the
 Laplacian $-\Delta$
on $\mathbb{R}^n$ defined,
in a distributional sense, by
setting, for any
$\xi\in\mathbb{R}^n$,
$$ \mathcal{F}\left((-\Delta)^{\alpha/2} f\right)(\xi):=|2\pi
\xi|^{ \alpha} 	\mathcal{F} f(\xi).
$$
Here and throughout the article, $\mathcal{F}$ denotes the
usual \emph{Fourier
transform} on $\mathbb{R}^n$ defined by setting, for any $f\in
L^1 $ and $\xi\in\mathbb{R}^n$,
\begin{align*}
	\mathcal{F}f(\xi)\equiv\widehat{f}(\xi)
:=\int_{\mathbb{R}^n}
	f(x)e^{-2\pi ix\cdot \xi}\,dx.
\end{align*}
Both the Fourier transform $\mathcal{F}$ and its inverse
$\mathcal{F}^{-1}$ extend to the space $\mathcal{S}' $ of
tempered distributions in the standard way.
Equivalently, in terms of the Fourier transform,
the fractional heat semigroup $T_{\alpha, t}$ in \eqref{fdsmf}
satisfies
\begin{align}\label{1-6-1} \mathcal{F}(T_{\alpha, t} f)(\xi)
=e^{-t|2\pi \xi|^{\alpha}} \mathcal{F}f(\xi),\
\xi\in\mathbb{R}^n.
\end{align}
Clearly, if $\alpha=1$, then $T_{\alpha, t}$ is the Poisson
semigroup on $\mathbb{R}^n$, while, for $\alpha=2$,
$T_{\alpha,2}$ is the usual heat semigroup on $\mathbb{R}^n$.

The fractional heat semigroup $T_{\alpha,t}$ can be used to
solve the fractional heat equation,
\begin{align}\label{1-6-0} \partial_t u(x,t)
+(-\Delta)^{\alpha/2} u(x,t)=0, \ \ x\in\mathbb{R}^n,\ \
t\in(0,\infty), \end{align}
which describes diffusion processes that are non-local and
exhibit anomalous behavior such as long-range interactions.
Clearly, the solution of this equation is
$ u(x,t) =T_{\alpha,t} u_0(x)$, where $u_0(x):=u(x,0)$ is the
initial condition.
The fractional heat semigroup has also been used extensively
to solve other fractional partial differential equations that
appear in probability, financial mathematics,
elasticity, and biology. For details, we refer to Section 1.1
of \cite{PS}.

The fractional heat semigroup plays a significant role in
characterizing various function spaces, including Besov,
Triebel--Lizorkin, and Sobolev spaces; see \cite{bb, RBCR1,
RBCR2} and the references therein.
 For example, given $s\in(0,\infty)$ and any integer $ r
 >s/\alpha$, a function $f\in \mathcal{C} \cap L^\infty $
 belongs to the space $\Lambda_s$ if and only if there exists
 a constant $\lambda\in(0,\infty)$ such that the estimate
 \begin{align*} \left|(-\Delta)^{\frac {\alpha r}2}
 (T_{\alpha, t^\alpha } f)(x) \right|\leq \lambda t^{s-\alpha r}
\end{align*}
holds for all $x\in\mathbb{R}^n$ and $0<t\leq 1$, where
$\mathcal{C}$ denotes the space of all continuous functions on
$\mathbb{R}^n$. In this case, we also have
\begin{align}\label{1-6}
	\|f\|_{\Lambda_s}\sim_{r,s,n} \|f\|_{L^\infty }+ \sup_{
t\in (0,\infty)}\sup_{x\in \mathbb{R}^n} t^{\alpha r-s}
\left|(-\Delta)^{\frac {\alpha r}2} (T_{\alpha, t^\alpha }
f)(x) \right|.
\end{align}

Our aim in this article is to establish results analogous to
those in \cite{dsy} with $y^{-s} \Delta_r f(x, y)$ replaced by
$y^{r-\frac s\alpha}| (-\Delta)^{\frac {\alpha r}2}
(T_{\alpha, y} f) (x) |$, within the unified framework
introduced in \cite{dsy}.
The main results will be described in the next section.

The remainder of this article is organized as follows. In
Section~\ref{sec:2}, we summarize the main results of this
article along with necessary background, using minimal
symbols. In particular, we formulate our main results more
precisely for several classical function spaces, including
Sobolev, Besov, Triebel--Lizorkin, and Besov-type spaces.

In Section~\ref{sec:3}, we describe in detail a unified
general framework within which our main results are
established. This framework was first introduced in
\cite{dsy}.
Theorems~\ref{thm-2-66} and \ref{thm-7-11} provide precise
formulations of our main results in this setting, where a
function subspace $V\subset \Lambda_s$ of smooth
order $s\in(0,\infty)$ is referred to as the
\emph{Daubechies $s$-Lipschitz $X$-based space}. This space is
defined via Daubechies wavelet
expansions and a given quasi-normed
lattice $X$ of function sequences.

Sections~\ref{sec:4}--\ref{sec:10} are devoted to the proofs
of our main results within this general framework. Given the
length and complexity of the arguments, the proofs are divided
across several sections. A key technical tool is the family of
higher-order ball average operators. Specifically,
Section~\ref{sec:4} collects several useful results from our
earlier article \cite{dsy}, which will be used repeatedly in
the subsequent analysis.

In Section \ref{sec-5}, we establish pointwise kernel
estimates for the fractional heat semigroup $T_{\alpha, t}$
with $\alpha, t\in (0,\infty)$, and use them to derive uniform
norm estimates for the derivatives $\partial_t^i
\partial_x^\beta [T_{\alpha, t} f(x)]$,
 where $f\in\Lambda_s$, $\beta\in\mathbb{Z}_+^n$ and
 $i\in\mathbb{Z}_+$.
In particular, we characterize the space $\Lambda_s$
 in terms of the uniform norms of these derivatives. These
 results play a crucial role in the proofs of the main
 theorems in later sections.

Section~\ref{oijjl} outlines the main steps of the proof of
Theorem~\ref{thm-2-66}, our first main result. This proof
relies on three major technical propositions, which are also
formulated in Section~\ref{oijjl}, but are proved in
 subsequent sections. A complete proof of
 Theorem~\ref{thm-2-66}, conditional on these propositions, is
  also given here.

 The proofs of the three key propositions are carried out in
 Sections~\ref{sec:7}--\ref{sec:9}. These arguments are quite
 technical, but constitute the core of the proof of
 Theorem~\ref{thm-2-66}.

 Section~\ref{sec:10} presents the proof of
 Theorem~\ref{thm-7-11}, which is deduced from the
 aforementioned
 three propositions established in
 Sections~\ref{sec:7}--\ref{sec:9}.

 Finally, in Section~\ref{sea8}, the last section, we
 illustrate how our general results apply to various classical
 function spaces. In each case, it remains to verify that the
 space in question satisfies the assumptions required by our
 general framework,
 a task largely accomplished in our earlier work \cite{dsy}.

We conclude this section with some
conventions on symbols. Throughout the article,
since we work in $\mathbb{R}^n$,
we omit the underlying space $\mathbb{R}^n$
in all the related symbols whenever
there is no risk of ambiguity.
Let $\mathbb{N}:=\{1,2,\ldots\}$
and $\mathbb{Z}_+:=\mathbb{N}\cup\{0\}$.
We use $\mathbf{0}$ to denote the
\emph{origin} of $\mathbb{R}^n$. All
functions on $\mathbb{R}^n$ and
subsets of $\mathbb{R}^n$ are
assumed to be Lebesgue measurable.
Given a set $E \subset \mathbb{R}^n$, we denote its Lebesgue
measure by $|E|$ and its characteristic function by
$\mathbf{1}_E$.
For any measurable set $E \subset \mathbb{R}^n$ with $|E| <
\infty$ and any locally integrable function $f \in
L^1_{\mathrm{loc}}$, we define the average of $f$ over $E$ by
$$f_E=\fint_Ef(x)\,dx:=
\frac1{|E|}\int_{E}f(x)\,dx.$$
We denote the upper half-space $\mathbb{R}^n \times (0,
\infty)$ by $\mathbb{R}_+^{n+1}$. For any $x \in \mathbb{R}^n$
and $r \in(0,\infty)$, let $B(x, r) := \{y \in \mathbb{R}^n :
|x - y| < r\}$. We write $\mathbb{B}$ for the collection of
all Euclidean balls $B(x, r)$ with $x \in \mathbb{R}^n$ and $r
\in (0, \infty)$. For any $\gamma \in(0,\infty)$ and any ball
$B := B(x_B, r_B)$, we write $\gamma B := B(x_B, \gamma r_B)$.
We use the letter $C$ to denote a general positive constant,
which may vary from line to line but may depend on parameters
indicated by subscripts.
The symbol $f \lesssim g$ means that $f \leq Cg$ for some
constant $C \in(0,\infty)$. If both $f \lesssim g$ and $g
\lesssim f$ hold, we write $f \sim g$. If $f \leq Cg$ and $g =
h$ or $g \leq h$, we write $f \lesssim g = h$ or $f \lesssim g
\leq h$, respectively. For a finite set $A$, we denote its
cardinality by $\sharp A$.
Finally, in all proofs we consistently retain the symbols
introduced in the original theorem (or related statement).

\section{Summary of main results}\label{sec:2}

Our goal is to characterize the distance
 $\mathrm{dist}
(f, V)_{\Lambda_s}: =
\inf_{g\in V} \|f-g\|_{\Lambda_s}$ from a function $f\in
\Lambda_s$ to a non-dense subspace $V\subset \Lambda_s$ in
terms of the fractional heat semigroup $T_{\alpha, t} =e^{-t
(-\Delta)^{\alpha/2}}$, $t\in(0,\infty)$.
Let us describe our problem more precisely as follows. Let
$f\in \mathcal{C}\cap L^\infty$.
For convenience, we let $$u_\alpha(x, t): =T_{\alpha, t}
f(x),\ \ x\in \mathbb{R}^n,\ \ t\in(0,\infty).$$
Let $r\in\mathbb{N}$ be such that $\alpha r>s$.
By definition,
$$ \partial_{n+1}^r u_\alpha(x, t^\alpha)= (-1)^r
(-\Delta)^{\frac {\alpha r}2} (T_{\alpha, t^\alpha} f)(x), $$
here
and
thereafter, $\partial_{n+1}:=\partial_t$.
Thus, by \eqref{1-6}, $f\in\Lambda_s$ if and only if
there exists a constant $\lambda\in(0,\infty)$ such that the
estimate
\begin{equation}\label{2-1-eq} |\partial_{n+1}^r u_\alpha(x,
t^\alpha)|\leq \lambda t^{s-\alpha r} \end{equation}
holds for all $x\in\mathbb{R}^n$ and $t\in (0, 1]$.
We denote by $\lambda_{\alpha, r,s}(f)$ the least such
constant.
 For any $0<\varepsilon<
\lambda_{\alpha,r,s}(f)$, define
\begin{align}\label{2-2:eq} D_{\alpha,
r}(s,f,\varepsilon):=\left\{ (x,t)\in
\mathbb{R}^n\times (0,1]:\
 |\partial_{n+1}^r u_\alpha(x, t^\alpha)|>\varepsilon
 t^{s-\alpha r}\right\}\end{align}
 to be the set of all points $(x,t)\in\mathbb{R}^n\times (0,
 1]$ at which \eqref{2-1-eq} with $\varepsilon$ in place of
 $\lambda$ fails.
 To quantify the size of this set,
we introduce a class of extended nonnegative valued
\emph{admissible set functions}
$\nu$ on the Borel $\sigma$-algebra
$\mathcal{B}(\mathbb{R}^n\times [0, 1])$
such that $\nu(\emptyset)=0$ and
$\nu(E_1)\leq \nu(E_2)$ whenever $E_1,
E_2\in \mathcal{B}(\mathbb{R}^n\times
[0, 1])$ and $E_1\subset E_2$.
Given each admissible function $\nu$, we define
$$ \varepsilon_{\alpha,
r,s,\nu}(f):=\inf\{\varepsilon\in(0,\infty):\
\nu(D_{\alpha, r}(s,f,\varepsilon))<\infty\}.$$
That is, $\varepsilon_{\alpha,r,s,\nu}(f)$ is the
\emph{critical index} for which
$\nu(D_{\alpha,r}(s,f,\varepsilon))<\infty$ for any
$\varepsilon>\varepsilon_{\alpha,r,s,\nu}(f)$.
Our main question is: what conditions should be imposed on a
subspace $V$ of $ \Lambda_s$ to guarantee the
existence of an admissible function $\nu$ such
that, for any $f\in \Lambda_s$, the distance
$\mathrm{dist}(f, V)_{\Lambda_s}$ can be characterized
by the critical index $\varepsilon_{\alpha,r,s,\nu}(f)$?

We will give a partial answer to this question in a unified
framework, showing that, under certain conditions
on a given subspace $V\subset \Lambda_s$, there always
exists an admissible function $\nu$ which can be expressed
explicitly
such that
$\varepsilon_{\alpha,r,s,\nu}(f)\sim \mathrm{dist}(f,
\Lambda_s)_{\Lambda_s}$.
Our general result applies to a broad class of subspaces
$V\subset \Lambda_s$, including intersections of $\Lambda_s$
with Sobolev, Besov, Triebel--Lizorkin, and
Besov-type spaces.
 Our results for these spaces will be in full analogy with the
 corresponding results proved in \cite{dsy} , where we need to
 replace $ y^{-s}\Delta_r f(x,y)$ by $ y^{\alpha r-s} |
 \partial_{n+1}^r u_\alpha (x, y^\alpha) |$. More
 specifically, the following results for these specific
 function spaces can be deduced directly from our general
 result:
Let $r\in\mathbb{N}$ be such that $\alpha r>s+3$, and let
$u_\alpha(x, y): =T_{\alpha, y} f(x)$ for any $f\in \Lambda_s$
and $(x,y)\in\mathbb{R}_+^{n+1}$.
 \begin{itemize}
 	\item For the space
 $\mathrm{J}_s(\mathop\mathrm{\,bmo\,})$ and any $f\in
 \Lambda_s$,
 	\begin{equation} \label{1-8b}	 d (f,
 \mathrm{J}_s(\mathop\mathrm{\,bmo\,}))_{\Lambda_s} \sim
 \inf\left\{ \varepsilon\in(0,\infty):\
 \mathrm{J}_{\alpha,r, s,\varepsilon}(f)
 <\infty\right\},\end{equation}
 where
 \begin{equation*}
\mathrm{J}_{\alpha,r, s,\varepsilon}(f) :=
\sup_{I\in{\mathcal D}} \frac 1 {|I|} \int_0^{\ell(I)}
\left|\left\{ x\in I:\
	 \left| \partial_{n+1}^r u_\alpha (x, y^\alpha) \right|
>\varepsilon y^{s-\alpha r}\right\}\right|\frac {dy}
y.\end{equation*}
 	\item For the Besov space $B_{p,q}^s$ with
 $p\in(0,\infty)$ and $q\in(0,\infty]$ and
 for any
 	$f\in \mathcal{C}_0 \cap \Lambda_s$,
 	\begin{align*}
 		&	 d (f, B_{p, q}^s)_{\Lambda_s}
 		\sim \inf \left\{ \varepsilon\in(0,\infty):\
 B_{\alpha,r, p,q,\varepsilon}^s(f) <\infty\right\},
 	\end{align*}
 	where \begin{equation}\label{1-8aq}
 B_{\alpha,r, p,q,\varepsilon}^s(f):= \left\| \left\{
 \fint_{2^{-j-1}}^{2^{-j}} \left| \left\{ x\in\mathbb{R}^n:\
 \left| \partial_{n+1}^r u_\alpha (x, y^\alpha)
 \right|>\varepsilon y^{s-\alpha r} \right\} \right|dy
 \right\}_{j=0}^\infty \right\|_{\ell^{\frac
 qp}(\mathbb{Z}_+)}.
\end{equation}
 	\item For the Triebel--Lizorkin spaces $F_{p,q}^s$ with
 $p\in(0,\infty)$ and $q\in(0,\infty]$ and
for any $f\in \mathcal{C}_0 \cap \Lambda_s$,
\begin{align*}
 &	 d (f, F_{p, q}^s)_{\Lambda_s} \sim \inf \left\{
 \varepsilon\in(0,\infty):\ F_{\alpha,r, p,
 q,\varepsilon}^s(f) <\infty\right\},
 	\end{align*}
 	where
 	$$ F_{\alpha,r, p, q,\varepsilon}^s(f):= \left\|
 \int_{\left\{y\in (0,1]:\
 | \partial_{n+1}^r u_\alpha (x, y^\alpha) | >\varepsilon
 y^{s-\alpha r}\right\}} \frac {dy} y
 \right\|_{L^{\frac pq}(\mathbb{R}^n, dx)}.$$
 	
 	\end{itemize}

For $\alpha=1$, the operators $T_{\alpha, t}$ reduce to the
Poisson semigroup. In this case, the estimate \eqref{1-8b} for
$0<s\leq 1$ and $r=2$ was previously established in
\cite[Theorem 4]{ss}.
By applying the fractional heat equation \eqref{1-6-0} twice,
one can show that
the function $u(x, t)=T_{1,t} f(x)$ satisfies the Laplace
equation: $$\partial_{t}^2 u(x, t) +\sum_{j=1}^n\partial_j^2
u(x, t)=0,\ \ (x, t)\in\mathbb{R}_+^{n+1}.$$
 This equation plays a crucial role in the proof for
 $\alpha=1$ in \cite{ss}, where harmonic properties of the
 Poisson semigroup are utilized.

For general fractional heat semigroups, however, the situation
is more subtle, and the proofs are technically more involved.
Indeed, when $\alpha\neq 1, 2$, the fractional heat equation
\eqref{1-6-0} no longer corresponds to a local differential
operator and, in particular, the Laplace equation
is unavailable. To address this issue, we develop a new
approach based on higher-order ball average operators and
certain combinatorial identities. This method not only
simplifies the analysis but also enables us to establish the
desired result for the full range
 $s\in(0,\infty)$ and $\alpha\in(0,\infty)$.

As a direct application of our main results, we can describe
  the closures $\overline{V \cap \Lambda_s}^{\Lambda_s}$ of
  various classical function spaces $V$ with respect to the
  topology of $\Lambda_s$.
For instance, in the case of the Besov space $B^s_{p,q}$ with
$s\in(0,\infty)$, $p\in(0,\infty]$ and $q\in(0,\infty]$,
our characterizations can be stated as follows. Let
$r\in\mathbb{N}$ be such that $\alpha r>s+3$, and let
$u_\alpha(x, y) :=T_{\alpha, y} f(x)$ for any $f\in \Lambda_s$
and $(x, y)\in\mathbb{R}_+^{n+1}$.
 	\begin{itemize}
 		\item If $p\in(0,\infty)$, then
 		$f\in\overline{B^s_{p,q}\cap\Lambda_s
 		}^{\Lambda_s}$
 		if and only if
 		 $B_{\alpha,r, p,q,\varepsilon}^s(f)<\infty$ for any
 $\varepsilon\in(0,\infty)$ and
 		$$	
\lim_{k\in\mathbb{Z}^n,\;|k|\rightarrow\infty}
\left|	\int_{\mathbb{R}^n}\varphi(x-k)f(x)\,dx\right|=0,$$
 		where $B_{\alpha,r, p,q,\varepsilon}^s(f)$ is
 defined in \eqref{1-8aq} and $\varphi$ is the father
 wavelet in a Daubechies wavelet system of regularity
 		 $L>\max\{s,n(\max\{\frac 1p,1\}-1)-s\}$.
 	 		
 		\item If $p=\infty$, then
 		$f\in\overline{B^s_{\infty,q}
 		}^{\Lambda_s}$
 		if and only if
 		$$
 		\sharp\{j\in\mathbb Z_+:\ D_{\alpha,r,j}(s, f,
 \varepsilon)\neq\emptyset\}<\infty,\ \forall\,
 \varepsilon\in(0,\infty),
 		$$
where, for any $j\in\mathbb{Z},$
\begin{align}\label{2-6-eq} D_{\alpha,r,j} ( s, f,
 \varepsilon) :=\left\{ (x,t)\in
\mathbb{R}^n\times (2^{-j-1}, 2^{-j}]:\
 |\partial_{n+1}^r u_\alpha(x, t^\alpha)|>\varepsilon
 t^{s-\alpha r}\right\}.
\end{align}
\end{itemize}

This article establishes a more general result within a
unified framework introduced in \cite{dsy}. In this setting,
a function space $V$ of smooth
order $s\in(0,\infty)$, referred to as the
\emph{Daubechies $s$-Lipschitz $X$-based space}, is defined
via Daubechies wavelet
expansions and a given quasi-normed
lattice $X$ of function sequences (see Definitions
\ref{Debqfs} and \ref{asdfs}).
 This framework includes all the classical function spaces
 mentioned earlier, making the previously stated results
 become
 direct corollaries of the general theorem. The details of
 this general framework will be presented in the next section.

Finally, we point out that various characterizations of
function spaces will be used throughout this article. For
wavelet characterizations of Besov and Triebel--Lizorkin
spaces, we refer to \cite{Me92,t06,t08,t10}; for their
difference characterizations, see \cite{T83,t92,T20}. Wavelet
characterizations of Besov-type and Triebel--Lizorkin-type
spaces can be found in
\cite{ysy10,lsuyy12,s011,ht23,syy,FSyy}, while their
difference characterizations are treated in
\cite{s08,ysy10,s011,hs20,h22}. For further extensions of
these spaces, see also
\cite{b5,b6,b7,b8,b9,HS13,HMS16,hl,hms23,hlms}.
In addition to the classical function spaces mentioned above,
our main result applies to a broader class of spaces that
admit wavelet-type characterizations.

\section{Framework and general results}\label{sec:3}

Before proceeding, let us briefly review
the Daubechies wavelet system on
$\mathbb{R}^n$ and related facts.
We denote by $\mathcal{C}_{\mathrm{c}}^r$
the set of all $r$ times continuously
differentiable real-valued functions on
$\mathbb{R}^n$ with compact support. For any
$\gamma:=(\gamma_1,\ldots,\gamma_n)\in\mathbb Z_+^n$ and
$x:=(x_1,\ldots,x_n)\in\mathbb{R}^n$, let
$x^\gamma:=x_1^{\gamma_1}\cdots x_n^{\gamma_n}$.
According to \cite[Sections~3.8 and 3.9]{Me92},
for any integer $L\in(1,\infty)$
there exist functions
$\varphi,\psi_l\in C_{\mathrm{c}}^L$ with
$l\in\{1,\ldots, 2^{n}-1\}=:\mathcal{N}$
such that,
for any $\gamma\in\mathbb{Z}_+^n$ with
$|\gamma|\le L$,
\begin{align*}
\int_{\mathbb{R}}\psi_{l}(x)x^\gamma\,dx=0,\
\forall\, l\in \mathcal{N}
\end{align*}
and the union of both
\begin{align}\label{fanofn}
\Phi:=\left\{\varphi_k(x):=\varphi(x-k):\
k\in \mathbb Z^n\right\}
\end{align}
and $$
\Psi:=\left\{ \psi_{l,j,k}(x):=2^{\frac{jn}2}
\psi_l (2^jx-k):l\in \mathcal{N},\
j\in \mathbb Z_+,\ k\in\mathbb Z^n\right\}
$$
forms an orthonormal basis of $L^2$.
The system $\Phi\cup \Psi$
is called the \emph{Daubechies wavelet system
of regularity $L$}
on $\mathbb{R}^n$.
We index this system as usual via dyadic
cubes as follows.
For any $j\in \mathbb{Z}_+$, let $\mathcal{D}_j$
denote the class of all
dyadic cubes in $\mathcal{D}$ with edge length
$2^{-j}$.
Let $\Omega:=\Omega_0\cup\Omega_1$, where
$\Omega_0:=\{0\}\times \mathcal{D}_0$ and
$\Omega_1:= \mathcal{N}\times \mathcal D$.
For any $\omega=(0, I)\in\Omega_0$ with
$I=k+[0,1)^n\in \mathcal{D}_0$, let
\begin{align*}
\psi_\omega(\cdot):=\varphi(\cdot-k),
\end{align*}
while, for any $\omega=(l_\omega,
I_\omega)\in\Omega_1$
with $l_\omega\in\mathcal{N}$
and $I_\omega=2^{-j} ( k+[0,1)^n) \in\mathcal
D_j$, let
\begin{align*}
\psi_{\omega}(\cdot): =2^{\frac{jn}2}
\psi_{l_\omega}(2^j \cdot-k).
\end{align*}
For any $\omega\in\Omega$ and
$f\in L^1_{\mathrm{loc}}$ (the set of all locally
integrable functions on $\mathbb{R}^n$), let
$$\langle f,\psi_{\omega}\rangle:=
\int_{\mathbb{R}^n} f(x) \overline{\psi_\omega(x)}\, dx.$$
Throughout this article, we will fix a Daubechies wavelet
system
$\Phi\cup \Psi$ with large regularity $L$.

The definition of the function space $V$ in
our framework also requires the concept of
quasi-normed lattice of function sequences,
which we now introduce.
For any given $i,j\in\mathbb{Z}_+$,
let $\delta_{i,j}:=1$ or $0$ depending on
whether $i=j$ or $i\neq j$.
We denote by $\mathscr M$ the
set of all measurable functions $f:
\mathbb{R}^n\to \mathbb{R}\cup\{\pm\infty\}$
that are finite almost everywhere on
$\mathbb{R}^n$ and by $\mathscr{M}_{\mathbb{Z}_+}$
the set of all sequences $	\{f_j\}_{j\in\mathbb{Z}_+}$
of functions in $\mathscr M$. We use
boldface letters $\mathbf G, \mathbf F,\ldots$
to denote elements in $\mathscr{M}_{\mathbb{Z}_+}$.
For any
 $\gamma\in \mathbb{C}$
and $\mathbf F:=\{f_j\}_{j\in\mathbb{Z}_+},\mathbf
G:=\{g_j\}_{j\in\mathbb{Z}_+}
\in \mathscr{M}_{\mathbb{Z}_+}$, define
$|\mathbf{F}|:=\{| f_j|\}_{j\in\mathbb{Z}_+},\
\gamma \mathbf F:=\{\gamma f_j\}_{j\in\mathbb{Z}_+}$,
and
$\mathbf F\pm \mathbf G:=
\{f_j\pm g_j
\}_{j\in\mathbb{Z}_+}.$
Also, we write $\mathbf{F}\leq \mathbf {G}$
if, for any $j\in\mathbb{Z}_+$,
$f_j\leq g_j$ almost everywhere on $\mathbb{R}^n$.
We define the \emph{left shift $\textnormal{Sh}_L$} and
the \emph{right shift} $\textnormal{Sh}_R$ on the space
$\mathscr{M}_{\mathbb{Z}_+}$
as follows: for any $\mathbf{F}:
=\{f_j\}_{j\in\mathbb{Z}_+}
\in\mathscr{M}_{\mathbb{Z}_+}$,
\begin{align*}
\textnormal{Sh}_L\left(\mathbf{F}\right)
:=\{f_{j+1}\}_{j\in\mathbb{Z}_+}
\ \mathrm{and}\
\textnormal{Sh}_R\left(\mathbf{F}\right)
:=\{f_{j-1}\}_{j\in\mathbb{Z}_+}, \
\text{where}\ f_{-1}:=0.
\end{align*}

\begin{definition}\label{Debqfs}
Let $\|\cdot\|_X:
\mathscr{M}_{\mathbb{Z}_+}\to [0,\infty]$ be
an extended-valued quasi-norm
defined on the entire space
$\mathscr{M}_{\mathbb{Z}_+}$ satisfying
that $\|\mathbf F\|_X\leq \|\mathbf{G}\|_X$
whenever $\mathbf{F},\mathbf{G}\in
\mathscr{M}_{\mathbb{Z}_+}$ and
$|\mathbf{F}|\leq |\mathbf{G}|$.
Let $X$ denote the space of all function
sequences $ \mathbf{F}\in
\mathscr{M}_{\mathbb{Z}_+}$ such that
$\|\mathbf{F}\|_X<\infty$.
We call
$X=(X,\|\cdot\|_X)$ a \emph{quasi-normed
lattice
of function
sequences} if it satisfies the following
two conditions:
\begin{itemize}
\item[\rm (i)]
there exists a positive constant $C$ such
that,
for any $\mathbf{F}\in X$,
\begin{align*}
\left\|\textnormal{Sh}_L\left(\mathbf{F}\right)
\right\|_X+ \left\|	\textnormal{Sh}_R\left(\mathbf{F}\right)
\right\|_X\leq C \left\|\mathbf{F}\right\|_X.
\end{align*}
\item[\rm (ii)]
for any bounded function $f:\mathbb{R}^n\to
\mathbb{R}$ with compact support,
$\{f, 0,0,\ldots\}\in X$.
\end{itemize}
\end{definition}

The function space with smooth order $s\in(0,\infty)$ in
our general framework is called the
Daubechies $s$-Lipschitz $X$-based space.
It is defined below through the Daubechies
wavelet expansion and a given quasi-normed
lattice $X$ of function sequences.

\begin{definition}\label{asdfs}
Let $X$ be a quasi-normed lattice
of function sequences. Let $s\in(0,\infty)$. Assume that
 the regularity $L$ of the \emph{Daubechies wavelet
 system} is strictly bigger than $s$. Then the
\emph{Daubechies $s$-Lipschitz $X$-based space}
$\Lambda_X^{s}$ is defined
to be the set of
all functions $f\in \Lambda_{s}$ such that
$$
\|f\|_{\Lambda_X^{s}}:=\left\|\left\{
\sum_{\omega\in \Omega,I_\omega\in\mathcal{D}_j}
|I_\omega|^{-\frac sn-\frac 12}\left|\langle
f,\psi_{\omega}\rangle\right|\mathbf{1}_{I_\omega}
\right\}_{j\in\mathbb{Z}_+}
\right\|_{X}<\infty.
$$
\end{definition}

\begin{remark}\label{rem-3-4} In the above definition, the
space
$\Lambda_X^{s}$ formally depends on a particular choice of the
wavelet
 systems. However, most function spaces in analysis that admit
 wavelet characterizations can also be equivalently defined
 through alternative tools such as finite differences and
 semigroup methods. In such cases, any
 Daubechies wavelet
 system with sufficient regularity may be used in their
 wavelet characterization. Therefore, for the validity of our
 main results for these function spaces, we may always fix a
 wavelet system with sufficiently large regularity $L$.
\end{remark}

Many classical function spaces, including those
mentioned function spaces above, can be viewed as the
\emph{Daubechies s-Lipschitz $X$-based spaces},
associated with a suitable quasi-normed lattice
$X$ of function sequences, according to
the usual wavelet characterizations of these spaces.
For example, the wavelet characterizations of Besov spaces
$B^{s}_{p,q}$ and the Triebel--Lizorkin spaces $F^{s}_{p,q}$
for any $p,q, s\in (0,\infty)$ (see \cite[Proposition 1.11 and
Corollary 2]{T20}) yield the following equivalences (assuming
the wavelet system is sufficiently regular):
\begin{align*}
\|f\|_{B^{s}_{p,q}}&\sim \left\{\sum_{j=0}^{\infty}
\left\|\sum_{\omega\in \Omega,I_\omega\in\mathcal{D}_j}
|I_\omega|^{-\frac sn-\frac 12}\left|\langle
f,\psi_{\omega}\rangle\right|\mathbf{1}_{I_\omega}\right\|_{L^p}
^{q}\right\}^{\frac 1q}
\end{align*}
and
\begin{align*}
\|f\|_{F^{s}_{p,q}}&\sim \left\|\left\{\sum_{j=0}^{\infty}
\left[\sum_{\omega\in \Omega,I_\omega\in\mathcal{D}_j}
|I_\omega|^{-\frac sn-\frac 12}\left|\langle
f,\psi_{\omega}\rangle\right|\mathbf{1}_{I_\omega}\right]^q
\right\}^{\frac 1q}\right\|_{L^p}.
\end{align*}
This means that both $B^{s}_{p,q}$ and $F^{s}_{p,q}$ can be
regarded as Daubechies s-Lipschitz $X$-based spaces, with the
quasi-norm on the underlying function sequence lattice $X$
given, respectively, by
$$
\left\|\left\{f_j\right\}_{j\in\mathbb{Z}_+}\right\|_X=\left(
\sum_{j=0}^\infty \|f_j\|_{L^p}^q\right)^{\frac 1q}$$
and
$$
\left\|\left\{f_j\right\}_{j\in\mathbb{Z}_+}\right\|_X=
\left\|\left(\sum_{j=0}^\infty |f_j|^q\right)^{\frac
 1q}\right\|_{L^p}.$$

We are now in a position to formulate the main results in the
general setting. For simplicity, we adopt the following
symbols for the remainder of this section:
\begin{itemize}
\item $X$ is a quasi-normed lattice
of function sequences, and $\Lambda_X^{s}$ is the
Daubechies $s$-Lipschitz $X$-based space for a fixed
$s\in(0,\infty)$;
\item $L$ denotes the regularity of the Daubechies wavelet
system $\{\psi_\omega\}_{\omega\in\Omega}$;
\item $T_{\alpha, t}$ is the fractional heat semigroup of
order $\alpha\in(0,\infty)$ defined in \eqref{1-6-1}, $r$ is
an integer such that $\alpha r>s+3$, and $u_\alpha(x, t)
:=T_{\alpha, t} f(x)$ for any $f\in\Lambda_s$ and
$(x,t)\in\mathbb{R}^{n+1}_+$.
\end{itemize}

For any $f\in\Lambda_s$ and $\varepsilon\in(0,\infty)$, recall
that the sets $D_{\alpha, r}(s, f, \varepsilon)$ and
$D_{\alpha, r, j}(s, f, \varepsilon)$, $j\in\mathbb{Z}_+$, are
defined in \eqref{2-2:eq} and \eqref{2-6-eq}, respectively. We
also define
\begin{align}
V_0(s, f, \varepsilon):&=
\left\{ I\in\mathcal D_0:\
\left|\langle f, \psi_{(0, I)}\rangle
\right|
>\varepsilon \right\}.\label{3-5-eq}
\end{align}
These assumptions and symbols are understood and will not be
repeated in the following statements.

For simplicity, we refer to $\alpha, s, r, L, n$ and the
positive constant $C$ in Definition \ref{Debqfs} as framework
parameters. The framework parameters also include the
parameter $\theta$ and the implicit positive constant $C$ in
Assumption I below if assumed as well as the general positive
constant $C$ in Definition \ref{Debqf2s} below if Assumption
II below is assumed. Most implicit constants in this article
depend on these framework parameters.

The formulation of our first main result requires the
following
assumption for the given quasi-normed
lattice $X$ of function sequences.

\begin{assumption}\label{a1} (\emph{Doubling condition of
$X$})
There exist positive constants $\theta$ and $C$
such that,
for any sequence $\{B_{k,j}\}_{k\in
\mathbb N,j\in\mathbb Z_+}\subset \mathbb{B}$
of Euclidean balls in $\mathbb{R}^n$, one has
\begin{align*}
\left\|\left\{ \left[\sum_{k\in\mathbb N}
\mathbf{1}_{2 B_{k,j}}\right]^\theta
\right\}_{j\in\mathbb Z_+}\right
\|_{X}
\le C\left\|\left\{ \left[
\sum_{k\in\mathbb N}\mathbf{1}_{B_{k,j}}\right]^\theta
\right\}_{j\in\mathbb Z_+}\right\|_{X}.
\end{align*}
\end{assumption}

\begin{remark}
From Assumption \ref{a1}
(Doubling condition of $X$), it follows that, for
any $\beta\in(1,\infty)$
and any sequence $\{B_{k,j}\}_{k\in
\mathbb N,j\in\mathbb Z_+}$ in $\mathbb{B}$
of Euclidean balls in $\mathbb{R}^n$, one has
\begin{align*}
\left\|\left\{ \left[\sum_{k\in\mathbb N}
\mathbf{1}_{\beta B_{k,j}}\right]^\theta
\right\}_{j\in\mathbb Z_+}\right
\|_{X}
\le C
^{\lceil\log_2\beta\rceil}
\left\|\left\{ \left[
\sum_{k\in\mathbb N}\mathbf{1}_{B_{k,j}}\right]^{\theta}
\right\}_{j\in\mathbb Z_+}\right\|_{X}.
\end{align*}
\end{remark}

Let us also define the following two quantities for any
$f\in\Lambda_s$ under Assumption I:
\begin{align}
\varepsilon_{X}^0f:&=
\inf\left\{\varepsilon\in(0,\infty):\ \left\|\left\{
\sum_{I\in V_0(s, f, \varepsilon)}
\mathbf{1}_{I}\delta_{0,j}
\right\}_{j\in\mathbb{Z}_+}
\right\|_{X}<\infty\right\}\label{4-6-1}
\end{align}
and
\begin{align}
\varepsilon_{X}f:&=
\inf\left\{\varepsilon\in(0,\infty):\
\left\|\left\{ \left[\int_{0}^1
\mathbf{1}_{D_{\alpha,r,j}(s,f, \varepsilon)}
(\cdot,t)\,\frac{dt}{t}
\right]^\theta\right\}_{j\in\mathbb Z_+}
\right\|_{X}<\infty\right\}, \label{ef-1}
\end{align}
where
the set $V_0(s, f,\varepsilon)$ is defined in \eqref{3-5-eq}.

\begin{theorem}\label{thm-2-66}
Assume that $L>\alpha r +1$ and $X$ satisfies Assumption
\ref{a1} for some constant $\theta\in(0,\infty)$.
Then, for any $f\in\Lambda_s$,
\begin{align*}
d \left(f, \Lambda_X^{s}\right)_{\Lambda_s}
\sim \varepsilon_{X}^0f+\varepsilon_{X}f,
\end{align*}
where the positive equivalence constants depend only on the
framework parameters.
\end{theorem}

\begin{remark} According to \cite[Corollary 4.1]{dsy}, if the
quasi-norm $\|\cdot\|_X$ satisfies
 the additional condition
 \begin{align}\left\|\{
\mathbf{1}_{E}\delta_{0,j}
\}_{j\in\mathbb{Z}_+}\right\|_X=\infty\ \ \text{ whenever
$E\subset \mathbb{R}^n$ and
$|E|=\infty$},\label{3-7}\end{align}
 then the quantity $\varepsilon_X^0 f$ defined in
 \eqref{4-6-1} admits the following equivalent expression in
 terms of the father wavelet $\varphi$ of the Daubechies
 wavelet system given in \eqref{fanofn}:
$$ \varepsilon_X^0 f \sim
\limsup_{k\in\mathbb{Z}^n,\;|k|\rightarrow\infty}
\left|\int_{\mathbb{R}^n}\varphi(x-k)f(x)\,dx\right|.$$
As a result, under this additional assumption \eqref{3-7} on
$\|\cdot\|_X$ and under the assumptions of Theorem
\ref{thm-2-66}, we have
\begin{align*}
 d \left(f,
 \Lambda_X^{s}\right)_{\Lambda_s}&\sim\varepsilon_Xf+
\limsup_{k\in\mathbb{Z}^n,\;|k|\rightarrow\infty}
\left|\int_{\mathbb{R}^n}\varphi(x-k)f(x)\,dx\right|.
\end{align*}
\end{remark}

As a corollary of Theorem~\ref{thm-2-66},
 we immediately obtain the following characterization of the
 closure $\overline{\Lambda_X^{s}}^{\Lambda_s}$ of
 $\Lambda_X^{s}$ under Assumption \ref{a1}.

\begin{corollary}\label{ppqq}
Assume that $L>\alpha r+1$ and $X$ satisfies Assumption
\ref{a1} for some constant $\theta\in(0,\infty)$.
 Then a function
$f\in\Lambda_s$ belongs to
$\overline{\Lambda_X^{s}}^{\Lambda_s}$ if and only if,
for any $\varepsilon\in(0,\infty)$,
$$\left\|\left\{
\sum_{I\in V_0(s, f, \varepsilon)}
\mathbf{1}_{I}\delta_{0,j}
\right\}_{j\in\mathbb{Z}_+}
\right\|_{X}+\left\|\left\{\left[\int_{0}^1
\mathbf{1}_{D_{\alpha,r,j}(s,f, \varepsilon)}
(\cdot,t)\,\frac{dt}{t}
\right]^{\theta}\right\}_{j\in\mathbb Z_+}
\right\|_{X}
<\infty.
$$
\end{corollary}

One limitation of Theorem~\ref{thm-2-66} is that Assumption
\ref{a1} does not hold for certain function spaces associated
with endpoint parameters, such as
the Triebel--Lizorkin space
$F^s_{\infty,q}$ and the Besov space
$B^s_{\infty,q}$. See Section \ref{sea8} for a detailed
discussion. As a result, Theorem~\ref{thm-2-66} is not
applicable to these spaces.
However, such endpoint spaces typically satisfy a different weaker
condition, namely, Assumption \ref{pplp} below. Under that
assumption, we will establish an alternative result,
Theorem~\ref{thm-7-11}, which extends our framework to include
function spaces with endpoint parameters.

The formulation of Assumption \ref{pplp} is more involved,
requiring introducing some additional symbols and definitions.
We denote by $\mathscr P({\mathbb R}^{n+1}_+)$
the
collection of all measurable subsets of
${\mathbb R}^{n+1}_+$.
Let
\begin{align*}
\mathscr{P}_{\mathbb{Z}_+}(\mathbb{R}^{n+1}_+)
:=
\left\{
\{A_j\}_{j\in\mathbb{Z}_+}:\forall\, j\in\mathbb{Z}_+,\
A_j\in\mathscr P({\mathbb R}^{n+1}_+)\right\}
\end{align*}
be the collection of all sequences of measurable subsets of
${\mathbb R}^{n+1}_+$.
For any $\{A_j\}_{j\in\mathbb{Z}_+},
\{B_j\}_{j\in\mathbb{Z}_+}\in
\mathscr{P}_{\mathbb{Z}_+}(\mathbb{R}^{n+1}_+)$,
we write
$\{A_j\}_{j\in\mathbb{Z}_+}\subset
\{B_j\}_{j\in\mathbb{Z}_+}$ if $A_j\subset B_j$
for any $j\in\mathbb{Z}_+$.

\begin{definition}
The \emph{left shift $\textnormal{Sh}_L$} and
the \emph{right shift} $\textnormal{Sh}_R$ on
$\mathscr{P}_{\mathbb{Z}_+}
(\mathbb{R}^{n+1}_+)$
are defined, respectively, as follows: for any
$\mathbf{F}:=\{A_j\}_{j\in\mathbb{Z}_+}
\in\mathscr{P}_{\mathbb{Z}_+}
(\mathbb{R}^{n+1}_+)$,
$$\textnormal{Sh}_L(\mathbf{F})
:=\{A_{j+1}\}_{j\in\mathbb{Z}_+}\ \text{ and }\ \
\textnormal{Sh}_R(\mathbf{F})
:=\{A_{j-1}\}_{j\in\mathbb{Z}_+},\ \text{
where $A_{-1}:=\emptyset$.}$$
\end{definition}

We also use the \emph{Poincar\'e
hyperbolic metric} $\rho:\mathbb{R}_+^{n+1}\times
\mathbb{R}_+^{n+1}
\to \mathbb{R}$ in the upper-half space
$\mathbb{R}_+^{n+1}$, which is defined
by setting, for any $\mathbf{x}:=( x, x_{n+1}),
\mathbf{y}:=( y, y_{n+1}) \in \mathbb{R}_+^{n+1}$,
\begin{align*}
\rho(\mathbf{x}, \mathbf{y}) &:
=\mathop\mathrm{\,arccosh\,}
\left( 1+\frac{|\mathbf{x}-\mathbf{y}|^2}{2x_{n+1}
y_{n+1}}\right).
\end{align*}
Here, we write a vector in $\mathbb{R}_+^{n+1}$ in the form
$\mathbf{x}:=( x, x_{n+1})$ or $\mathbf{y}:=( y, y_{n+1})$,
where $x, y\in \mathbb{R}^n$ and $x_{n+1},
y_{n+1}\in(0,\infty)$.
Properties of the metric $\rho$ can be found
in Appendix of \cite{dsy}.

For any $\mathbf x\in\mathbb{R}^{n+1}_+$ and $t\in(0,\infty)$,
 define
$$ B_\rho( \mathbf x, t):=\left\{ \mathbf
y\in\mathbb{R}^{n+1}_+:\ \rho\left(\textbf{x},
\textbf{y}\right)< t\right\}.$$
Given $R\in(0,\infty)$, we define
the \emph{hyperbolic $R$-neighborhood} $A_R$
of a subset $A\subset \mathbb{R}_+^{n+1}$
by setting
$$A_R:=\{\mathbf x\in\mathbb{R}^{n+1}_+:
\rho(\mathbf x, A)< R\},$$
where
$\rho\left(\textbf{x}, A\right):=
\inf_{\textbf{y}\in A} \rho\left(\textbf{x},
\textbf{y}\right)$.

\begin{definition}
(\emph{Property I}) We say that a set $A\subset
\mathbb{R}_+^{n+1}$ has Property I
 with constants $\delta, \delta'\in(0,\infty)$
 if
\begin{equation}\label{5-1-00}
|B_\rho(\mathbf{z},\delta) \cap A |\ge
\delta' |B_\rho(\mathbf{z},\delta)|,\ \forall\, \mathbf{z}\in
A.
\end{equation}
In general, we say a set $A\subset \mathbb{R}_+^{n+1}$ has
Property I
 if \eqref{5-1-00} holds for some constants $\delta,
 \delta'\in (0, \frac 1{10})$.
\end{definition}

Observe that Property I describes some geometric
properties of the boundary
of sets.

\begin{definition}\label{Debqf2s}
An extended non-negative valued set function
$\nu:\mathscr{P}_{\mathbb{Z}_+}
(\mathbb{R}^{n+1}_+)\rightarrow [0,\infty]$
is called a \emph{Carleson-type measure}
if it satisfies the following conditions for some constant
$C\in(0,\infty)$ and all $\{A_j\}_{j\in\mathbb{Z}_+},
\{B_j\}_{j\in\mathbb{Z}_+}\in
\mathscr{P}_{\mathbb{Z}_+}
(\mathbb{R}^{n+1}_+)$:
\begin{itemize}
\item[\rm(i)]
 $\nu(\{A_j\}_{j\in\mathbb{Z}_+})
\le\nu(\{B_j\}_{j\in\mathbb{Z}_+})$ whenever
$\{A_j\}_{j\in\mathbb{Z}_+}\subset
\{B_j\}_{j\in\mathbb{Z}_+}$;
\item[\rm(ii)]
$\nu(\{A_j\cup B_j\}_{j\in\mathbb{Z}_+})
\le C\left[\nu(\{A_j\}_{j\in\mathbb{Z}_+})
+\nu(\{B_j\}_{j\in\mathbb{Z}_+})\right]$;
\item[\rm(iii)]
$\nu(\textnormal{Sh}_L\{A_j\}_{j\in\mathbb{Z}_+})
+\nu(\textnormal{Sh}_R\{A_j\}_{j\in\mathbb{Z}_+})
\le C \nu(\{A_j\}_{j\in\mathbb{Z}_+})$;
\item[\rm(iv)] if every set in the sequence
$\{A_j\}_{j\in\mathbb{Z}_+}$ has Property I with uniform
constants $\delta,
\delta'\in (0, \frac{1}{10})$ (independent of $j$),
then, for any $R\in(1,\infty)$,
$$\nu(\{(A_j)_R\}_{j\in\mathbb Z_+})<\infty\iff
\nu(\{A_j\}_{j\in\mathbb Z_+})<\infty.$$
\end{itemize}
\end{definition}
Given a function
$\nu:\mathscr{P}_{\mathbb{Z}_+}
(\mathbb{R}^{n+1}_+)\rightarrow
[0,\infty]$, we also use a slight abuse of symbol that,
for any $A\in \mathscr
P({\mathbb R}^{n+1}_+)$, $\nu(A)
:=\nu\left(\left\{A, \emptyset, \emptyset,\ldots\right\}
\right).$
For each dyadic cube $I\in \mathcal D$, we define
$T(I): =I \times (
\frac{\ell(I)}{2}, \ell(I)]$.

In our second result, we assume the quasi-normed lattice $X$
satisfies
the following assumption.

\begin{assumption}\label{pplp}
(\emph{Carleson-type measure condition} of $X$) There exists a
Carleson-type measure
$\nu:\mathscr{P}_{\mathbb{Z}_+}
(\mathbb{R}^{n+1}_+)\rightarrow [0,\infty]$ such that,
for any collection $\mathcal{A}$ of dyadic cubes in $\mathcal
D$,
\begin{align*}
\left\|\left\{
\sum_{I\in \mathcal{A}\cap \mathcal D_j}
\mathbf{1}_{I}
\right\}_{j\in\mathbb{Z}_+}
\right\|_{X}<\infty \iff \nu\left(\left\{
\bigcup_{I\in \mathcal{A}\cap \mathcal D_j}
T(I)\right\}_{j\in\mathbb{Z}_+}\right)<\infty,
\end{align*}
where it is agreed that $\sum_{I\in\emptyset} \mathbf{1}_I=0$.
\end{assumption}

\begin{remark}
\begin{itemize}
\item[(i)] It was pointed out in \cite[Remark 2.8]{dsy}
that, if $X$ satisfies Assumption \ref{a1},
then $X$ also satisfies Assumption \ref{pplp} if
we choose, for any $\{A_j\}_{j\in\mathbb{Z}_+}\in
\mathscr{P}_{\mathbb{Z}_+}
(\mathbb{R}^{n+1}_+)$,
$$\nu(\{A_j\}_{j\in\mathbb{Z}_+}):=
\left\|\left\{\int_{0}^{\infty}
\mathbf{1}_{A_j}(\cdot,y)\,\frac{dy}{y}
\right\}_{j\in\mathbb{Z}_+}\right\|_{X}.
$$

\item[(ii)] It was shown in \cite[Lemmas 5.7 and 5.15]{dsy}
that Assumption \ref{pplp} holds for
$\Lambda_X^{s}:=F^s_{\infty,q}$
with $q\in(0,\infty]$ and for $\Lambda_X^{s}:=B^s_{\infty,q}$
with $q\in(0,\infty]$, where $s\in(0,\infty)$.
\end{itemize}
\end{remark}

\begin{theorem}\label{thm-7-11}
Assume that $L>\alpha r+1$ and $X$ satisfies Assumption
\ref{pplp} for some Carleson-type
measure $\nu:
\mathscr{P}_{\mathbb{Z}_+}
(\mathbb{R}^{n+1}_+)\rightarrow [0,\infty]$.
Then,
for any $f\in\Lambda_s$,
\begin{align*}
d\left(f, \Lambda_X^{s}
\right)_{\Lambda_s}&\sim
\varepsilon_{X,\alpha,\nu}f+\varepsilon_{X,\nu}^0,
\end{align*}
where
\begin{align*}
\varepsilon_{X,\alpha, \nu}f&
:=\inf\left\{\varepsilon\in
(0,\infty):\
\nu\left(\left\{D_{\alpha,r,j}(s,f,
\varepsilon)\right
\}_{j\in\mathbb Z_+}\right)<\infty\right\},
\end{align*}
\begin{align}
\varepsilon_{X,\nu}^0:&=\inf\left\{\varepsilon\in(0,\infty):\
\nu\left(\bigcup_{I\in V_0(s, f, \varepsilon)}
T(I)\right)<\infty\right\},\label{3-11c}
\end{align}
 and
the positive equivalence constants are
independent of $f$.
\end{theorem}

As a direct corollary of Theorem \ref{thm-7-11}, we
immediately obtain the following characterization of the
closure $\overline{\Lambda_X^{s}}^{\Lambda_s}$ of
$\Lambda_X^{s}$ under Assumption
\ref{pplp}.

\begin{corollary}\label{pppz2qq}
Assume that $L>\alpha r+1$ and $X$ satisfies Assumption
\ref{pplp} for some Carleson-type
measure $\nu:
\mathscr{P}_{\mathbb{Z}_+}
(\mathbb{R}^{n+1}_+)\rightarrow [0,\infty]$.
 Then a function
$f\in\Lambda_s$ belongs to
$\overline{\Lambda_X^{s}}^{\Lambda_s}$ if and only if,
for any $\varepsilon\in(0,\infty)$,
$$
\nu\left(\left\{ D_{\alpha,r,j}(s,f, \varepsilon)
\right\}_{j\in\mathbb Z_+}
\right) +\nu\left(
\bigcup_{I\in V_0(s, f, \varepsilon)}
T(I)
\right)
<\infty.
$$
\end{corollary}

The proofs of Theorems \ref{thm-2-66}
and \ref{thm-7-11} are quite long, so we divide them
into several sections (Sections~\ref{sec:4}--\ref{sec:10}) and
steps. A key technical tool is the family of higher-order ball
average operators.
Specifically, Section~\ref{sec:4} collects several useful
results from our earlier article \cite{dsy}, which will be
used repeatedly in the subsequent analysis.
In Section \ref{sec-5}, we establish pointwise kernel
estimates for the fractional heat semigroup $T_{\alpha, t}$,
with $\alpha, t\in (0,\infty)$, and use them to derive uniform
norm estimates for the derivatives $\partial_t^i
\partial_x^\beta [T_{\alpha, t} f(x)]$,
 where $f\in\Lambda_s$, $\beta\in\mathbb{Z}_+^n$, and
 $i\in\mathbb{Z}_+$.
In particular, we characterize the space $\Lambda_s$
 in terms of the uniform norms of these derivatives. These
 results play a crucial role in the proofs of the main
 theorems in later sections.
 Section~\ref{oijjl} outlines the main steps of the proof of
 Theorem~\ref{thm-2-66}, our first main result. This proof
 relies on three major technical propositions, which are also
 formulated in Section~\ref{oijjl}, but are proved in
 subsequent sections. A complete proof of
 Theorem~\ref{thm-2-66}, conditional on these propositions, is
 also given here.
 The proofs of the three key propositions are carried out in
 Sections~\ref{sec:7}--\ref{sec:9}. These arguments are quite
  technical, but constitute the core of the proof of
  Theorem~\ref{thm-2-66}.
Section~\ref{sec:10} presents the proof of
Theorem~\ref{thm-7-11}, which is deduced from the
aforementioned three propositions established in
Sections~\ref{sec:7}--\ref{sec:9}.

\section{Preliminary results}\label{sec:4}

In proving the main results, we will use several useful
estimates
established in \cite{dsy}, which we present in this section.

 We first recall some facts
 related to the \emph{Poincar\'e
hyperbolic metric} in the upper-half space
$\mathbb{R}_+^{n+1}$.

\begin{lemma}\textnormal{\cite[(A.4) and (A.5)]{dsy}}
\label{lem-4-2zz}
If
$\mathbf{x}, \mathbf{z}\in\mathbb{R}^{n+1}_+$ and
$\rho(\mathbf{x},\mathbf{z})\le1/2$, then
\begin{equation}\label{dafg1}
(1+2\rho(\mathbf{x},\mathbf{z}))^{-1} x_{n+1} \leq z_{n+1}\leq
(1+2\rho(\mathbf{x},\mathbf{z})) x_{n+1}
\end{equation}
and
\begin{equation*}
\frac12 {|\mathbf{x}-\mathbf{z}| }\leq { x_{n+1}}
\rho(\mathbf{x},\mathbf{z}) \leq 2 |\mathbf{x}-\mathbf{z}|.
\end{equation*}
\end{lemma}

\begin{lemma}\textnormal{\cite[Lemma
A.11]{dsy}}\label{lem-3-2-0}
There
exists a positive constant $C$ such that, for any
$R\in(1,\infty)$ and
$(x,y)\in\mathbb{R}_+^{n+1}$,
\begin{align*}
\left\{ (u,v)\in\mathbb{R}_+^{n+1}:\ |u-x|\leq R y,\ R^{-1}y
\leq v \leq Ry \right\}
\subset B_\rho\left( (x,y), CR\right),	
\end{align*}
where $C\in(0,\infty)$ is a constant depending only on $n$.
\end{lemma}

\begin{lemma}\textnormal{\cite[Lemma A.10]{dsy}}\label{dda2f}
\begin{itemize}
\item[\rm(i)]
If $\{A^\alpha\}_{\alpha\in \mathcal{A}}$ is a collection of
subsets of $\mathbb{R}^{n+1}_+$ having Property I with
constants $\delta, \delta'\in (0,\frac 1{10})$ (i.e.,
\eqref{5-1-00} is satisfied for any set in the collection),
then the union $\bigcup_{\alpha\in \mathcal{A}} A^\alpha$
also has Property I with constants depending only on
$\delta, \delta'$.
\item[\rm(ii)]
For any $E\subset\mathbb{R}^{n+1}_+$ and any constant
$a\in(0,\infty)$, the set $E_a:=\{ \mathbf{x}\in
\mathbb{R}_+^{n+1}:\ d (\mathbf{x}, E) <a\}$ has Property I
with constants depending only on $a$.
\end{itemize}
\end{lemma}

Finally, the \emph{hyperbolic length}
$L(\gamma)$
of a piecewise $C^1$-curve
$$
\gamma:=\left\{\gamma(t):=\left(\gamma_1(t),\ldots,
\gamma_{n+1}(t)\right)\in\mathbb{R}_+^{n+1}:\ t\in
[a,b]\right\},\ -\infty<a<b<\infty,
$$
 is given by setting
$$
L(\gamma) :=\int_a^b\frac {|\gamma'(t)|}
{\gamma_{n+1}(t)}\,dt.
$$
It is well known (see \cite[Theorem A.4]{dsy}) that,
for any two distinct points $\mathbf p, \mathbf
q\in\mathbb{R}_+^{n+1}$, there exists a unique $C^1$-curve
$\gamma(\mathbf p, \mathbf q)$ connecting
$\mathbf p$ and $\mathbf q$ such that
$$\rho(\mathbf p,\mathbf q)=L(\gamma(\mathbf p, \mathbf
q))=\inf_\gamma L(\gamma)$$
with the infimum being taken over all piecewise $C^1$-curves
$\gamma$ in $\mathbb{R}_+^{n+1}$ connecting $\mathbf p,
\mathbf q\in\mathbb{R}_+^{n+1}$. The curve $\gamma(\mathbf p,
\mathbf q)$ is called the \emph{geodesic} connecting
$\mathbf{p}$ and $\mathbf{q}$.

\begin{lemma}\label{cor-3-9}
Let $f\in C^1(\mathbb{R}_+^{n+1})$ and $\mathbf p, \mathbf
q\in \mathbb{R}_+^{n+1}$. Then
	\begin{align*}
	|f(\mathbf p)-f(\mathbf q)|\leq \left[\max_{\mathbf x \in
\gamma(\mathbf p,\mathbf q)} x_{n+1}|\nabla f(\mathbf
x)|\right] \rho(\mathbf p,\mathbf q),
	\end{align*}
where $\mathbf{x}=(x, x_{n+1})\in\mathbb{R}_+^{n+1}$ and
$\gamma(\mathbf p,\mathbf q)$ is the geodesic connecting
$\mathbf{p}$ and $\mathbf{q}$.
\end{lemma}

\begin{proof}
Let $\gamma:\ [0,1]\to\mathbb{R}_+^{n+1}$ be a parametric
representation of the geodesic $\gamma(\mathbf p,\mathbf q)$
such that $\gamma(0)=\mathbf{p}$ and $\gamma(1)=\mathbf{q}$.
Then we have
\begin{align*}
|f(\mathbf p)-f(\mathbf q)|&=|f(\gamma(0))-f(\gamma(1))|\leq
\int_0^1 |\nabla f(\gamma(t))| |\gamma'(t)| \, dt\\
& \leq \left[\max_{t\in [0,1]} |\nabla f(\gamma(t))|
|\gamma_{n+1} (t)|\right] \int_0^1 \frac {|\gamma'(t)|}
{\gamma_{n+1}(t)}\, dt\\
&= \left[\max_{\mathbf x\in\gamma(\mathbf p,\mathbf q)}
x_{n+1} |\nabla f(\mathbf
 x)| \right]\rho(\mathbf p\mathbf , \mathbf q).
\end{align*}
This finishes the proof of Lemma \ref{cor-3-9}.
\end{proof}

We now state several key estimates established in \cite{dsy}.
Let $X$ be a quasi-normed lattice of function sequences, and
let $s \in(0,\infty)$. Denote by $\Lambda_X^{s}$ the
Daubechies $s$-Lipschitz $X$-based space. Let $r_1 \in
\mathbb{N}$ be such that $r_1 > s$, and let $L$ denote the
regularity of the Daubechies wavelet system
$\{\psi_\omega\}_{\omega \in \Omega}$. We retain this symbol
in the statements of the next four lemmas.
All positive constants of equivalence in the estimates below
are independent of $f \in \Lambda_s$.

The following lemmas, established in \cite{dsy}, play a
crucial role in the proofs of the main results of this
article.

\begin{lemma} \textnormal{\cite[Lemma
4.13]{dsy}}\label{a2.15s}
Assume that $X$ satisfies Assumption \ref{a1}
for some $\theta\in(0,\infty)$. Let $\delta\in(0,1/4)$ and
$R>\delta$.
Then, for any sequence $\{A_j\}_{j\in\mathbb{Z}_+}$ of
measurable subsets of $\mathbb{R}^{n+1}_+$,
\begin{equation*}
\left\|\left\{\left[\int_{0}^{\infty}
\mathbf{1}_{(A_j)_R}(\cdot,y)\,\frac{dy}{y}\right]^{\theta}
\right\}_{j\in\mathbb{Z}_+}\right\|_{X}
\le C\left\|\left\{ \left[\int_{0}^{\infty}
\mathbf{1}_{(A_j)_{\delta}}(\cdot,y)\,\frac{dy}{y}
\right]^{\theta}
\right\}_{j\in\mathbb{Z}_+}\right\|_{X},
\end{equation*}
where $C\in(0,\infty)$ is a constant depending only on $R$,
$\delta$ and the framework parameters.
\end{lemma}

\begin{lemma}\textnormal{\cite[Theorem 2.4]{dsy}}\label{wfawd}
Assume that $L>\max\{s, r_1-1\}$ and $X$ satisfies Assumption
\ref{a1}
for some $\theta\in(0,\infty)$.
Then,
for any $f\in\Lambda_s $,
 \begin{align*}
 d \left(f, \Lambda_X^{s}\right)_{\Lambda_s} \sim
 \varepsilon_{X}^0f+ \varepsilon_{X,s,\theta}f,
\end{align*}
where the positive equivalence constants depend only on the
framework parameters, the quantity
$\varepsilon_{X}^0f$ is given in \eqref{4-6-1},
\begin{align}
\varepsilon_{X,s,\theta}f:&=\inf\left\{\varepsilon\in(0,\infty):\
\left\|\left\{ \left[\int_{0}^1\mathbf{1}_{S_{r_1,j}(s,f,
\varepsilon)}(\cdot,y)\,\frac{dy}{y}\right]^{\theta}
\right\}_{j\in\mathbb{Z}_+}
\right\|_{X}<\infty\right\},\label{4-5-1}
\end{align}
and, for any $j\in\mathbb{Z}$,
\begin{align}\label{dfmawlmfp-1}
S_{r_1,j}(s,f, \varepsilon):&=\left\{
(x, y)\in\mathbb{R}_+^{n+1}:\
\Delta_{r_1} f(x, y)>\varepsilon y^s,\
y\in(2^{-j-1},2^{-j}]\right\}.
\end{align}
\end{lemma}

\begin{lemma}\textnormal{\cite[Theorem
4.4]{dsy}}\label{thm3adf}
 Assume that $f\in\Lambda_s $ and
 $0<\varepsilon_1<\varepsilon$. Let $S_{r_1,j}(s, f, \cdot)$
 be defined in \eqref{dfmawlmfp-1}.
Then there exists a positive constant $\delta$, depending only
on $f$, $\varepsilon$, $\varepsilon_1$ and the framework
parameters,
such
that, for any $j\in\mathbb{Z}$,
$$
 [S_{r_1,j}(s, f, \varepsilon)]_\delta\subset
 S_{r_1,j-1}(s, f, \varepsilon_1)\cup
 S_{r_1,j}(s, f, \varepsilon_1)\cup S_{r_1,j+1}(s, f,
 \varepsilon_1).
$$
\end{lemma}

\begin{lemma}\label{thm-711}
\textnormal{\cite[Theorem
2.5]{dsy}}
Assume that $L>\max\{s, r_1-1\}$ and $X$ satisfies Assumption
\ref{pplp} for some Carleson-type measure
$\nu$. Then,
for any $f\in\Lambda_s$,
 \begin{align*}
 d \left(f, \Lambda_X^{s}\right)_{\Lambda_s}&\sim
 \varepsilon_{X,\nu}^0+
\varepsilon_{X, S, \nu},
\end{align*}
where the positive equivalence constants depend only on the
framework parameters,
 $\varepsilon_{X,\nu}^0$ is defined in \eqref{3-11c}, and
\begin{align*}
\varepsilon_{X, S,
\nu}:&=\inf\left\{\varepsilon\in(0,\infty):\
\nu\left(\left\{S_{r_1,j}(s,f,
\varepsilon)\right\}_{j\in\mathbb{Z}_+}\right)<\infty\right\}.
\end{align*}
\end{lemma}

\section{Kernel estimates of fractional heat
semigroups}\label{sec-5}

In this section, we establish pointwise kernel estimates for
the fractional heat semigroup $T_{\alpha, t}$, with
$t\in(0,\infty)$ and $\alpha\in(0,\infty)$, and use them to
derive uniform norm estimates for the derivatives
$\partial_t^i \partial_x^\beta [T_{\alpha, t} f(x)]$,
 where $f\in\Lambda_s$, $\beta\in\mathbb{Z}_+^n$ and
 $i\in\mathbb{Z}_+$.
In particular, we characterize the space $\Lambda_s$
 in terms of the uniform norms of these derivatives. These
 results play a crucial role in the proofs of the main
 theorems presented in subsequent sections. Finally, we note
 that some of these results are likely known at least in
 special cases such as $\alpha=1$ or $\alpha=2$; however, due
 to the lack of precise references, we provide self-contained
 proofs for completeness.

We begin with estimates for fractional heat kernels.
For any $\alpha\in (0, \infty)$, let $K_\alpha$ be the Fourier
 transform of the function $ e^{-|2\pi \cdot|^{\alpha}}$; that
 is,
$$ K_\alpha(x):= \int_{\mathbb{R}^n} e^{-|2\pi \xi|^{\alpha}}
e^{2\pi i x\cdot \xi}\, d\xi,\ \ x\in\mathbb{R}^n.$$
In particular, when $\alpha=1$, $K_\alpha$ is the classical
Poisson kernel,
 		$K_{1} (x) =C_n (1+|x|^2) ^{-\frac {n+1}2}$ for any
 $x\in\mathbb{R}^n$,
 		while, for $\alpha=2$, $K_\alpha$ coincides with the
 heat kernel,
 		$ K_2(x) = C_n e^{- |x|^2/4}.$
 For general $\alpha\in(0,\infty)$, the function $K_\alpha$
 satisfies the following estimates.

\begin{theorem}\label{lem-2-1} For any given
$\alpha\in(0,\infty)$, $K_\alpha$ is a $C^\infty$-function on
$\mathbb{R}^n$ satisfying that
there exists a positive constant $C_{k, \alpha, n}$, depending
only on $k,\alpha$ and $n$, such that, for any
 $k\in\mathbb{Z}_+$ and $ x\in\mathbb{R}^n$,
 	\begin{align}\label {1-9} | \nabla^k K_\alpha (x)|\leq
 C_{k,\alpha,n} (1+|x|)^{-\alpha-n-k}.
 	\end{align}
\end{theorem}

In particular, Theorem \ref{lem-2-1} implies that $K_\alpha$
 is a radial and integrable function on $\mathbb{R}^n$. Since
 $$ \mathcal{F}(T_{\alpha, t} f)(\xi) =e^{-t|2\pi
 \xi|^{\alpha}} \mathcal{F}f(\xi),$$
it follows that
\begin{align}T_{\alpha,t} f(x)=\int_{\mathbb{R}^n}
K_{\alpha}(x-z, t) f(z)\, dz, \ \ x\in\mathbb{R}^n,\ \
t\in(0,\infty), \label{5-2}
\end{align}
where
 \begin{align*}K_{\alpha} (x, t) = \left( \frac 1 {t^\frac
 1{\alpha}}\right)^nK_\alpha \left( \frac x{ t^{ \frac
 1{\alpha}}} \right).
 \end{align*}
 As a direct consequence of Theorem \ref{lem-2-1}, we obtain
 the following result.
 \begin{corollary}\label{cor-5-2}
Let $\alpha\in(0,\infty)$ and $k, j\in\mathbb{Z}_+$. Then
there exists a positive constant $C$ such that, for any $
x\in\mathbb{R}^n$
and $ t\in(0,\infty)$,
 		\begin{align}\label {1-10}
 		\left|\partial_t^j\nabla_x^k (K_\alpha (x, t)) \right|
 \leq C\frac {t^{-j+1}}{( |x|+t^\frac
 1{\alpha})^{n+\alpha+k}}.
 \end{align}
 \end{corollary}

 For later applications, it will be more convenient to express
 the estimate \eqref{1-10} in the form
 \begin{align}\label{5-5c}
 \left|\partial_t^j\nabla_x^k (K_\alpha (x, t)) \right| \leq C
 t^{-\frac {k+n}\alpha -j} H_k\left( \frac x{t^{\frac
 1\alpha}}\right),\ \ x\in\mathbb{R}^n,\ \ t\in(0,\infty),
 \end{align}
 where
 \begin{align}\label{5-6c}
 H_k(x):=\frac1{(1+|x|)^{n+\alpha+k}},\ \
 x\in\mathbb{R}^n.\end{align}

The estimates in Theorem \ref{lem-2-1} for the fractional heat
kernels are likely known, particularly for $\alpha\in(0,2]$,
but since we were unable to locate precise references, we will
include a self-contained proof for completeness. When $\alpha$
is a positive even integer, the function $e^{-|\cdot|^\alpha}$
is a Schwartz function on $\mathbb{R}^n$, and therefore the
estimate \eqref{1-9} is not sharp. For example, in the case
$\alpha = 2$, the kernel $K_2(x)=c_n e^{-|x|^2}$ is the
classical heat kernel, which decays exponentially at infinity.
Nevertheless, the estimates given in the theorem are
sufficient for our purposes.

To prove Theorem \ref{lem-2-1}, we begin by recalling some
known facts about homogeneous distributions. Let $\mathcal{S}$
denote the space of all Schwartz functions on $\mathbb{R}^n$.
For any
$\gamma:=(\gamma_1,\ldots,\gamma_n)\in\mathbb Z_+^n$ and
$x:=(x_1,\ldots,x_n)\in\mathbb{R}^n$, let
$|\gamma|:=\gamma_1+\cdots+\gamma_n$,
$x^\gamma:=x_1^{\gamma_1}\cdots x_n^{\gamma_n}$, and
$\partial^\gamma:=(\frac{\partial}{\partial
x_1})^{\gamma_1}\cdots(\frac{\partial}{\partial
x_n})^{\gamma_n}$.

\begin{definition}\textnormal{\cite[pp.\,128,\,(2.4.7)]{Gr}}\label{def-5-3}
	Let $z\in\mathbb{C}$ be such that $\text{Re}\, z >-n-m-1$
for a positive integer $m\in\mathbb{N}$. The homogeneous
tempered distribution $u_z: \mathcal{S}\to\mathbb{C}$ is
defined for any $f\in\mathcal{S}$ by setting
\begin{align}\label{5-6b}
\langle u_z, f\rangle:&=a_{n,z}\left[\int_{|x|> 1} f(x)|x|^z\,
dx \right.\notag\\
&\quad\left.+\sum_{|\beta|\leq m }b(n,\beta,z) \cdot\partial^\beta
f(\mathbf{0}) + \int_{|x|\leq 1} \left(f(x)-T_0^m
f(x)\right)|x|^z d x \right],
\end{align}
	where $a_{n,z} =
\frac{
\pi^{\frac{z+n}{2}}}{
\Gamma\left(\frac{z+n}{2}\right)}, $
$$b(n,\beta, z):=\frac{ \int_{\mathbf{S}^{n-1}} \theta^\beta d
\sigma( \theta)}{(|\beta|+z+n)\cdot \beta!},\ \ \text{ and}\ \
T_0^m f(x)=\sum_{|\beta| \leq m} \frac{\partial^\beta
f(\mathbf{0})}{\beta!} x^\beta.$$
\end{definition}
It is well known that the above definition is independent of
the selection of the integer $m$ and, moreover, if
$\mathrm{Re}z > -n$, then the homogeneous distribution $u_z$
coincides with the locally integrable function
$a_{n,z} |\cdot|^z$ on $\mathbb{R}^n$ (see
\cite[p.\,127]{Gr}).
Furthermore, a straightforward calculation shows that
\eqref{5-6b} can be equivalently expressed, in terms of a
$C^\infty$-function $\eta$ on $\mathbb{R}^n$ such that
$\eta(x)=1$ for $|x|\leq 1$ and $\eta(x)=0$ for $|x|\ge 2$, as
follows:
\begin{align}
\langle u_z, f\rangle&=a_{n,z}\left[\int_{\mathbb{R}^n}
(1-\eta(x)) f(x)|x|^z\, dx +\sum_{
	 		|\beta|\leq m } c(n,\beta,z,\eta)
\cdot\partial^\beta f(\mathbf{0})\right.\label{5-8b}\notag\\
&\quad+ \left.\int_{\mathbb{R}^n} \left(f(x)-T_0^m
f(x)\right)\eta(x)|x|^z d x \right],
\end{align}
where
$$c(n,\beta,z,\eta)=b(n,\beta, z)+\frac 1 {\beta!}\int_{|x|\ge
1} \eta(x) x^\beta |x|^z\, dx.$$
We now recall the following well-known result on the
distributional Fourier transform of homogeneous distributions.

\begin{lemma}\textnormal{\cite[Theorem
2.4.6]{Gr}}\label{thm1-6-8} For any $z\in\mathbb{C}$,
$\widehat{u_z}=u_{-n-z},$
where $\widehat{u_z}$ denotes the distributional Fourier
transform of $u_z$.
\end{lemma}

For later use, we define the modulation and the translation
operators for any $f \in L_{\mathrm{loc}}^1 $, $\xi \in
\mathbb{R}^n$, and $x \in \mathbb{R}^n$, respectively,
by setting
$$
\mod_\xi f(x) := f(x)\, e^{2\pi i x \cdot \xi} \quad
\text{and} \quad \text{Trans}_\xi f(x) := f(x - \xi).$$
The following lemma will play an important role in the proof
of Theorem~\ref{lem-2-1}.

\begin{lemma}\label{lem-5-5} Assume that
$\gamma\in(0,\infty)$, $\beta\in\mathbb{Z}_+^n$, and $g\in
\mathcal{S}$.
Let
$$H_{\gamma,\beta}(x):=\int_{\mathbb{R}^n} |y|^\gamma y^\beta
g(y) e^{2\pi ix\cdot y}\, dy,\ \ x\in\mathbb{R}^n.$$
Then there exists a positive constant $C_{n,\gamma,\beta, g}
$, depending only on $n,\gamma,\beta$ and $g$, such that, for
any $x\in\mathbb{R}^n$,
$$ |H_{\gamma,\beta} (x)|\leq C_{n,\gamma,\beta, g}
(1+|x|)^{-n-\gamma-|\beta|}.$$
\end{lemma}
\begin{proof} For convenience, we let $g_\beta(x):=x^\beta
g(x)$ for any $x\in\mathbb{R}^n$ and let
$G=\mathcal{F}^{-1}g$.
 Then both $G$ and $g_\beta$ are Schwartz functions, and
 $\mathcal{F}^{-1} g_\beta =c_\beta \partial^\beta G$. Since
 the homogeneous distribution $u_\gamma$ coincides with the
 locally integrable function $c_{n,\gamma} |\cdot|^\gamma$, we
 have
\begin{align*}
H_{\gamma,\beta} (x)
&=c_{n,\gamma} \langle u_\gamma,\mathrm{mod}_{x} g_\beta
\rangle=c_{n,\gamma, \beta} \langle u_\gamma, \mathcal{F}
(\text{Trans}_{-x} \partial^\beta G)\rangle\\
&=c_{n,\gamma, \beta} \langle \widehat {u_\gamma},
\text{Trans}_{-x} \partial^\beta G\rangle=c_{n,\gamma, \beta}
\langle u_{-n-\gamma}, \partial^\beta G(x+\cdot)\rangle,
\end{align*}
where the last step uses Lemma \ref{thm1-6-8}.

Next, let $m\in\mathbb{N}$ be such that $m\ge \gamma+1$, and
let $\eta$ be a $C^\infty$-function on $\mathbb{R}^n$ such
that $\eta(x)=1$ for $|x|\leq 1$ and $\eta(x)=0$ for $|x|\ge
2$. Using \eqref{5-8b} with $z=-n-\gamma$ and taking into
account the fact that $G\in\mathcal{S}$, we conclude that
\begin{align*}
|H_{\gamma,\beta} (x) |&\lesssim \left|\int_{\mathbb{R}^n}
\partial^\beta G(x+y) |y|^{-n-\gamma}(1-\eta(y))\, dy\right| +
\max_{0\leq k\leq m+|\beta|} |\nabla^k G(x)|\\
&\quad+ \sup_{|y|\leq 2}
|\nabla^{m+1+|\beta|}G(x+y)|\int_{|y|\leq 2} |y|^{-n-\gamma
+m+1}\, dy\\
&\lesssim \left|\int_{|y|\ge 1} G(x+y) \cdot \partial^\beta
\left[|y|^{-n-\gamma}(1-\eta(y))\right]\,
dy\right|+(1+|x|)^{-n-\gamma-|\beta|}. \end{align*}
However, a straightforward calculation shows that
\begin{align*}
\left|\int_{|y|\ge 1} G(x+y) \cdot \partial^\beta
\left[|y|^{-n-\gamma}(1-\eta(y))\right]\, dy\right|&\lesssim
\int_{\mathbb{R}^n}
(1+|x+y|)^{-n-\gamma-|\beta|}(1+|y|)^{-n-\gamma-|\beta|}\,
dy\\
&\lesssim (1+|x|)^{-n-\gamma-|\beta|}.
\end{align*}
Thus, putting the above estimates together, we obtain the
desired estimate for $|H_{\gamma,\beta} (x) |$, which hence
completes the proof of Lemma \ref{lem-5-5}.
\end{proof}

\begin{proof} [Proof of Theorem \ref{lem-2-1}] Let
$G_\alpha(x):=e^{-|2\pi x|^\alpha}$ for any
$x\in\mathbb{R}^n$.
Without loss of generality, we may assume that $\alpha$ is not
an even integer because otherwise $G_\alpha$ is a Schwartz
function and the stated estimate holds trivially. Since
$K_\alpha =\widehat {G_\alpha}$ and
 the function $(1+|\cdot|)^i G_\alpha$ is integrable over
 $\mathbb{R}^n$ for any integer $i\in\mathbb{N}$, it follows
 that $K_\alpha$ is a $C^\infty$-function on $\mathbb{R}^n$.

To prove the estimate \eqref{1-9}, let
 $\beta\in\mathbb{Z}_+^n$ be such that $|\beta|=k$, and let
 $\eta$ be a $C^\infty$-function on $\mathbb{R}^n$ such that
 $\eta(x)=1$ for $|x|\leq 1$ and $\eta(x)=0$ for $|x|\ge 2$.
 We then write
 \begin{align}
 \partial^\beta K_\alpha(x) =c\int_{\mathbb{R}^n} e^{-|2\pi
 y|^\alpha} y^\beta e^{-2\pi i x\cdot y}\, dy =
 I_{\alpha,\beta}(x) +J_{\alpha,\beta}
 (x),\label{5-9a}\end{align}
 where
$
I_{\alpha,\beta}(x):=c\int_{\mathbb{R}^n} \eta(y) e^{-|2\pi
y|^\alpha} y^\beta e^{-2\pi i x\cdot y}\, dy,
$
\begin{align*}
J_{\alpha,\beta}(x):=c\int_{\mathbb{R}^n}[1- \eta(y)]
e^{-|2\pi y|^\alpha} y^\beta e^{-2\pi i x\cdot y}\, dy,
\end{align*}
and $c$ is a constant independent of $x$.

 For the integral $J_{\alpha,\beta}$, we write
 $J_{\alpha,\beta}(x)=c \widehat{F_{\alpha,\beta}}(x)$, where
 $F_{\alpha,\beta}(y):= [1- \eta(y)] e^{-|2\pi y|^\alpha}
 y^\beta$
for any $y\in \mathbb{R}^n$.
 Since $F_{\alpha,\beta}$ is a Schwartz function on
 $\mathbb{R}^n$, we deduce that
 \begin{align}|J_{\alpha,\beta}(x)|\leq C_N (1+|x|)^{-N},\ \
 \forall \,x\in\mathbb{R}^n,\ \ \forall \,N\in
 \mathbb{N}.\label{5-10a}\end{align}

 To estimate the integral $I_{\alpha,\beta}(x)$, let
 $N,\ell\in\mathbb{N}$ be such that $n+\alpha +k \leq N
 <n+k+\alpha+1$ and $N+2\leq \ell \alpha$. Write
 \begin{align}I_{\alpha, \beta }(x) =c\int_{\mathbb{R}^n }
 D_\ell(y) e^{-2\pi i x\cdot y}\, dy+\sum_{j=0}^\ell \frac
 {(-1)^j}{j!} I_{\alpha,\beta, j}(x),\label{5-11a}\end{align}
 where
\begin{align*}
D_\ell(y)&:=\left[e^{-|2\pi y|^\alpha}-\sum_{j=0}^\ell \frac
{(-1)^j |2\pi y|^{j\alpha}}{j!}\right] \eta(y) y^\beta
\end{align*}
and
\begin{align*}
 I_{\alpha,\beta, j}(x)&:=c\int_{\mathbb{R}^n } \eta(y)
 y^\beta |2\pi y|^{j\alpha}e^{-2\pi i x\cdot y}\, dy,\ \ 0\leq
 j\leq \ell.
\end{align*}
 Since $\ell\alpha+k\ge N+2$, it follows from Taylor's theorem
 that
 $D_\ell \in C_c^{N}$, which implies
 $$ \left|\int_{\mathbb{R}^n } D_\ell(y) e^{-2\pi i x\cdot
 y}\, dy\right|=\left|\widehat {D_\ell} (x)\right|\leq C
 (1+|x|)^{-N}\leq C (1+|x|)^{-n-\alpha-k}.$$
 It remains to estimate the integral $I_{\alpha,\beta, j}(x)$
 for $0\leq j\leq \ell$.
 For $j=0$, since $y^\beta \eta(y)$ is a Schwartz function, it
 follows that
 $$ |I_{\alpha,\beta,0}(x)|\leq C_N (1+|x|)^{-N},\ \ \forall
 x\in\mathbb{R}^n,\ \ \forall N\in \mathbb{N}.$$
 For $1\leq j\leq \ell$, we use Lemma \ref{lem-5-5} to obtain
 $$ |I_{\alpha, \beta, j} (x)|\lesssim
 (1+|x|)^{-n-j\alpha-k}\lesssim (1+|x|)^{-n-\alpha-k}.$$
 Substituting the above estimates into \eqref{5-11a} yields
 \begin{align}
 |I_{\alpha,\beta}(x)|\lesssim (1+|x|)^{-n-\alpha-k}.
 \label{5-12a}\end{align}

 Finally, combining \eqref{5-10a} and \eqref{5-12a} with
 \eqref{5-9a}, we conclude the desired estimate \eqref{1-9}.
 This finishes the proof of Theorem \ref{lem-2-1}.
 \end{proof}

 Next, we apply the estimates from Theorem \ref{lem-2-1} to
 characterize the space $\Lambda_s$ through the derivatives of
 the fractional heat semigroup.
 To this end, let us first review some semigroup properties of
 $T_{\alpha, t}$ with $\alpha\in(0,\infty)$ and
 $t\in[0,\infty)$.
 Let $Y$ denote the Banach space of all bounded uniformly continuous
 functions on $\mathbb{R}^n$, equipped with the uniform norm
 $\|\cdot\|_\infty$. For simplicity, we fix
 $\alpha\in(0,\infty)$ and let
 $\mathcal{A}:=(-\Delta)^{\alpha/2}$.
For each integer $r\in\mathbb{N}$, define
$$ {\mathcal D}(\mathcal{A}^r):=\left\{ f\in Y:\ \mathcal{A}^i
f \in Y\ \ \text{for any $ i=1,\ldots, r$}\right\}.$$
Clearly, ${\mathcal D}(\mathcal{A}^r)$ contains all bounded
continuous functions on $\mathbb{R}^n$ whose distributional
Fourier transforms have compact support. In particular, this
implies that ${\mathcal D}(\mathcal{A}^r)$ is a dense subset
of $Y$.

By Corollary \ref{cor-5-2} and the definition \eqref{1-6-1},
it is easy to verify that the family $\{T_{\alpha,
t}\}_{t\in(0,\infty)}$ consists of uniformly bounded operators
on $L^\infty$
 satisfying the following conditions:
\begin{enumerate}[\rm (i)]
		\item 	 $T_{\alpha, 0}=I$ and $\displaystyle\lim
_{t \rightarrow 0+}\|T_{\alpha,t} f-f\|_\infty=0$ for any $f
\in Y$, here and hereafter, we denote $\|\cdot\|_{L^\infty}$
simply by $\|\cdot\|_{\infty}$;
		\item $T_{\alpha,
t_1+t_2}=T_{\alpha,t_1}T_{\alpha,t_2}$ for all $t_1, t_2\in
[0,\infty)$;
\item for any $f \in \mathcal{D}(\mathcal{A})$,
$$ \lim_{t\to 0+}\left\| \frac {T_{\alpha, t} f-f} t
-\mathcal{A} f\right\|_\infty=0;$$
\item for any	 $t\in(0,\infty)$ and $f\in Y$, we have
$T_{\alpha, t} f\in {\mathcal D}(\mathcal{A})$ and
 $ \frac {\partial} {\partial t} T_{\alpha,t} f= \mathcal{A}
 T_{\alpha,t} f$ ;
 moreover, there exists a constant $C\in(1,\infty)$ such
 that, for any $ t\in(0,\infty)$ and $ f\in Y$,
	\begin{align*}
	t\|\mathcal{A} T_t f\|_\infty\leq C \|f\|_\infty;
	\end{align*}
\item for any $t\in(0,\infty)$ and $f \in
\mathcal{D}(\mathcal{A})$, we have
		$$
		\frac{d}{d t} T_{\alpha, t}f =\mathcal{A}
T_{\alpha,t} f =T_{\alpha, t} \mathcal{A} f.$$
\end{enumerate}
In particular, this means that
$\{T_{\alpha,t}\}_{t\in(0,\infty)}$ is a holomorphic semigroup
on the space $Y$ with the
infinitesimal generator $\mathcal{A}=(-\Delta)^{\alpha/2}$
(see \cite[Section 5]{di}).
It follows by \cite[Theorem 5.1]{di} that,
for any $r\in\mathbb{N}$ and $f\in Y$,
\begin{align}
\|(I-T_{\alpha, t})^r f\|_\infty \sim K (f, \mathcal{A}^r;
t^r)_\infty,\ \ t\in(0,\infty),\label{5-15}
\end{align}
where the positive constants of equivalence are independent of
$f$ and $t$, and $K (f, \mathcal{A}^r, t)_\infty$ is the
$K$-functional defined by
$$ K (f, \mathcal{A}^r, t)_\infty :=\inf \left\{
\|f-g\|_\infty+ t \|\mathcal{A} ^r g\|_\infty:\ \
g\in{\mathcal D}(\mathcal{A}^r)\right\},\ \ t\in(0,\infty). $$
Furthermore, from the proof of \cite[Theorem 5.1]{di},
it follows that
there exists a positive integer $m$, depending only on
$\alpha$ and $n$, such that, for any $r\in\mathbb{N}$ and
$f\in Y$,
\begin{align}
t^{r}\left\| \mathcal{A}^r (T_{\alpha, mrt} f) \right\|_\infty
\leq C_{r,\alpha,n} \left\| (I-T_{\alpha,t})^r
f\right\|_\infty,\ \ t\in(0,\infty). \label{5-16}
\end{align}
	
The space $\Lambda_s$ can be characterized in terms of the
fractional heat semigroup as follows.
	
 \begin{theorem}\label{dsfad}
Let $s,\alpha\in(0,\infty)$ and $r\in\mathbb{N}$ be
such that $\alpha r>s$, and let $f\in L^\infty \cap
\mathcal C $. Then $f$ belongs to the space $\Lambda_s$ if and
only if
\begin{align}
 \sup_{
			t\in(0,\infty)} t^{r-\frac s\alpha}
\left\|(-\Delta)^{\frac {\alpha r}2} (T_{\alpha, t} f)
\right\|_\infty<\infty, \label{5-17}
\end{align}
in which case,
\begin{align}
 \|f\|_{\Lambda_s}\sim \sup_{t\in (0,\infty)} t^{r-\frac
 s\alpha } \left\|(-\Delta)^{\frac {\alpha r}2} (T_{\alpha, t}
 f) \right\|_\infty.\label{5-18}
\end{align}		
\end{theorem}

\begin{proof}
As is well known (see \cite[Theorem 6.7.4]{bl}),
\begin{align}\label{zuihou} f\in\Lambda_s\iff \sup_{t\in(0,\infty)} \frac {K(f,
\mathcal{A}^r, t^{\alpha r} )_\infty} {t^s} <\infty,
\end{align}
in which case, we have
\begin{align}\sup_{t\in(0,\infty)} \frac {K(f, \mathcal{A}^r,
t^{\alpha r} )_\infty} {t^s}\sim
\|f\|_{\Lambda_s}.\label{5-19}\end{align}

If $f\in\Lambda_s$, then, using \eqref{5-15} and \eqref{5-16},
we obtain
\begin{align*}
\sup_{t\in(0,\infty)} t^{r-\frac s\alpha}
\left\|(-\Delta)^{\frac {\alpha r}2} (T_{\alpha, t} f)
\right\|_\infty&=c \sup_{t\in(0,\infty)} t^{r-\frac
s\alpha}\left\| (-\Delta)^{\frac {\alpha r}2} (T_{\alpha, mrt}
f) \right\|_\infty\\
&\lesssim \sup_{t\in(0,\infty)} t^{-\frac
s\alpha}K(f,\mathcal{A}^r, t^r)_\infty\leq C
\|f\|_{\Lambda_s},
\end{align*}
where the last step used \eqref{5-19}.

Conversely, assume that \eqref{5-17} holds; that is, there
exists a constant $A\in(0,\infty)$ such that
\begin{align}\left\|(-\Delta)^{\frac {\alpha r}2} (T_{\alpha,
t} f) \right\|_\infty\leq A t^{\frac s\alpha-r},\ \ \forall\,
t\in(0,\infty).\label{5-19b}\end{align}
Using \eqref{5-15} and \eqref{5-19}, we find that
\begin{align}\|f\|_{\Lambda_s}\lesssim \frac {K(f,
\mathcal{A}^r, t^{\alpha r} )_\infty} {t^s}\lesssim \frac
{\|(I-T_{\alpha,t^\alpha})^r f\|_\infty} {t^s}. \label{5-20}
\end{align}
Let $u_\alpha(x, t):=T_{\alpha, t} f(x)$ for any
$x\in\mathbb{R}^n$ and $t\in(0,\infty)$. Then, for any $t\ge
0$, we have
\begin{align*}
(T_{\alpha,t }-I)^r f(x)&=\sum_{j=0}^r\binom{r} j (-1)^{r-j}
T_{\alpha, jt} f (x) =\sum_{j=0}^r\binom{r} j (-1)^{r-j}
u_\alpha(x, jt)\\
&=\int_{[0, t]^r} \partial_{n+1}^r u_\alpha (x,
t_1+\cdots+t_r) \, dt_1\cdots dt_r.
\end{align*}
Since
$$ \partial_{n+1}^r u_\alpha(x, t)=\left(\frac
{\partial}{\partial t}\right)^r[T_{\alpha, t}
f(x)]=(-\Delta)^{-\frac {\alpha r}2} T_{\alpha, t} f(x),$$
it follows that
\begin{align*}
 (T_{\alpha,t }-I)^r f = \int_{[0, t]^r} (-\Delta)^{-\frac
 {\alpha r}2} T_{\alpha, \ t_1+\cdots+t_r} ( f) \, dt_1\cdots
 dt_r,
\end{align*}
 which, using \eqref{5-19b}, implies
\begin{align*}
\left\|(T_{\alpha,t }-I)^r f\right\|_\infty&\leq \int_{[0,
t]^r} \left\|(-\Delta)^{-\frac {\alpha r}2} T_{\alpha, \
t_1+\cdots+t_r} (f) \right\|_\infty \, dt_1\cdots dt_r\\
 &\leq A \int_{[0, t]^r} (t_1+\cdots+t_r)^{\frac s\alpha -r}\,
 dt_1\cdots dt_r\\
 &\lesssim_r A \int_{\{ z\in \mathbb{R}^r:\ |z|\leq \sqrt{r}
 t\}}|z|^{\frac s\alpha -r}\,
 dz\lesssim_r t^{\frac s\alpha}.
\end{align*}

Substituting this last estimate with $t^\alpha$ in place of
$t$ into \eqref{5-20} yields the desired estimate:
\begin{align*}
\|f\|_{\Lambda_s}&\lesssim \frac {\|(I-T_{\alpha,t^\alpha})^r
f\|_\infty} {t^s}
\lesssim A.
\end{align*}
This finishes the proof of Theorem \ref{dsfad}.
\end{proof}

Let $f\in\Lambda_s$. Define $u_\alpha(x, t)=T_{\alpha, t}
f(x)$ for any $x\in\mathbb{R}^n$ and $t\in(0,\infty)$. By
\eqref{5-18}, if $r\in\mathbb{N}$ and $r>\frac s\alpha$, then
\begin{align}\label{5-21}
\sup_{t\in (0,\infty)}t^{r-\frac s\alpha}
\sup_{x\in\mathbb{R}^n} |\partial_{n+1}^r u_\alpha(x,
t)|\lesssim \|f\|_{\Lambda_s}.
\end{align}
The following theorem extends the inequality \eqref{5-21}.

\begin{theorem}\label{lem-2-4} Let $i\in \mathbb{Z}_+$ and
$\beta\in \mathbb{Z}_+^n$ with $|\beta| +i\alpha >s$. Let
$f\in \Lambda_s$ and
$u_\alpha(x, t) :=T_{\alpha, t} f(x)$ for any
$x\in\mathbb{R}^n$ and $t\in(0,\infty)$. Then there exists a
positive constant $C_{\alpha,\beta, n, i}$, depending only on
$\alpha,\beta, n$ and $i$, such that, for any
$x\in\mathbb{R}^n$ and $t\in(0,\infty)$,
\begin{align}\label{2-9} \sup_{t\in (0,\infty)}t^{\frac
{|\beta|}\alpha +i-\frac s\alpha}\sup_{x\in\mathbb{R}^n}
\left|\partial_{n+1}^{i} \partial_x^\beta u_\alpha(x, t)
\right|\leq C_{\alpha,\beta, n, i}
\|f\|_{\Lambda_s}.\end{align}
\end{theorem}

\begin{proof} We divide the proof into two steps.

{\bf Step 1.} We prove that \eqref{2-9} holds for all
$\beta\in\mathbb{Z}_+^n$ and $i\in\mathbb{Z}_+$ with $i>\frac
s\alpha$.

Let $x\in \mathbb{R}^n$ and $t\in(0,\infty)$. By \eqref{5-2},
  we have, for any $t_1, t_2\in(0,\infty)$,
\begin{align*}
u_\alpha(x, t_1+t_2) =T_{\alpha, t_1+t_2} f(x)=T_{\alpha, t_2}
T_{\alpha, t_1} f(x) =\int_{\mathbb{R}^n} u_\alpha(z, t_1)
K_\alpha (x-z, t_2)\, dz.
\end{align*}
This implies
\begin{align*}
	\partial_x^\beta	\partial_{n+1}^i u_\alpha(x,t_1+t_2)
&=\int_{\mathbb{R}^n} \left[\partial_{n+1}^i u_\alpha(z,
t_1)\right]\left[\partial_x^\beta K_\alpha (x-z, t_2)\right]\,
dz.
\end{align*}
Setting $t_1=t_2=\frac t2$ yields
\begin{align*}
	| \partial_x^\beta \partial_{n+1}^i u_\alpha(x,t)|&=
\left|\int_{\mathbb{R}^n} \left[\partial_{n+1}^i
u_\alpha\left(x-z, 2^{-1} t\right) \right]
	\left[ \partial_z^\beta K_\alpha\left(z, 2^{-1}
t\right)\right]\, dz \right|.\end{align*}
	By \eqref{5-21} applied to $r=i>\frac s\alpha$, we then
conclude that
	\begin{align*}
	| \partial_x^\beta \partial_{t}^i u_\alpha(x,t)|	
&\lesssim t^{\frac s\alpha -i} \|f\|_{\Lambda_s}
\int_{\mathbb{R}^n} \left| \partial_z^\beta K_\alpha\big(z,
2^{-1} t\big)\right|\, dz,\end{align*}
which, using \eqref{1-10}, is estimated by
\begin{align*} \lesssim t^{\frac s\alpha -i} \|f\|_{\Lambda_s}
\int_{\mathbb{R}^n}\frac {t}{( |z|+t^\frac
1{\alpha})^{n+\alpha+|\beta|}}\, dz\lesssim t^{\frac s\alpha
-i-\frac {|\beta|}\alpha} \|f\|_{\Lambda_s}.
\end{align*}
This proves \eqref{2-9} for all $\beta\in\mathbb{Z}_+^n$ and
$i\in\mathbb{Z}_+$ with $i>\frac s\alpha$.

{\bf Step 2.} Prove that, if $j\in\mathbb{Z}_+$, $\beta\in
\mathbb{Z}_+^n$, and $|\beta| +j\alpha > s$, then, for any
$(x, t)\in\mathbb{R}_+^{n+1}$, we have
 \begin{align}\label{2-10-c}
t^{ j+\frac {|\beta|} \alpha-\frac s\alpha } 	|
\partial_x^\beta \partial_{n+1} ^j u_\alpha(x,t) |\lesssim
\sup_{ y\in (t, \infty)}
 \left[ y^{j+1 +\frac {|\beta|}\alpha-\frac s\alpha }|
 \partial_{x}^\beta \partial_{n+1}^{j+1} u_\alpha(x,
 y)|\right].
\end{align}

We claim that the present theorem follows from \eqref{2-10-c}
and \eqref{2-9} by using the case already established in Step 1.
Indeed, once \eqref{2-10-c} is proven, then, for any
$j\in\mathbb{Z}_+$ and $\beta\in \mathbb{Z}_+^n$ with $|\beta|
+j\alpha > s$, we have
$$\sup_{t\in(0,\infty)} \left[t^{ j+\frac {|\beta|}
\alpha-\frac s\alpha } \sup_{x\in \mathbb{R}^n}	|
\partial_x^\beta \partial_{n+1} ^j u_\alpha(x,t)
|\right]\lesssim \sup_{ t\in(0,\infty)}
 \left[ t^{j+1 +\frac {|\beta|}\alpha-\frac s\alpha
 }\sup_{x\in\mathbb{R}^n} | \partial_{x}^\beta
 \partial_{n+1}^{j+1} u_\alpha(x, t)|\right].$$
This means that, if \eqref{2-9} holds for some $i=j+1$ and
$\beta\in \mathbb{Z}_+^n$ with $|\beta| +(j+1)\alpha > s$,
then it also holds for $i=j$ and the same $\beta\in
\mathbb{Z}_+^n$ satisfying $|\beta| +j\alpha > s$. Since
\eqref{2-9} has already been verified for all $i>\frac
s\alpha$ and $\beta\in\mathbb{Z}_+^n$ in Step 1, this implies
that
\eqref{2-9} holds for all $i\in\mathbb{Z}_+$ and $\beta\in
\mathbb{Z}_+^n$ such that $|\beta| +j\alpha > s$.

To show \eqref{2-10-c}, we first observe from \eqref{5-2} that
\begin{align*}
	\partial_{n+1}^{j} \partial_{x}^\beta u_\alpha(x,t)=f\ast
\left(\partial_{n+1}^{j}
	\partial^\beta K_\alpha (\cdot, t)\right)(x).
\end{align*}
Thus, using \eqref{1-10}, we obtain
\begin{align*}\label{eq:5-7} |\partial_{n+1}
^{j}\partial_x^\beta u_\alpha(x,t) |\leq
\left\|\partial_{n+1}^{j}
	 \partial^\beta K_\alpha(\cdot, t)\right\|_{L^1}
	\|f\|_\infty\lesssim t^{-\frac {|\beta|}\alpha -j} \to 0\
\ \text{as}\ \ t\to \infty.\end{align*}
It follows that
\begin{align*}
	|\partial_{n+1}^{j} \partial_x^\beta u_\alpha(x,t) | &\leq
\int_{t}^\infty |\partial^{j+1}_{n+1} \partial_x^\beta
u_\alpha(x, y)|\, dy=\int_{t}^\infty \left[
|\partial^{j+1}_{n+1} \partial_x^\beta u_\alpha(x, y)| y^{j+1
+\frac {|\beta|}\alpha -\frac {s}\alpha }\right] y^{-j-1
-\frac {|\beta|}\alpha +\frac s\alpha } \, dy\\
	& \leq \sup_{y\ge t } \left[ |\partial^{j+1}_{n+1}
\partial_x^\beta u_\alpha(x, y)| y^{j+1 +\frac {|\beta|}\alpha
-\frac s\alpha }\right]\cdot t^{-j -\frac {|\beta|} \alpha
+\frac s\alpha },
\end{align*}
where the last step used the assumption that $j+\frac
{|\beta|}\alpha >\frac s\alpha$.
This yields the desired estimate \eqref{2-10-c},
which completes
the proof
of Theorem \ref{lem-2-4}.
\end{proof}

\section{Proof of Theorem \ref{thm-2-66} (heat semigroup
characterization
of $d(f,\!\Lambda_{\!X}^{\!s})_{\!\Lambda_s}$ under Assumption
\ref{a1}) based on Propositions
\ref{thm-3adf}, \ref{cor-6-7-1}, and
\ref{cor-6-7-0}}\label{oijjl}

The proof of Theorem \ref{thm-2-66} is quite long. We outline
 it in this section.
Let $s,\alpha\in(0,\infty)$ and $r\in\mathbb N$ be such that
$\alpha r>s+3$. Let $L$ denote the regularity of the
Daubechies wavelet system $\{\psi_\omega\}_{\omega\in\Omega}$.
Recall that we refer to $\alpha, s, r, L, \theta$ and $n$ as
 framework parameters.
Throughout this section, the letters $C$ and $c$ denote
sufficiently large and sufficiently small constants,
respectively, both depending only on the framework parameters.
The exact values of these constants are not important.

We decompose the proof of Theorem \ref{thm-2-66} into several
steps. In the first step, we establish the following
proposition, which provides an estimate for the neighborhood
of the set $D_{\alpha,r, j}(s, f,\varepsilon)$ under the
hyperbolic metric.

\begin{proposition}\label{thm-3adf}
 Let $f\in\Lambda_s$ and
$\varepsilon,\varepsilon_1\in(0,\infty)$
with $\varepsilon>\varepsilon_1$.
Then there exists a positive constant $\delta$,
depending only on $\|f\|_{\Lambda_s}$, $\varepsilon$,
$\varepsilon_1$,
and the framework parameters,
such that, for any $j\in\mathbb \mathbb{Z}$,
$$
[D_{\alpha, r,j}(s, f, \varepsilon)]_\delta\subset D_{\alpha,
r,j-1}(s, f, \varepsilon_1)\cup
D_{\alpha, r,j}(s, f, \varepsilon_1)\cup D_{\alpha, r,j+1}(s,
f, \varepsilon_1),
$$
where, for any $ j\in\mathbb{Z}$,
$$ D_{\alpha, r, j}(s,f,\varepsilon):=\left\{ (x,t)\in
\mathbb{R}^n\times (2^{-j-1},2^{-j}]:\
 |\partial_{n+1}^r u_\alpha(x, t^\alpha)|>\varepsilon
 t^{s-\alpha r}\right\}.$$
\end{proposition}

The proof of this proposition will be given in Section
\ref{sec:7}.

Next, in the second step, we use finite differences to
establish certain lower estimates of the neighborhood of the
set $D_{\alpha,r, j}(s, f,\varepsilon)$.
For any integer $r_1>s$, $\varepsilon\in(0,\infty)$,
$j\in\mathbb{Z}$, and $f\in\Lambda_s$, define
\begin{align*}
S_{r_1,j}(s,f, \varepsilon):&=\left\{ (x, y)\in
\mathbb{R}_+^{n+1}:\
 \Delta_{r_1} f(x, y)>\varepsilon y^s,\ \ 2^{-j-1}<y\leq
 2^{-j}\right\}.
\end{align*}

\begin{proposition}\label{cor-6-7-1}
Let $r_1\in\mathbb{N}$ be such that $r_1>\alpha r+1$.
Then, given any $f\in\Lambda_s$ and
$\varepsilon\in(0,\infty)$,
there exists a constant $R\in(2,\infty)$, depending only on
$\|f\|_{\Lambda_s}$, $\varepsilon$, $r_1$ and the framework
parameters, such that, for any $j\in\mathbb Z_+$,
$$S_{r_1, j}(s,f,\varepsilon)
\subset \bigcup_{i=j+1}^{j+m}\left[D_{\alpha, r,i}(s,f,
c\varepsilon)\right]_R,$$
where $m$ denotes the positive integer such that $2^{m-1}<
R\leq 2^m$.
\end{proposition}

The proof of this proposition will be given in Section
\ref{sec:9}.

Finally, in the last step, we use finite differences to
establish certain upper estimates of the neighborhood of the
set $D_{\alpha, r,j}(s,f, \varepsilon)$.

\begin{proposition}\label{cor-6-7-0} Assume that $\alpha
r>s+3$. Let $\ell\in\mathbb{N}$ be such that $\frac s 2 <\ell
<\frac {\alpha r -1}2$.
Then, given any $\varepsilon\in(0,\infty)$
and $f\in\Lambda_s$, there exists a positive constant
$R\in(1,\infty)$, depending only on $\|f\|_{\Lambda_s}$,
$\varepsilon$, and the framework parameters, such that, for
any $j\in\mathbb{Z}_+$,
$$D_{\alpha, r,j}(s,f, \varepsilon)\subset
\bigcup_{i=j+1}^{j+8\ell} \left[S_{2\ell, i}(s,f,
c\varepsilon)\right]_R.$$
\end{proposition}

The proof of Proposition \ref{cor-6-7-0} will be given in
Section \ref{sec:9}.

For the moment, we take Propositions
\ref{thm-3adf}-\ref{cor-6-7-0} for granted and proceed with
the proof of Theorem \ref{thm-2-66}.
\begin{proof}[Proof of Theorem \ref{thm-2-66} assuming
Propositions \ref{thm-3adf}-\ref{cor-6-7-0}]
First, we prove the upper bound,
\begin{equation} d \left(f,
\Lambda_X^{s}\right)_{\Lambda_s}
\leq C\left[ \varepsilon_{X}f+
\varepsilon_X^0 f\right], \label{5-4}\end{equation}
where the constant $C$ depends only on the framework
parameters.
 Let $r_1\in\mathbb{N}$ be such that $r_1>s$.
By Lemma \ref{wfawd}, it is enough to show that there exists a
constant $c\in (0, 1)$, depending only on the framework
parameters, such that, for any $\varepsilon\in(0,\infty)$,
\begin{equation}
\left\|\left\{\left[\int_{0}^1 \mathbf{1}_{S_{r_1, j}(s,f,
\varepsilon)}(\cdot,y)\,\frac{dy}{y}\right]^{\theta}
\right\}_{j\in\mathbb{Z}_+}\right\|_{X}
\lesssim
\left\|\left\{\left[\int_{0}^1 \mathbf{1}_{D_{\alpha,
r,j}(s,f,
c\varepsilon)}(\cdot,y)\,\frac{dy}{y}\right]^{\theta}
\right\}_{j\in\mathbb{Z}_+}\right\|_{X}.\label{5-5}
\end{equation}
Here and throughout the proof, the implicit constant in
$\lesssim$ may depend on $f$, the framework parameters, and
other parameters involved the proof, such as $\varepsilon$,
$r_1$, $R$, $m$ and $\delta$.

Indeed, once \eqref{5-5} is proven, then, by \eqref{ef-1} and
\eqref{4-5-1}, we conclude that
 $$ \varepsilon_{X,S,\theta} f\leq \frac 1c \varepsilon_{X},$$
which, using Lemma \ref{wfawd}, will imply the upper estimate
\eqref{5-4}.

To show \eqref{5-5}, we use Proposition \ref{cor-6-7-1} to
obtain
\begin{align*}
\left\|\left\{\left[\int_{0}^1 \mathbf{1}_{S_{r_1, j}(s,f,
\varepsilon)}(\cdot,y)\,\frac{dy}{y}\right]^{\theta}
\right\}_{j\in\mathbb{Z}_+}\right\|_{X}
&\lesssim
\sum_{i=1}^m
\left\|\left\{\left[\int_{0}^{\infty} \mathbf{1}_{(D_{\alpha,
r,j+i}(s,f,
c\varepsilon))_R}(\cdot,y)\,\frac{dy}{y}\right]^{\theta}
\right\}_{j\in\mathbb{Z}_+}\right\|_{X}, \end{align*}
which, using Lemma \ref{a2.15s} and the boundedness of the
left shift, is estimated by
\begin{align*}
&\lesssim \left\|\left\{\left[\int_{0}^{\infty}
\mathbf{1}_{(D_{\alpha, r,j+1}(s,f,
c\varepsilon))_\delta}(\cdot,y)\,\frac{dy}{y}\right]^{\theta}
\right\}_{j\in\mathbb{Z}_+}\right\|_{X}=:\Sigma.
\end{align*}
Applying Proposition \ref{thm-3adf} to the pair of parameters
$(c\varepsilon, 2^{-1} c \varepsilon)$ and
choosing the constant $\delta\in (0, 1)$ small enough yield
\begin{align*}
\Sigma&\lesssim\sum_{k\in\{\pm 1, 0\}}
\left\|\left\{\left[\int_{0}^1 \mathbf{1}_{D_{\alpha,r,
j+1+k}(s,f,
2^{-1}c\varepsilon)}(\cdot,y)\,\frac{dy}{y}\right]^{\theta}
\right\}_{j\in\mathbb{Z}_+}\right\|_{X}\\
&\lesssim
\left\|\left\{\left[\int_{0}^1 \mathbf{1}_{D_{\alpha,r,j}(s,f,
2^{-1}c\varepsilon)}(\cdot,y)\,\frac{dy}{y}\right]^{\theta}
\right\}_{j\in\mathbb{Z}_+}\right\|_{X},
\end{align*}
where the last step used the boundedness of the left shift.
This shows the estimate \eqref{5-5} and hence the upper bound
\eqref{5-4}.

Next, we prove the lower estimate:
\begin{equation} d \left(f,
\Lambda_X^{s}\right)_{\Lambda_s}
\gtrsim \varepsilon_{X}f+
\varepsilon_X^0 f. \label{5-6z}\end{equation}
 Let $\ell\in\mathbb{N}$ be such that $\frac s 2 <\ell <\frac
 {\alpha r -1}2$.
By Lemma \ref{wfawd} applied to $r=2\ell$, it is sufficient to
show that
\begin{equation}\label{5-7}
\left\|\left\{\left[\int_{0}^1 \mathbf{1}_{D_{\alpha,r,j}(s,f,
\varepsilon)}(\cdot,y)\,\frac{dy}{y}\right]^{\theta}
\right\}_{j\in\mathbb{Z}_+}\right\|_{X}\lesssim
\left\|\left\{\left[\int_{0}^1 \mathbf{1}_{S_{2\ell,j}(s,f,
c\varepsilon)}(\cdot,y)\,\frac{dy}{y}\right]^{\theta}
\right\}_{j\in\mathbb{Z}_+}\right\|_{X}.
\end{equation}
Indeed, \eqref{5-7} implies $$ \varepsilon_{X} f\leq \frac 1c
\varepsilon_{X,S,\theta},$$ and hence the lower estimate
\eqref{5-6z}, according to Lemma \ref{wfawd}.

To show \eqref{5-7}, we use Proposition \ref{cor-6-7-0} and
the boundedness of left shift to obtain
\begin{align*}
\left\|\left\{\left[\int_{0}^1 \mathbf{1}_{D_{\alpha,r,j}(s,f,
\varepsilon)}(\cdot,y)\,\frac{dy}{y}\right]^{\theta}
\right\}_{j\in\mathbb{Z}_+}\right\|_{X}&\lesssim
\sum_{i=1}^{8\ell}\left\|\left\{\left[\int_{0}^{\infty}
\mathbf{1}_{(S_{2\ell,j+i}(s,f,
c\varepsilon))_R}(\cdot,y)\,\frac{dy}{y}\right]^{\theta}
\right\}_{j\in\mathbb{Z}_+}\right\|_{X}\\
&\lesssim
\left\|\left\{\left[\int_{0}^{\infty}
\mathbf{1}_{(S_{2\ell,j+1}(s,f,
c\varepsilon))_R}(\cdot,y)\,\frac{dy}{y}\right]^{\theta}
\right\}_{j\in\mathbb{Z}_+}\right\|_{X},
\end{align*}
which, using Lemma \ref{a2.15s} again, is estimated by
\begin{align*}
&\lesssim
\left\|\left\{\left[\int_{0}^{\infty}
\mathbf{1}_{(S_{2\ell,j+1}(s,f,
c\varepsilon))_\delta}(\cdot,y)\,\frac{dy}{y}\right]^{\theta}
\right\}_{j\in\mathbb{Z}_+}\right\|_{X}.
\end{align*}
Then \eqref{5-7} follows from Lemma \ref{thm3adf} with
$\varepsilon_1=\frac 12 c\varepsilon$ and the boundedness of
the right shift. This finishes the proof of Theorem
\ref{thm-2-66}.
\end{proof}

It remains to prove Propositions
\ref{thm-3adf}--\ref{cor-6-7-0},
which will given in the next three sections (Section
\ref{sec:7}--\ref{sec:9}), respectively.

\section{Proof of Proposition \ref{thm-3adf}}\label{sec:7}

This section is devoted to the proof of Proposition
\ref{thm-3adf}.
Let $s,\alpha\in(0,\infty)$, $r\in\mathbb{N}$,
and $\alpha r>s$. We will fix these framework
parameters throughout this section.

The following lemma plays a crucial role in the proof of
Proposition \ref{thm-3adf}.
\begin{lemma}\label{lem-4-5-0} Let $f\in\Lambda_s$ and let
$u_\alpha(x, t):=T_{\alpha, t} f(x)$ for any $x\in
\mathbb{R}^n$ and $t\in(0,\infty)$. Then there exists a
positive constant $C_{n, r, s, \alpha}$, depending only on $r,
s,\alpha$ and $n$, such that,
for any $\mathbf{x}=(x,x_{n+1}), \mathbf{z}=(z, z_{n+1}) \in
\mathbb{R}_+^{n+1}$,
 \begin{equation}
		\left| x_{n+1}^{\alpha r-s } \partial_{n+1}^r u_\alpha
(x,x_{n+1}^\alpha) -z_{n+1}^{\alpha r-s }
\partial_{n+1}^ru_\alpha (z,z_{n+1}^\alpha)\right|\leq C_{n,
r, s, \alpha}\|f\|_{\Lambda_s} \cdot \rho( \mathbf{x},
\mathbf{z}).\label{6-1-0}
	\end{equation}
\end{lemma}

\begin{proof}
Let $F(\mathbf{x}):= x_{n+1}^{\alpha r-s} \partial_{n+1}^r
u_\alpha(x,x_{n+1}^\alpha)$ for any $\mathbf{x}=(x,x_{n+1})\in
\mathbb{R}_+^{n+1}$. Then, by Lemma \ref{cor-3-9}, we have,
for any $\mathbf{x}=(x, x_{n+1})$, $\mathbf{z}=(z,
z_{n+1})\in\mathbb{R}_+^{n+1}$,
\begin{align}
	|F(\mathbf{x})-F(\mathbf{z})|& \leq \sup_{\mathbf{w}\in
\mathbb{R}_+^{n+1}}|\nabla F(\mathbf{w})| w_{n+1}\cdot
\rho(\mathbf{x}, \mathbf{z}),\label{6-6c}
\end{align}
where $\mathbf{w}=(w, w_{n+1})\in\mathbb{R}^n\times
(0,\infty)$.
However, a straightforward calculation shows that, for any
$\mathbf{w}:=(w, w_{n+1})\in\mathbb{R}^n\times (0,\infty)$,
\begin{align*}
	|\nabla F(\mathbf{w})|w_{n+1}&\leq (\alpha r-s)
w_{n+1}^{\alpha r-s} |\partial_{n+1}^r
u_\alpha(w,w_{n+1}^\alpha)|+\alpha w_{n+1}^{\alpha r-s+\alpha}
|\partial_{n+1}^{r+1} u_\alpha(w,w_{n+1}^\alpha)|+\\
	&\ \ +w_{n+1}^{\alpha r-s+1} \sum_{j=1}^n |\partial_j
\partial_{n+1}^r u_\alpha(w,w_{n+1}^\alpha)|,
\end{align*}
which, using Theorem \ref{lem-2-4}, is bounded above by $ C
\|f\|_{\Lambda_s}$.
This combined with \eqref{6-6c} yields the desired estimate
\eqref{6-1-0}, and hence finishes the proof of the Lemma
\ref{lem-4-5-0}.
\end{proof}

\begin{proof}[Proof of Proposition \ref{thm-3adf}] Define, for
any $\lambda\in(0,\infty)$,
 $$D^{+}_{\alpha, r}(s,f,\lambda):=\left\{ (x,t)\in
\mathbb{R}_+^{n+1}:\
 |\partial_{n+1}^r u_\alpha(x, t^\alpha)|>\lambda t^{s-\alpha
  r}\right\}.$$
Assume that $\varepsilon_1\in(0,\infty)$ and
$\varepsilon>\varepsilon_1$.
Let $\delta\in (0, 1/10)$ be a constant such that
$$ C_{n, r, s, \alpha} \|f\|_{\Lambda_s}
\delta<\varepsilon-\varepsilon_1,$$
where $C_{n, r, s, \alpha}$ is the same constant as in
\eqref{6-1-0}.
By \eqref{6-1-0}, for any $\mathbf{x}=(x, x_{n+1})$,
$\mathbf{z}=(z, z_{n+1})\in \mathbb{R}_+^{n+1}$ with
$\rho(\mathbf{x}, \mathbf{z})<\delta$, we have
 \begin{equation*}
		\left| x_{n+1}^{\alpha r-s } \partial_{n+1}^r u_\alpha
(x,x_{n+1}^\alpha) -z_{n+1}^{\alpha r-s }
\partial_{n+1}^ru_\alpha (z,z_{n+1}^\alpha)\right|\leq C_{n,
r, s, \alpha}\|f\|_{\Lambda_s} \rho(\mathbf{x},
\mathbf{z})<\varepsilon-\varepsilon_1,
	\end{equation*}
implying
$$ z_{n+1}^{\alpha r-s } |\partial_{n+1}^ru_\alpha
(z,z_{n+1}^\alpha)|> \left| x_{n+1}^{\alpha r-s }
\partial_{n+1}^r u_\alpha (x,x_{n+1}^\alpha)
\right|-\varepsilon
+\varepsilon_1.$$
This means that $B_\rho(\mathbf{x}, \delta)\subset D_{\alpha,
r}(s, f, \varepsilon_1)$ for any $\mathbf{x}\in D_{\alpha,
r}(s, f, \varepsilon)$.
Thus,
\begin{equation}
(D^{+}_{\alpha,r}(s,
f,\varepsilon))_\delta:=\bigcup_{\mathbf{x}\in D_{\alpha,
r}(s, f,\varepsilon)}
B_\rho(\mathbf{x}, \delta) \subset D_{\alpha, r}(s,
f,\varepsilon_1).\label{5-15b}
\end{equation}

To conclude the proof of Proposition \ref{thm-3adf}, let
$\mathbf{z}:=(z, z_{n+1})\in (D_ {\alpha,r,j}(s,
f,\varepsilon))_\delta$. Then there exists
$\mathbf{x}=(x, x_{n+1})\in D_{\alpha,r,j}(s, f,\varepsilon)$
such that $\rho(\mathbf{x},\mathbf{z})<\delta$.
Furthermore, \eqref{5-15b} implies that $\mathbf{z}\in
D_{\alpha,r}(s, f,\varepsilon_1)$.
However, using \eqref{dafg1}, we obtain
\begin{equation*}
2^{-j-2}<(1+2\delta)^{-1} 2^{-j-1}<(1+2\delta)^{-1} x_{n+1}
\leq z_{n+1}\leq
(1+2\delta) x_{n+1}\leq(1+2\delta)2^{-j}\leq 2^{-j+1}.
\end{equation*}
It follows that
$$\mathbf{z}:=(z, z_{n+1})
\in D_{\alpha,r,j-1}(s, f, \varepsilon_1)\cup
D_{\alpha,r,j}(s, f, \varepsilon_1)\cup
D_{\alpha,r,j+1}(s, f, \varepsilon_1).
$$
Since this holds for any $\mathbf{z}\in (D_{\alpha,r,j}(s,
f,\varepsilon))_\delta$, we infer that
$$ (D_{\alpha,r,j}(s, f,\varepsilon))_\delta\subset
D_{\alpha,r,j-1}(s, f, \varepsilon_1)\cup
D_{\alpha,r,j}(s, f, \varepsilon_1)\cup
D_{\alpha,r,j+1}(s, f, \varepsilon_1).$$
This finishes the proof of Proposition \ref{thm-3adf}.
\end{proof}

\section{Proof of Proposition \ref{cor-6-7-1}}\label{sec:8}

This section is devoted to the proof of Proposition
\ref{cor-6-7-1}.
Throughout this section, we fix the framework parameters
$\alpha, s\in (0,\infty)$ and $r\in\mathbb{N}$ such that
$\alpha r>s$.
For each given function $f\in\Lambda_s$, we define $u_{\alpha,
f}(x, t):=T_{\alpha, t} f(x)$ for any $x\in\mathbb{R}^n$ and
$t\in(0,\infty)$.
 To simplify symbol, we will use $u$ to denote $u_{\alpha, f}$
 whenever no confusion is possible.
This symbol will be used consistently throughout this section.

Proposition \ref{cor-6-7-1} follows directly from the lemma
below, which provides an estimate of finite differences in
terms of the derivatives $\partial_{n+1}^{r} u$.

\begin{lemma}\label{thm-691} Let
$r_1\in\mathbb{N}$ be such that
 $r_1>\alpha r+1$.
Then there exists a constant $\delta\in (0, 1)$, depending
only on $\alpha$ and $s$,
such that, for any $f\in\Lambda_s $, $(x,y)
\in\mathbb{R}_+^{n+1}$, and
$R>2(r_1+2)^{\frac{\alpha+1}\alpha}$,
\begin{equation*}
y^{-s}\Delta_{r_1} f (x,y)\leq
	C \left[\sup_{|x'-x|<Ry} \
\sup_{R^{-1} y\leq t \leq 2^{-1}y}
t^{\alpha r-s}|\partial_{n+1} ^{r} u(x', t^\alpha)|+
 \|f\|_{\Lambda_s } R^{-\delta}\right],
\end{equation*}
where the constant $C$ depends only on $r_1$ and the framework
parameters.
\end{lemma}

For the moment, we take Lemma \ref{thm-691} for granted and
 proceed with the proof of
Proposition \ref{cor-6-7-1}.

\begin{proof} [Proof of Proposition \ref{cor-6-7-1} (assuming
Lemma \ref{thm-691})]
Let $(x, y)\in S_{r_1,j}(s,f,\varepsilon)$. By
\eqref{dfmawlmfp-1}, we have $2^{-j-1}<y\leq 2^{-j}$ and
$y^{-s}\Delta_{r_1} f (x,y)>\varepsilon$.
Furthermore, by Lemma \ref{thm-691}, for any
$R_1>2(r_1+2)^{\frac{\alpha+1}\alpha}$,
\begin{equation*}
\varepsilon <
	C \left[\sup_{|x'-x|<R_1y}\
	\sup_{	R_1^{-1} y\leq t \leq 2^{-1} y}
t^{\alpha r-s}|\partial_{n+1} ^{r} u(x', t^\alpha)|+
\|f\|_{\Lambda_s } R_1^{-\delta}\right],
\end{equation*}
where $C\in(0,\infty)$ is independent of $f$, $R_1$, and $(x,
y)\in\mathbb{R}_+^{n+1}$.
 Now, choosing the constant $R_1$ sufficiently large so that $
 C R_1^{-\delta}\|f\|_{\Lambda_s}\leq \frac \varepsilon2$
yields
 $$ \sup_{|x'-x|<R_1y} \ \sup_{R_1^{-1} y\leq t \leq 2^{-1}y}
t^{\alpha r-s}|\partial_{n+1} ^{r} u(x', t^\alpha)|>\frac
1{2C}\varepsilon =c\varepsilon.$$
 Thus, there must exist $(x', t)\in\mathbb{R}_+^{n+1}$ such
 that
 \begin{align*} R_1^{-1} y<t\leq 2^{-1} y,\ \
 |x'-x|\leq R_1 y,\end{align*}
 and
 \begin{align}t^{\alpha r-s}|\partial_{n+1} ^{r} u(x',
 t^\alpha)|>c\varepsilon. \label{6-9-eq} \end{align}
 In particular, by Lemma \ref{lem-3-2-0}, this implies
 \begin{equation}\label{5-23}(x,y)\in B_\rho( (x', t),
 R),\end{equation}
 with $R=C R_1$ for some constant $C>1$ depending only on the
 framework parameters.

 Next, let $m$ be a positive integer such that $2^{m-1}< R
 \leq 2^m$. Then
 $ 2^{-m-j-1}<R^{-1} y<t\leq 2^{-1}y \leq 2^{-1-j}.$
 This together with \eqref{6-9-eq} implies
 \begin{equation}\label{5-22}
 (x',t)\in \bigcup_{i=j+1}^{j+m}D_{\alpha, r, i}(s, f,
 c\varepsilon).\end{equation}
 Finally, combining \eqref{5-22} with \eqref{5-23}, we obtain
 $ (x,y)\in \bigcup_{i=j+1}^{j+m} (D_{\alpha, r, i}(s, f,
 c\varepsilon))_R.$
 Since this holds for any $(x, y)\in S_{r_1,j} (s, f,
 \varepsilon)$, we complete the proof of Proposition
 \ref{cor-6-7-1}.
\end{proof}

It remains to prove Lemma \ref{thm-691}. We need two
additional lemmas.
\begin{lemma}\label{lem-8-2b}
Let $r_1, r_2\in\mathbb{N}$ be such that $r_1\ge r_2>\frac
s\alpha $, and let $f\in\Lambda_s$. Then, for any $R>1$,
$\tau\in (0, 1)$,
and
$(x,y)\in\mathbb{R}_+^{n+1}$,
\begin{equation}\label{2-4}
 y^{ r_1 - \frac {s}\alpha } \left| \partial_{n+1}^{r_1}
 u(x,y)
 \right|\leq C \left[ \sup_{ |x-x'| \leq R y^\frac 1{\alpha}}
 y^{ {r_2}- \frac s\alpha } |\partial_{n+1} ^{r_2} u(x',\tau
 y)|+ R^{-\alpha}\|f\|_{\Lambda_s}\right],
\end{equation}
where the constant $C$ depends only on $r_1, r_2, \tau$, and
the framework parameters.
\end{lemma}

\begin{proof} Using \eqref{5-2} and the semigroup property of
$T_{\alpha, t}$, we obtain, for any $ x\in\mathbb{R}^n$ and
$t_1, t_2\in(0,\infty)$,
$$ u(x, t_1+t_2) =T_{\alpha, t_2} T_{\alpha, t_1} f(x)
=\int_{\mathbb{R}^n} u(x-z, t_1) K_\alpha (z, t_2)\, dz.$$
This implies that
\begin{align*}
 \partial_{n+1}^{r_1} u(x,t_1+t_2) &=\int_{\mathbb{R}^n}
 \partial_{n+1}^{r_2} u(x-z, t_1)\cdot
 \partial_{n+1}^{r_1-{r_2}} K_\alpha (z, t_2)\, dz.
\end{align*}
Setting $t_1=\tau y$ and $t_2= (1-\tau) y$, we obtain
\begin{align*}
|\partial_{n+1}^{r_1} u(x,y)|&=
\left|\int_{\mathbb{R}^n}\left[ \partial_{n+1}^{r_2} u(x-z,
\tau y) \right] \partial_{n+1}^{r_1-{r_2}} K_\alpha \left(z,
(1-\tau) y\right)\, dz \right|.\end{align*}
This, combined with \eqref{5-5c} and \eqref{5-6c}, yields
\begin{align*}
|\partial_{n+1}^{r_1} u(x,y)|&\leq C y^
{{r_2}-r_1-\frac n\alpha } \int_{\mathbb{R}^n}
\left|\partial_{n+1}^{r_2} u(x-z, \tau y) \right| H_{0}
\left([(1-\tau)y]^{-\frac 1\alpha} z\right)\, dz.
\end{align*}
Given any constant $R>1$, we break this last integral into two
parts $\int_{|z|\leq R y^\frac 1{\alpha}}+\int_{|z|> R y^\frac
1{\alpha}}$. We then deduce
$$|\partial_{n+1}^{r_1} u(x,y)|\leq C[ I_1(x,y) +
I_2(x,y)],$$
where
\begin{align*} I_1(x,y):&= y^ {{r_2}-r_1-\frac n\alpha }
\int_{|z|\leq R y^\frac 1{\alpha}} \left|\partial_{n+1}^{r_2}
 u(x-z, \tau y) \right| H_{0}\left([(1-\tau) y]^{-\frac
 1\alpha} z\right)\, dz
\end{align*}
and
\begin{align*}
 I_2(x,y):&= y^ {{r_2}-r_1-\frac n\alpha } \int_{|z|> R
 y^\frac 1{\alpha}} \left|\partial_{n+1}^{r_2} u(x-z, \tau y)
 \right| H_{0}\left([(1-\tau) y]^{-\frac 1\alpha} z\right)\, dz.
\end{align*}
The term $I_1(x,y)$ can be estimated as follows:
\begin{align*}
	I_1(x,y)&\lesssim \left[\sup_{|x-x'|\leq Ry^\frac
1{\alpha}} \ y^{{r_2}-\frac s\alpha}|\partial_{n+1}^{r_2}
u(x', \tau y)|\right] \cdot y^{-r_1+\frac s\alpha }
\|H_{0}\|_1 \\
	&\lesssim y^{-r_1+\frac s\alpha } \left[\sup_{|x-x'|\leq
 Ry^\frac 1{\alpha}} y^{{r_2}-\frac
 s\alpha}|\partial_{n+1}^{r_2} u(x', \tau y)|\right] .
\end{align*}
To estimate the term $I_2(x,y)$, we use Theorem \ref{dsfad}.
Recalling that
$$ \partial_{n+1}^{r_2} u(\cdot, y)=(-\Delta)^{\alpha {r_2}/2}
 T_{\alpha, y} f,$$
we deduce
\begin{align*}
	I_2(x,y) &\lesssim \|f\|_{\Lambda_s} y^{-r_1-\frac n\alpha
+\frac s\alpha } \int_{|z|> R y^\frac 1{\alpha}} H_{0}
\left([(1-\tau) y]^{-\frac 1\alpha} z\right)\, dz\\
	&\lesssim y^{-r_1 +\frac s\alpha }
R^{-\alpha}\|f\|_{\Lambda_s}.
\end{align*}
A combination of the above estimates then yields the estimate
\eqref{2-4} for all $r_1\ge r_2>\frac s\alpha$, which
completes the proof
of Lemma \ref{lem-8-2b}.
\end{proof}

\begin{lemma}\label{lem-8-3b} Let $i\in\mathbb{Z}_+$ and
$\beta\in \mathbb{Z}_+^n$ be such that $0\leq i<r$ and
$|\beta|>\alpha (r-i)+1$. Let $\tau\in (0, 1)$.
Then there exists a positive constant $C$, depending only on
$\tau$, $\beta$, and the framework parameters, such that, for
any $f\in\Lambda_s$,
$(x, y)\in\mathbb{R}_+^{n+1}$, and $R>2$,
\begin{align*}
 y^{\alpha i +|\beta|-s} |\partial_x^{\beta} \partial_{n+1}^i
 u(x, y^\alpha)|\leq C \left[\sup_{|x-x'|\leq Ry} y^{\alpha
 r-s} |\partial_{n+1}^r u(x', (\tau y)^\alpha)|
 +R^{-1}\|f\|_{\Lambda_s}\right].
\end{align*}
\end{lemma}

\begin{proof} Since
$\partial_{n+1}^i u(x, y^\alpha) =(-\Delta)^{\frac {\alpha i}
2} T_{\alpha, y^\alpha} f(x),$
it follows that, for any $\xi\in \mathbb{R}^n$,
\begin{align*}
\mathcal{F} \left( \partial_x^\beta \partial_{n+1}^i u(\cdot,
y^\alpha)\right) (\xi) =C \xi^\beta |\xi|^{\alpha i} e^{-|2\pi
y\xi|^\alpha} \widehat f(\xi).
\end{align*}
Let $\tau_2\in (0, 1)$ be such that
$\tau_2^\alpha=1-\tau^\alpha $. We then write
\begin{align}\mathcal{F} \left( \partial_x^\beta
\partial_{n+1}^i u(\cdot, y^\alpha)\right) (\xi)&=C \frac
{\xi^\beta}{|\xi|^{\alpha r-\alpha i}}e^{-|2\pi \tau_2
y\xi|^\alpha}|\xi|^{\alpha r} e^{-|2\pi \tau y
\xi|^\alpha}\widehat f(\xi)\notag\\
&= y^{-|\beta| +\alpha r -\alpha i}M_i(y\xi) \mathcal{F}
\left( \partial_{n+1}^r u(\cdot, (\tau y)^\alpha)\right)
(\xi),\label{8-6-b}
\end{align}
where $C$ is a constant and
$ M_i(\xi)=C \frac {\xi^\beta} {|\xi|^{\alpha r-\alpha i}}
e^{-|2\pi \tau_2 \xi|^\alpha}.$
Since $|\beta|> \alpha r -\alpha i +1$, it is easily seen that
$$ \|\nabla^j M_i\|_{L^1} <\infty,\ \ j=0,1,\ldots, n+1.$$
Thus, there exists a function $G_i\in L^1$ such that $\widehat
{G_i} =M_i$ and
$ |G_i(x)|\lesssim (1+|x|)^{-n-1}$ for any $x\in\mathbb{R}^n.$
It then follows from
\eqref{8-6-b} that
\begin{align*}
\partial_x^\beta \partial_{n+1}^i u(x, y^\alpha) =y^{-|\beta|
+\alpha r-\alpha i} y^{-n} \int_{\mathbb{R}^n}
\partial_{n+1}^r u(x-z, (\tau y)^\alpha) G_i ( y^{-1} z) \,
dz.
\end{align*}
Breaking this last integral into two parts $\int_{|z|\leq
Ry}+\int_{|z|>Ry}$, we conclude that
\begin{align*}
|\partial_x^\beta \partial_{n+1}^i u(x, y^\alpha)|&\lesssim
y^{-|\beta|+\alpha r -\alpha i} \sup_{|z|\leq Ry}
|\partial_{n+1}^r u(x-z, (\tau y)^\alpha)|+ R^{-1}
y^{-|\beta|+\alpha r -\alpha i} y^{-\alpha r +s}
\|f\|_{\Lambda_s}\\
& \lesssim y^{-|\beta|-\alpha i+s} \left[ \sup_{|x-x'|\leq Ry}
y^{\alpha r-s} |\partial_{n+1}^r u(x', (\tau
y)^\alpha)|+R^{-1} \|f\|_{\Lambda_s}\right],
\end{align*}
which completes
the proof
of Lemma \ref{lem-8-3b}.
\end{proof}

\begin{proof} [Proof of Lemma \ref{thm-691} ] Applying
Taylor's theorem to the function $t\mapsto u(x, t)$ at the
point $t=y^\alpha\in(0,\infty)$, we obtain
\begin{align}f(x)= u(x, 0) =\sum_{i=0}^{r-1} \partial_{n+1}^i
u (x, y^\alpha ) \frac { (-1)^i y^{\alpha i} } {i!} + \frac
{(-1)^{r}} {(r-1)!} \int_0^{y^\alpha} t^{r-1} \partial_{n+1}
^{r} u(x, t) \, dt.\label{8-7-b} \end{align}
Let $h\in\mathbb{R}^n$ be such that $|h|=y$.
Taking $r_1$-th order symmetric difference with respect to the
variable $x$ in the direction of $h\in\mathbb{R}^n$ on both
sides of \eqref{8-7-b} yields
\begin{equation*}
	\Delta_h^{r_1} f(x) =\sum_{i=0}^{r-1}\frac {(-1)^{i}}{i!}
y^{\alpha i} \Delta_{h,x}^{r_1} \partial_{n+1}^i u(x, y^\alpha
)+\frac {(-1)^{r}}{(r-1)!} \int_0^{y^\alpha} t^{r-1}
\Delta_{h,x}^{r_1} \partial_{n+1}^{r} u(x,t) \, dt,
\end{equation*}
where $\Delta_h^{r_1}$ denotes the symmetric difference of
order $r$ given in \eqref{diff-eq} and the symbol
$\Delta_{h,x}^{r_1}$ means that the difference operator
$\Delta_h^{r_1}$ is acting on the variable $x$.
It follows from \eqref{diff-eq} that \begin{align}
|\Delta_h^{r_1} f(x)| &\leq C_{r_1, r}\left(I_1+
I_2\right),\label{6-11}
\end{align}
where
\begin{align*}
I_1:&=\max_{0\leq i\leq r-1} y^{\alpha i +r_1} \sup_{x'\in
\overline{[x-r_1h, x+r_1h]}} |\nabla_x^{r_1} \partial_{n+1}^i
u(x', y^\alpha)|,\\
I_2:&=\int_0^{y^\alpha} t^{r-1}|\Delta_{h,x}^{r_1}
\partial_{n+1}^{r} u(x,t)| \, dt,
\end{align*}
and $\overline{[x, z]}:=\{ t z+(1-t)x:\ \ t\in [0, 1]\}$ for
any $x, z\in\mathbb{R}^n$.

For the first term $I_1$, we use Lemma \ref{lem-8-3b} to
obtain,
for any $R_1>2r_1$,
\begin{align}
I_1 \lesssim y^{s}\left[\sup_{|z-x|<R_1y} \
y^{\alpha r-s }|\partial_{n+1} ^{r} u(z, (\tau y)^\alpha)|+
R_1^{-1}\|f\|_{\Lambda_s} \right]. \label{6-12}\end{align}

The second term $I_2$ can be estimated as follows: for any
constant $R_1>2r_1$,
\begin{align}\label{6-13}
I_2&\leq C_{r}
	 \int_0^{y^\alpha} t^{\frac s\alpha-1} \sup_{ x'\in
[x-r_1h, x+r_1h]} t^{r-\frac s\alpha} |\partial_{n+1}^{r}
u(x', t)| \, dt\notag\\
	&\lesssim \int_0^{y^\alpha/R_1} t^{-1+\frac s\alpha}\, dt
\|f\|_{\Lambda_s} + \sup_{|x-x'|\leq r_1y}\ 			
\sup_{R_1^{-1} y^\alpha \leq v\leq y^\alpha} v^{r-\frac
s\alpha} |\partial_{n+1}^{r} u(x', v)|\int_{R_1^{-1} y^\alpha
}^{y^\alpha} t^{\frac s\alpha-1}\, dt\notag\\
		&\lesssim y^s R_1^{-s/\alpha}\|f\|_{\Lambda_s} +y^s
\sup_{|x-x'|\leq r_1y}\	\sup_{R_1^{-1/\alpha} y \leq t\leq y}
t^{\alpha r-s} |\partial_{n+1}^{r} u(x', t^\alpha)|\notag\\
&\lesssim y^s R_1^{-s/\alpha}\|f\|_{\Lambda_s} +y^s
\sup_{|x-x'|\leq R_1y}\	\sup_{ R_1^{-1/\alpha} y \leq t\leq y}
t^{\alpha r-s} |\partial_{n+1}^{r} u(x', (\tau
t)^\alpha)|,
\end{align}
where we used \eqref{diff-eq} in the first step, \eqref{5-18}
in the second step and Lemma \ref{lem-8-2b} in the last step.

Finally, combining \eqref{6-12}, \eqref{6-13} with
\eqref{6-11}, and setting $\tau=\frac 12$, we obtain
\begin{align*}
y^{-s}\Delta_{r_1} f(x,y)&\lesssim
\sup_{|x'-x|<R_1y}\sup_{R_1^{-1/\alpha} y \leq t \leq y} \
		 t^{\alpha r-s }|\partial_{n+1} ^{r} u(x', (2^{-1}
t)^\alpha)|
+( R_1^{-1}+R_1^{-s/\alpha})\|f\|_{\Lambda_s} \\
&\lesssim \sup_{|x'-x|<Ry} \
		\sup_{R^{-1} y\leq t \leq 2^{-1} y} t^{\alpha r-s
}|\partial_{n+1} ^{r} u(x', t^\alpha)|+
R^{-\delta}\|f\|_{\Lambda_s},
\end{align*}
where $R=2 R_1^{1+\frac 1\alpha}$ and $\delta:=\frac
\alpha{\alpha+1} \min\{ 1, \frac s\alpha\}$. This finishes the
proof of Lemma \ref{thm-691}.
\end{proof}

\section{Proof of Proposition \ref{cor-6-7-0}}\label{sec:9}

This section is devoted to the proof of Proposition
\ref{cor-6-7-0}.
We fix the framework parameters $\alpha, s\in (0,\infty)$ and
$r\in\mathbb{N}$ such that $\alpha r>s+3$.
For any given function $f\in\Lambda_s$, we will use $u(x,y)$
to denote $u_{\alpha, f}(x, t):=T_{\alpha, t} f(x)$ whenever
no confusion arises.

Proposition \ref{cor-6-7-0} follows directly from the lemma
below, which provides an estimate of the quantity
$|\partial_{n+1}^r u(x, t^\alpha )|$ in terms of
finite differences.

\begin{lemma}\label{lem-6-10} Assume that $\alpha r>s+3$. Let
$\ell\in\mathbb{N}$ be such that $\frac s 2 <\ell <\frac
{\alpha r -1}2$.
Then, for any $f\in\Lambda_s $, $(x,t) \in\mathbb{R}_+^{n+1}$,
and $R>8\ell$,
\begin{align}t^{\alpha r-s} |\partial_{n+1}^r u(x, t^\alpha
)|\leq C \left[ \sup_{\frac 1{8\ell} t\leq y \leq \frac 12t} \
\sup_{|x'-x|\leq R t} y^{-s} \Delta_{2\ell} f (x', y)+
\|f\|_{\Lambda_s} R^{-1}\right],\label{6-26}\end{align}
where the positive constant $C$ is independent of $f$, $R$,
and $(x, t)$.
\end{lemma}

For the moment, we take Lemma \ref{lem-6-10} for granted and
proceed with the proof of
Proposition \ref{cor-6-7-0}.

\begin{proof} [Proof of Proposition \ref{cor-6-7-0} (assuming
Lemma \ref{lem-6-10})]
Let $(x, t)\in D_{\alpha, r, j}(s,f,\varepsilon)$. Then
$2^{-j-1}<t\leq 2^{-j}$ and, by Lemma \ref{lem-6-10},
$$\varepsilon
< t^{ \alpha r -s} \left|\partial_{n+1}^{r} u (x,t^\alpha)
\right|\leq C_1 \left[ \sup_{|x'-x|<R_1t,\
		c_\ell t\leq y \leq 2^{-1} t}\ y^{-s}\Delta_{2\ell}
f(x', y)+ \|f\|_{\Lambda_s} {R_1}^{-1}\right],$$
where $c_\ell =(8\ell)^{-1}$ and $R_1$ is a sufficiently large
constant such that $$C_1\|f\|_{\Lambda_s} R_1^{-1}<\frac
\varepsilon2.$$
Thus, there exists $(x', y)\in \mathbb{R}_+^{n+1}$ such that
$c_\ell t \leq y\leq 2^{-1} t$, $|x-x'|\leq R_1 t$, and
\begin{align}\label{6-37a} y^{-s} \Delta_{2\ell} f (x',
y)>c\varepsilon,\ \ \ \text{where $c=\frac
1{2C_1}$.}\end{align}
In particular, by Lemma \ref{lem-3-2-0}, this implies
$(x, t)\in B_\rho\left( (x', y), R\right)$ for some constant
$R=CR_1>R_1$.

Let $m=m_\ell$ denote the positive integer such that
$2^{-m}<(8\ell)^{-1}\leq 2^{-m+1}$. Then
\begin{align*}2^{-m-j-1} <c_\ell t <y\leq 2^{-1}
t\leq 2^{-j-1},\end{align*}
and hence, by \eqref{6-37a}, we have
\begin{align*}(x', y) \in \bigcup_{i=j+1}^{j+m} S_{2\ell, i}
(s, f, c\varepsilon).\end{align*}
Since $(x, t)\in B_\rho\left( (x', y), R\right)$, it follows
that
$$ (x, t) \in \bigcup_{i=j+1}^{j+m} \left[S_{2\ell, i} (s, f,
c\varepsilon)\right]_R.$$
From this holds for an arbitrary $(x, t)\in D_{\alpha, r,
j}(s,f,\varepsilon)$, we deduce that
$$ D_{\alpha, r, j}(s,f,\varepsilon)\subset
\bigcup_{i=j+1}^{j+m} \left[S_{2\ell, i} (s, f,
c\varepsilon)\right]_R,$$
which is as desired,
This finishes
the proof of Proposition \ref{cor-6-7-0} because $m\leq
8\ell$.
\end{proof}

It remains to prove Lemma \ref{lem-6-10}, which is technically
more involved.
The crucial tool in the proof is the higher-order ball average
operators, along with some previously established properties
of these operators, which we now recall.

Given $\ell\in\mathbb{N}$
and $t\in(0,\infty)$, the $2\ell$-th order ball average
operator $B_{\ell,t}$ is defined by setting, for any $f\in
L_{\text{loc}}^1$ and $x\in\mathbb{R}^n$,
\begin{equation*}
	B_{\ell,t} (f)(x): =\frac{-2}{\binom{2\ell}{\ell} }
	\sum_{j=1}^\ell (-1)^j \binom {2\ell}{\ell-j} B_{jt} f(x),
\end{equation*}
where
$$ B_t f(x):=\frac{1}{|B(\mathbf{0},1)|}\int_{B(\mathbf{0},1)}
f(x+t z)\, dz.$$
We will use the following known result on the approximation by
ball average operators.

\begin{lemma} (\textnormal{\cite[Lemma 5.11] {dsy2}} and
\textnormal{\cite[Theorem 1]{dw}}\label{lem-4-2}) Given any
$\ell\in\mathbb{N}$, there exists a constant $C=C_{\ell
,n}\in(0,\infty)$, depending only on $\ell$ and $n$, such
that,
 for any $t\in(0,\infty)$,
$f\in \mathcal{C} \cap L^\infty$, and $x\in\mathbb{R}^{n}$,
\begin{align}\label{1-3-0}
		 |f(x) -B_{\ell, t} f(x)|&\leq C
\sup_{ \frac 1{8\ell}t\leq y \leq 2^{-1}t} \ \sup_{|x'-x|\leq
4\ell t} \Delta_{2\ell} f (x', y).
\end{align}
Furthermore, if $\ell>s/2$, then there exists a positive
constant $C_{\ell, s, n}$, depending only on $\ell, s$ and
$n$,
such that, for any $t\in(0,\infty)$ and $f\in \Lambda_s$,
\begin{align}\label{6-20} \|f-B_{\ell, t} f \|_{\infty} \leq
C_{\ell, s, n} t^s \|f\|_{\Lambda_s}. \end{align}
\end{lemma}

The estimate \eqref{1-3-0} was established in our recent
article \cite[Lemma 5.11] {dsy2}, while \eqref{6-20} follows
directly from \cite[Theorem 1]{dw} and \eqref{zuihou}.

We also need some known properties for the Fourier transform
of $B_{\ell, t} f$.
A straightforward computation (see \cite[Lemma 2]{dw}) shows
that, for any
$\xi\in\mathbb{R}^n$,
\begin{equation}\label{4:edu}
	\mathcal{F}(B_{\ell, t} f) (\xi) = m_\ell(2\pi t\xi)
\mathcal{F}f(\xi),
\end{equation}
where
\begin{align*}
	m_{\ell}(\xi) &:= \frac{ n \Gamma (\frac n2)}{\sqrt{\pi}
\Gamma(\frac {n+1}2)} \cdot \frac {-2}{ \binom{2\ell}{\ell}}
\sum_{j=1}^\ell
	(-1)^j \binom{2\ell}{\ell-j} \int_0^1 \cos (j u|\xi|)
(1-u^2) ^{\frac {n-1}2} \, du.
\end{align*}
The following lemma collects some useful estimates of the
function $m_\ell$, which can be found in \cite{dw}.

\begin{lemma}\textnormal{\cite[(9), Lemma 3, (16),\ (23)--(25)]{dw}}
\label{lem-4-3} Let $\ell\in\mathbb{N}$. Then the following
statements hold.
	\begin{enumerate}[\rm (i)]
\item For any $\xi\in\mathbb{R}^{n}$,
\begin{equation*}
1-m_\ell(\xi)\sim \min\left\{1, |\xi|^{2\ell}\right\}.
\end{equation*}
		\item For any $j\in\mathbb{Z}_+$,
there exists a
positive constant $C=C_{j,\ell}$ such that, for any
$\xi\in\mathbb{R}^n$,
\begin{equation*}
			|\nabla^j m_\ell(\xi)|\leq C \left ( \frac
1{1+|\xi|}\right)^{\frac {n+1}2}.
\end{equation*}
\item There exists a constant
$\gamma\in (0, 1)$ such that
 $0<m_\ell (\xi)\leq \gamma<1$ whenever $|\xi|\ge 1$.
		\item $\frac{1-m_\ell(\cdot)}{|\cdot|^{2\ell}}$
extends to a positive $C^\infty$-function on $\mathbb{R}^n$.
	\end{enumerate}	
\end{lemma}

\begin{proof}[Proof of Lemma \ref{lem-6-10}]
Throughout the proof below, the letter $C$ denotes a general
positive constant depending only on $\ell$ and the framework
parameters.
Note that $\alpha r>2\ell+1$.

First, by definition, we have, for any $(\xi,t)
\in\mathbb{R}_+^{n+1}$,
$$ t^{\alpha r} \mathcal{F}\left[ \partial_{n+1}^r u(\cdot, t^\alpha
)\right](\xi)= c_{n, r} |t\xi|^{\alpha r} e^{- |2\pi
t\xi|^\alpha}\widehat f(\xi)= \phi_{\alpha, r}(t\xi) \widehat
f(\xi), $$
	where $
\phi_{\alpha, r} (\xi) :=c_{n,r}|\xi|^{\alpha r}
e^{-|2\pi\xi|^\alpha}.$
Here and throughout the proof, the Fourier transform is
understood in a distributional sense.
For simplicity, we let $F_{\ell, t} :=f-B_{\ell, t} f$. Then,
using \eqref{4:edu}, we may write, for any $(\xi,t)
\in\mathbb{R}_+^{n+1}$,
\begin{align}t^{\alpha r} \mathcal{F}\left[ \partial_{n+1}^r
u(\cdot, t^\alpha )\right] (\xi)
	=\frac {\phi_{\alpha, r} (t\xi) } {1-m_\ell (2\pi t\xi)}
\mathcal{F}(f-B_{\ell,t} f ) (\xi)= \mu_{\alpha, r}(t\xi)
\widehat{F_{\ell, t}} (\xi),\label{6-27}
\end{align}
where
\begin{align}
	\mu_{\alpha, r}(\xi):=\frac {\phi_{\alpha, r} (\xi) }
{1-m_\ell (2\pi \xi)}.\label{6-28}
\end{align}

Next, we claim that there exists a function $G_{\alpha, r}\in
L^1$ satisfying that
$\widehat {G_{\alpha, r}}=\mu_{\alpha, r}$ and
\begin{align}|G_{\alpha, r}(x)|\lesssim (1+|x|)^{-n-1},\ \
\forall x\,\in \mathbb{R}^n. \label{6-29}\end{align}

Once this claim is proven, we define a family of uniformly
bounded convolution operators $\{U_t\}_{t\in (0,\infty)}$ on
$L^\infty$ by setting, for any $g\in L^\infty$ and
$x\in\mathbb{R}^n$,
$$ U_t g(x):=t^{-n}\int_{\mathbb{R}^n} g(x-z) G_{\alpha, r}
(t^{-1} z) \, dz.$$
By \eqref{6-27}, we then obtain the following integral
representation for $t^{\alpha r} \partial_{n+1}^r u(x,
t^\alpha )$:
\begin{align}t^{\alpha r} \partial_{n+1}^r u(x, t^\alpha
):=U_{t} F_{\ell,t}(x)=\frac1{t^{n}}\int_{\mathbb{R}^n}
F_{\ell, t} (x-z) G_{\alpha, r}\left(\frac zt\right) \,
dz.\label{6-30}\end{align}

To show the claim, we define
\begin{align}\label{6-31}
G_{\alpha, r}(x) := \int_{\mathbb{R}^n} \mu_{\alpha, r}(\xi)
e^{2\pi i x\cdot\xi}\, d\xi,\ \ x\in\mathbb{R}^n
\end{align}
and
we use Lemma \ref{lem-4-3}(i) and the
knowledge $\alpha r>2\ell+1$.
Since $\mu_{\alpha,r}$ is a radial, integrable function on
$\mathbb{R}^n$, $G_{\alpha, r}=\widehat{\mu_{\alpha, r}}$ is a
uniformly bounded, continuous, radial function on
$\mathbb{R}^n$ with distributional Fourier transform
$\mu_{\alpha, r}$. Using \eqref{6-28}, we write, for any
$\xi\in\mathbb{R}^{n}$,
$$ \mu_{\alpha, r}(\xi)=c_{n,r,\ell}|\xi| ^{\alpha r-2\ell}
e^{-|2\pi \xi|^\alpha} \frac {|2\pi \xi|^{2\ell}}
{1-m_\ell(2\pi \xi)}.$$
Lemma \ref{lem-4-3}(i) implies that,
for any given $i\in\mathbb{Z_+}$ and any $\xi\in\mathbb{R}^n$,
$$ \left|\nabla^i \left(\frac {|2\pi \xi|^{2\ell}}
{1-m_\ell(2\pi \xi)}\right)\right|\leq C
(1+|\xi|)^{2\ell-i}.$$
Notice also that, for any given $i\in\mathbb{Z_+}$ and any
$\xi\in\mathbb{R}^n$,
$$ \left|\nabla^i e^{-|2\pi \xi|^\alpha}\right|\leq C_{\alpha,
i} e^{-|\xi|^\alpha
}\left[1+|\xi|^{\alpha-i}\mathbf{1}_{B(\mathbf{0},1)}(\xi)\right].
$$
Thus, by the Leibniz rule, we obtain, for any given
$i\in\mathbb{Z_+}$ and any $\xi\in\mathbb{R}^n$,
$$ |\nabla^i \mu_{\alpha, r} (\xi)| \leq C_{\alpha,
i}\left[|\xi|^{\alpha r-2\ell-i}
\mathbf{1}_{B(\mathbf{0},1)}(\xi)+1\right] e^{-|\xi|^\alpha
}.$$
Since $\alpha r > 2\ell+1$, this implies that $ |\nabla^i
\mu_{\alpha, r} (\xi)|\in L^1 $ for $i=0,\ldots, n+1$. Thus,
by \eqref{6-31}, it follows that, for any
$\beta\in\mathbb{Z}_+^n$ with $|\beta|\leq n+1$,
\begin{equation*}
 x^\beta G_{\alpha, r}(x) :=c_{n,\beta} \int_{\mathbb{R}^n}
 \left(\partial^\beta \mu_{\alpha, r}(\xi) \right) e^{2\pi i
 x\cdot\xi}\, d\xi,\ \ x\in\mathbb{R}^n,
\end{equation*}
which implies the desired estimate \eqref{6-29} and hence
completes the proof of the claim.

Finally, we prove \eqref{6-26}. Indeed, using \eqref{6-29} and
\eqref{6-30}, we obtain, for any $R>1$ and $(x,
t)\in\mathbb{R}_+^{n+1}$,
\begin{align*}
t^{\alpha r-s} |\partial_{n+1}^r u(x, t^\alpha )|&
	\lesssim t^{-n-s}\int_{|z|\lesssim R t} |F_{\ell, t}
(x-z)|\left| G_{\alpha,r}\left(\frac zt\right)\right| \, dz\\
&\quad+ t^{-s} \|F_{\ell, t} \|_\infty
	t^{-n}\int_{|z|> R t}\left | G_{\alpha,r}\left(\frac
zt\right)\right| \, dz\\
&\lesssim t^{-s} \sup_{|x-x'|\leq Rt} |F_{\ell, t} (x')|+
t^{-s} \|F_{\ell, t} \|_\infty R^{-1},
\end{align*}
which, using Lemma \ref{lem-4-2}, is estimated by
\begin{align*}	
&\lesssim \sup_{|x-x'|\leq Rt} \ \sup_{\frac t{8\ell}\leq y
\leq 2^{-1}t} \sup_{\|x'-z\|\leq 4\ell t} y^{-s}
\Delta_{2\ell} f (z, y)+ \|f\|_{\Lambda_s}R^{-1}\\
	 &\lesssim \sup_{\frac t{8\ell}\leq y \leq \frac 12 t}\
\sup_{|x-x'|\leq (R+4\ell) t} y^{-s} \Delta_{2\ell} f (x', y)+
\|f\|_{\Lambda_s}R^{-1}.
\end{align*}
This proves \eqref{6-26} for $R>8\ell$,
which completes
the proof of Lemma \ref{lem-6-10}.
\end{proof}

\section{Proof of Theorem \ref{thm-7-11} (heat semigroup
characterization
of \\ $d(f,\Lambda_X^s)_{\Lambda_s}$ under Assumption
\ref{pplp}) }\label{sec:10}

This section is devoted to the proof of Theorem
\ref{thm-7-11}, which relies on Propositions
\ref{thm-3adf}-\ref{cor-6-7-0} established in previous
sections.

Let $\ell\in\mathbb{N}$ be such that $\frac s2<\ell <\frac
{\alpha r-1}2$. Let
$$ \varepsilon_{X, S,
\nu}:=\inf\left\{\varepsilon\in(0,\infty):\
\nu\left(\left\{S_{2\ell,j}(s,f,
\varepsilon)\right\}_{j\in\mathbb{Z}_+}\right)<\infty\right\}.$$

 First, we prove the upper bound:
\begin{align}\label{6-2a}
 d \left(f, \Lambda_X^{s} \right)_{\Lambda_s }&\lesssim
 \varepsilon_{X,\nu}^0+
\varepsilon_{X, \alpha, \nu}.
\end{align}

By Lemma \ref{thm-711} applied to $r_1=2\ell$, it is enough to
 show that
\begin{align}\label{6-3}
\varepsilon_{X, S,\nu}\lesssim \varepsilon_{X,\alpha, \nu}.
\end{align}

To prove \eqref{6-3}, we let $\varepsilon\in(0,\infty)$ and
use Proposition \ref{cor-6-7-1} to obtain a constant
$R\in(1,\infty)$ such that
\begin{align*}
\nu\left(\left\{S_{2\ell, j}(s,f, \varepsilon)
\right\}_{j\in\mathbb{Z}_+}\right) &\lesssim \sum_{i=1}^m
\nu\left(\left\{\left(D_{\alpha,r,j+i}(s,f,
c\varepsilon)\right)_R
\right\}_{j\in\mathbb{Z}_+}\right)\\
&\lesssim \nu\left(\left\{\left(D_{\alpha,r,j+1}(s,f,
c\varepsilon)\right)_R
\right\}_{j\in\mathbb{Z}_+}\right),
\end{align*}
where $c\in (0, 1)$ is a sufficiently small constant depending
only on the framework parameters.
By Lemma \ref{dda2f}(ii), for any $\delta\in (0, 1)$, the sets
$\left(D_{\alpha, r,j+1}(s,f, c\varepsilon)\right)_\delta$,
$j\in\mathbb{Z}$, has Property I with uniform constants. Thus,
it follows from Definition \ref{Debqf2s} (iii) that, for any
$\delta\in (0, 1)$,
$$\nu\left(\left\{\left(D_{\alpha,r,j+1}(s,f,
c\varepsilon)\right)_R
\right\}_{j\in\mathbb{Z}_+}\right)<\infty\iff
\nu\left(\left\{\left(D_{\alpha,r,j+1}(s,f,
c\varepsilon)\right)_\delta
\right\}_{j\in\mathbb{Z}_+}\right)<\infty.$$
On the other hand, however, by Proposition \ref{thm-3adf} and
both (i) and (ii) of Definition \ref{Debqf2s}, we can find a
constant $\delta\in (0, 1)$, depending on $f$ and
$\varepsilon$, such that
\begin{align*} \nu\left(\left\{\left(D_{\alpha,r,j+1}(s,f,
c\varepsilon)\right)_\delta
\right\}_{j\in\mathbb{Z}_+}\right)&\lesssim
\nu\left(\left\{\left(D_{\alpha,r,j}(s,f,
2^{-1}c\varepsilon)\right)
\right\}_{j\in\mathbb{Z}_+}\right),
\end{align*}
where the last step used the boundedness of the right shift.
Thus, the following implication holds:
$$\nu\left(\left\{\left(D_{\alpha,r,j}(s,f,
2^{-1}c\varepsilon)\right)
\right\}_{j\in\mathbb{Z}_+}\right)<\infty\implies
\nu\left(\left\{S_{2\ell, j}(s,f, \varepsilon)
\right\}_{j\in\mathbb{Z}_+}\right)<\infty,$$
from which the desired inequality \eqref{6-3} follows. This
proves the upper estimate \eqref{6-2a}.

Next, we show the lower estimate:
\begin{align}
\varepsilon_{X,\nu}^0+
\varepsilon_{X, \alpha, \nu}&\lesssim d \left(f, \Lambda_X^{s}
\right)_{\Lambda_s }.\label{modq}
\end{align}
By Lemma \ref{thm-711}, it is enough to show that
\begin{align}\label{6-5}
\varepsilon_{X,\alpha, \nu} \lesssim \varepsilon_{X, s, \nu}.
\end{align}
Using Proposition \ref{cor-6-7-0} and both (i) and (ii) of
Definition \ref{Debqf2s},
we find a constant $R\in(1,\infty)$ such that
\begin{align*}
\nu\left(\left\{D_{\alpha,r,j}(s,f, \varepsilon)
\right\}_{j\in\mathbb{Z}_+}
\right)&\lesssim
\nu\left(\left\{\left(S_{2\ell,j+1}(s,f,
c\varepsilon)\right)_R\right\}_{j\in\mathbb{Z}_+}\right),
\end{align*}
where the constant $c\in (0, 1)$ depends only on the framework
parameters.
Similar to the above proof \eqref{6-3}, using Lemma
\ref{dda2f}(ii) and Definition \ref{Debqf2s}(iii), we conclude
that, for any $\delta\in (0, 1)$,
$$\nu\left(\left\{\left(S_{2\ell,j+1}(s,f,
c\varepsilon)\right)_R
\right\}_{j\in\mathbb{Z}_+}\right)<\infty\iff
\nu\left(\left\{\left(S_{2\ell,j+1}(s,f,
c\varepsilon)\right)_\delta
\right\}_{j\in\mathbb{Z}_+}\right)<\infty.$$
On the other hand, however, by Lemma \ref{thm3adf} and both
(i) and (ii) of
Definition \ref{Debqf2s}, there exists a constant $\delta\in
(0, 1)$, depending on $f$ and
$\varepsilon$, such that
\begin{align*} \nu\left(\left\{\left(S_{2\ell,j+1}(s,f,
c\varepsilon)\right)_\delta
\right\}_{j\in\mathbb{Z}_+}\right)
&\lesssim \nu\left(\left\{\left(S_{2\ell,j}(s,f,
2^{-1}c\varepsilon)\right)
\right\}_{j\in\mathbb{Z}_+}\right).
\end{align*}
Thus, the following implication holds:
$$\nu\left(\left\{\left(S_{2\ell,j}
(s,f, 2^{-1}c\varepsilon)\right)
\right\}_{j\in\mathbb{Z}_+}\right)
<\infty\implies \nu\left(\left\{D_{\alpha,r,j}(s,f,
\varepsilon)
\right\}_{j\in\mathbb{Z}_+}\right)<\infty,
$$
from which \eqref{6-5} follows. This proves the upper estimate
\eqref{modq}
and hence finishes the proof of
Theorem \ref{thm-7-11}.

\section{Applications to specific spaces\label{sea8}}

In this section, we give some specific examples
and show how to apply the main results
obtained in the previous sections
to these examples.

\subsection{Besov spaces}

Now, we present the concept of Besov spaces
as follows; see, for example,
\cite[Definition 1.1]{T20}.
As above, we denote by $\mathcal{S}$
the space of all Schwartz functions on ${\mathbb{R}^n}$ and
$\mathcal{S}'$ its topological dual space (i.e., the space of
all tempered distributions on $\mathbb{R}^n$).
Let $\eta\in\mathcal S$ satisfy $0\leq\eta\leq1,$
\begin{equation*}
\eta\equiv1\ \text{on}\
\left\{x\in\mathbb{R}^n:\ |x|\leq1\right\},\ \text{and}\
\eta\equiv0\ \text{on}\ \left\{x\in\mathbb{R}^n:\ |x|\geq\frac
32\right\}.
\end{equation*}
For any $x\in\mathbb{R}^n$, let
$
a(x):=\eta\left(x\right)-\eta(2x)
$.
Then $\mathop\mathrm{\,supp\,} a\subset
\left\{x\in\mathbb{R}^n:\ 1/2 \leq|x|\leq 3/2\right\}$ and,
for any $x\in\mathbb{R}^n$,
\begin{align*}
1=\lim_{k\to \infty} \eta\left(\frac x {2^k}\right) =\eta(x)
 +\sum_{j=1}^\infty \left[\eta\left(\frac x {2^j}\right)
 -\eta\left(\frac x {2^{j-1}}\right)\right] =\eta(x)
 +\sum_{j=1}^\infty a\left(\frac x {2^j}\right).
\end{align*}
Define $\{\phi_k\}_{k\in\mathbb N}$ by setting, for any
$x\in\mathbb{R}^n$,
\begin{equation}\label{eq-phi1}
\phi_0(x):=\mathcal F^{-1}\eta(x)
\end{equation}
and, for any $k\in\mathbb N$,
\begin{equation}\label{eq-phik}
\phi_k(x):=2^{kn}\mathcal F^{-1}a(2^kx).
\end{equation}
Clearly, $\sum_{k\in\mathbb Z_+} \widehat\phi_k=1$ on
$\mathbb{R}^n$.

\begin{definition}\label{df-Triebel}
Let $p,q\in (0,\infty]$ and $s\in {\mathbb R}$.
Then the
\emph{Besov space} $B^{s}_{p,q}$
is defined to be the set of
all $f\in\mathcal S'$ such that
$$
\|f\|_{B^{s}_{p,q}}:=
\left\{\sum_{j=0}^{\infty}\left[2^{js}
\left\|\phi_j\ast f\right\|_{L^p}\right]^{q}\right\}^{\frac
1q}<\infty,
$$
where the usual modification is made
when $q=\infty$.
\end{definition}

Let $q,p\in(0,\infty]$. The \emph{space
$l^q(L^p)_{\mathbb{Z}_+}$} is
defined to be the set of all
$\mathbf{F}:=\{f_j\}_{j\in\mathbb{Z}_+}
\in \mathscr{M}_{\mathbb{Z}_+}$
such that
\begin{align*}
\|\mathbf{F}\|_{l^q(L^p)_{\mathbb{Z}_+}}
:=\left[\sum_{j\in\mathbb{Z}_+}
\left\|f_j\right\|^q_{L^p}
\right]^{\frac{1}{q}}<\infty,
\end{align*}
where the usual modification is made
when $q=\infty$.
From the definition of the space $l^q(L^p)_{\mathbb{Z}_+}$, it
 is easy to deduce that
$l^q(L^p)_{\mathbb{Z}_+}$ is a quasi-normed lattice of
function sequences.

The following lemma is about the wavelet characterization of
Besov spaces (see, for example, \cite[Proposition 1.11]{T20}).

\begin{lemma}\label{as}
Let $s\in(0,\infty)$ and $p,q\in(0,\infty]$.
Assume that
the regularity parameter $L\in\mathbb N$ of the Daubechies
wavelet system $\{\psi_\omega\}_{\omega\in\Omega}$
satisfies that
$L>\max\{s,n(\max\{\frac 1p,1\}-1)-s\}$.
Then $B^s_{p,q}\cap\Lambda_s
=\Lambda_X^{s}$
with equivalent quasi-norm, where $\Lambda_X^{s}$ is the
Daubechies $s$-Lipschitz $X$-based space with
$X:=l^q(L^p)_{\mathbb{Z}_+}$.
\end{lemma}

We still need the following two
lemmas, which are precisely, respectively, \cite[Lemmas 4.3
and 4.4]{dsy}.

\begin{lemma}\label{as123}
Let $p\in(0,\infty)$ and $q\in(0,\infty]$.
Then $l^q(L^p)_{\mathbb{Z}_+}$ satisfies Assumption \ref{a1}.
\end{lemma}

\begin{lemma}\label{as12k312}
Let $q\in(0,\infty)$.
Then $l^q(L^\infty)_{\mathbb{Z}_+}$ satisfies Assumption
\ref{pplp}.
\end{lemma}

Using Theorems \ref{thm-2-66} and \ref{thm-7-11}
and Lemmas \ref{as}, \ref{as123}, and \ref{as12k312},
we obtain the following conclusions.

\begin{theorem}\label{falsnlf}
Let $\alpha,s\in(0,\infty)$, $p\in(0,\infty)$, and
$q\in(0,\infty]$.
Assume that $r\in\mathbb N$ with $\alpha r>s+3$ and that
the regularity parameter $L\in\mathbb N$ of the Daubechies
wavelet system $\{\psi_\omega\}_{\omega\in\Omega}$ satisfies
that
$L>\max\{s,n(\max\{\frac 1p,1\}-1)-s\}$ and $L\geq r-1$. Let
$\varphi$
be the same as in \eqref{fanofn}. Then the following
statements hold.
\begin{itemize}
\item[\rm(i)]
For any $f\in\Lambda_s$,
\begin{align*}
 d \left(f, B^s_{p,q}
\cap\Lambda_s
\right)_{\Lambda_s}&\sim
\inf\left\{\varepsilon\in(0,\infty):\
\left\{\sum_{j=0}^{\infty}
\left[\mu(D_{\alpha,r,j}(s,f, \varepsilon))\right]^{\frac
qp}\right\}^{\frac 1q}<\infty\right\}\\
&\quad+
\limsup_{k\in\mathbb{Z}^n,\;|k|\rightarrow\infty}\left|
\int_{\mathbb{R}^n}\varphi(x-k)f(x)\,dx\right|
\end{align*}
with positive equivalence constants independent of $f$.

\item[\rm(ii)]
For any $f\in\Lambda_s$ and $q\in(0,\infty)$,
$$
 d \left(f, B^s_{\infty,q}\right)_{\Lambda_s}\sim
\inf\left\{\varepsilon\in(0,\infty):\
\sharp\{j\in\mathbb Z_+:\ D_{\alpha,r,j}(s, f,
\varepsilon)\neq\emptyset\}<\infty
\right\}
$$
with positive equivalence constants independent of $f$.
\end{itemize}
\end{theorem}

Using Corollaries \ref{ppqq} and
\ref{pppz2qq} and
Lemmas \ref{as}, \ref{as123}, and \ref{as12k312},
we obtain the following conclusions.

\begin{theorem}\label{fasnof6}
Under the same assumption as in Theorem \ref{falsnlf},
then the following statements hold.
\begin{itemize}
\item[\rm(i)]
$f\in\overline{B^s_{p,q}\cap\Lambda_s
}^{\Lambda_s}$
if and only if
$f\in \Lambda_s$ and,
for any $\varepsilon\in(0,\infty)$,
 \begin{align*}
\left\{\sum_{j=0}^{\infty}
\left[\mu(D_{\alpha,r,j}(s,f, \varepsilon))
\right]^{\frac qp}\right\}^{\frac 1q}<\infty\
\mathrm{and}\
\limsup_{k\in\mathbb{Z}^n,\;|k|\rightarrow\infty}\left|
\int_{\mathbb{R}^n}\varphi(x-k)f(x)\,dx\right|=0.
\end{align*}
\item[\rm(ii)]
$f\in\overline{B^s_{\infty,q}
}^{\Lambda_s}$
if and only if
$f\in \Lambda_s$ and,
for any $\varepsilon\in(0,\infty)$,
$$
\sharp\{j\in\mathbb Z_+:\ S_{r,j}(s, f,
\varepsilon)\neq\emptyset\}<\infty.
$$
\end{itemize}
\end{theorem}

\subsection{Triebel--Lizorkin spaces}

Next, we recall the definition of Triebel--Lizorkin spaces
as follows; see, for example,
\cite[Definition 1.1]{T20}.
 Let $\{\phi_j\}_{j\in\mathbb{Z}_+}$
be the same as in \eqref{eq-phi1} and \eqref{eq-phik}.

\begin{definition}\label{sanfnlk}
Let $q\in (0,\infty]$ and $s\in {\mathbb R}$.
\begin{itemize}
\item[(i)]
If $p\in (0,\infty)$, then the
\emph{Triebel--Lizorkin space} $F^{s}_{p,q}$
is defined to be the set of
all $f\in\mathcal S'$ such that
$$
\|f\|_{F^{s}_{p ,q}}:=
\left\|\left\{\sum_{j=0}^{\infty}
\left|2^{js}\phi_j\ast f\right|^q\right\}^{\frac
1q}\right\|_{L^p}<\infty.
$$
\item[(ii)]
The
\emph{Triebel--Lizorkin space} $F^{s}_{\infty,q}$
is defined to be the set of
all $f\in\mathcal S'$ such that
$$
\|f\|_{F^{s}_{\infty ,q}}:=
\sup_{\genfrac{}{}{0pt}{}{J\in\mathbb Z_+,M\in\mathbb
Z^n}{I_{J,M}\in\mathcal{D}}
}
\left\{\fint_{I_{J,M}}\sum_{j=J}^{\infty}
\left|2^{js}\phi_j\ast f(x)\right|^q\,dx
\right\}^{\frac 1q}<\infty,
$$
where, for any $J\in\mathbb Z_+$ and $M\in\mathbb
Z^n$, $I_{J,M}:=2^{-J}(M+[0,1)^n)$.
\end{itemize}
\end{definition}

Let $q,p\in(0,\infty]$. The \emph{space}
$L^p(l^q)_{\mathbb{Z}_+}$ is
defined to be the set of all
$\mathbf{F}:=\{f_j\}_{j\in\mathbb{Z}_+}
\in \mathscr{M}_{\mathbb{Z}_+}$
such that
\begin{align*}
\|\mathbf{F}\|_{L^p(l^q)_{\mathbb{Z}_+}}
:=\left\|\left[\sum_{j\in\mathbb{Z}_+}
|f_j|^q
\right]^{\frac{1}{q}}\right\|_{L^p}<\infty,
\end{align*}
where the usual modification is made
when $q=\infty$.
The space $F_{\infty,q}(\mathbb{R}^n,\mathbb{Z}_+)$
is defined to
 be the set of all
$\mathbf{F}:=\{f_j\}_{j\in\mathbb{Z}_+}
\in \mathscr{M}_{\mathbb{Z}_+}$
such that
\begin{align*}
\|f\|_{F_{\infty ,q}(\mathbb{R}^n,\mathbb{Z}_+)}:=
\sup_{l\in\mathbb Z_+,m\in\mathbb Z^n}
\left\{\fint_{I_{l,m}}\sum_{j=l}^{\infty}
\left|f_j(x)\right|^q\,dx\right\}^{\frac 1q}<\infty,
\end{align*}
where the usual modification is made
when $q=\infty$.

From these definitions, it is easy to infer that
$L^p(l^q)_{\mathbb{Z}_+}$ and $F_{\infty,q}(\mathbb{R}^n,\mathbb{Z}_+)$
are quasi-normed lattices of function sequences.

The following lemma is about the semigroup characterization of
Triebel--Lizorkin spaces (see, for example, \cite[Corollary
2]{T20}).

\begin{lemma}\label{asljk}
Let $s\in(0,\infty)$ and $p,q\in(0,\infty]$.
Assume that
the regularity parameter $L\in\mathbb N$ of the Daubechies
wavelet system $\{\psi_\omega\}_{\omega\in\Omega}$ satisfies
that
$L>\max\{s,n(\max\{\frac 1p,\frac 1q,1\}-1)-s\}$.
Then
$F^s_{p,q}\cap\Lambda_s
=\Lambda_X^{s}$,
where $\Lambda_X^{s}$ is the
Daubechies $s$-Lipschitz $X$-based space with
$X:=L^p(l^q)_{\mathbb{Z}_+}$ when $p\neq\infty$ or with
$X:=F_{\infty,q}(\mathbb{R}^n,\mathbb{Z}_+)$ when $p=\infty$.
\end{lemma}

Here, and thereafter, the \emph{Hardy--Littlewood maximal
operator} $\mathcal M$
is defined by setting, for any $f\in
L_{{\mathop\mathrm{\,loc\,}}}^1$
and $x\in\mathbb{R}^n$,
\begin{equation*}
\mathcal M(f)(x):=\sup_{B\ni x}\frac1{|B|}\int_B|f(y)|\,dy,
\end{equation*}
where the supremum is taken over all balls
$B\subset\mathbb{R}^n$
containing $x$.

We need the following
lemma, which is precisely \cite[Lemma 4.11]{dsy}.

\begin{lemma}\label{as123jl}
Let $s,p\in(0,\infty)$ and $q\in(0,\infty]$.
Then $ L^p(l^q)_{\mathbb{Z}_+}$ satisfies Assumption \ref{a1}.
\end{lemma}

For any cube $I\subset\mathbb{R}^n$, we always use $\ell(I)$
to denote its edge length and let $\widehat I:=I\times [0,
\ell(I))$.
For any given measurable set $A\subset \mathbb{R}_+^{n+1}$,
let
\begin{align}\label{asdf}
M(A):=\sup_{I\in \mathcal D} \frac 1 {|I|} \int_I
\left[\int_0^{\ell(I)} \mathbf{1}_A(x,y) \,\frac {dy}
{y}\right]\, dx=\sup_{I\in\mathcal D} \frac 1 {|I|}
\iint_{\widehat I} \frac {\mathbf{1}_A(x,y)}y \,dy\,dx.
\end{align}

Let $d\mu (x,x_{n+1}):=\frac {dx\,dx_{n+1}} {x_{n+1}}.$
Since we assume that the edge length of any cube in $\mathcal
D$ is at most $1$, we deduce that, for any $A\subset
\mathbb{R}_+^{n+1}$,
$M(A)=M\left(A\cap (\mathbb{R}^n\times (0, 1])\right).$
In other words, $M(A)<\infty$ is equivalent to that
$\mathbf{1}_{A\cap (\mathbb{R}^n\times (0, 1])}d\mu$
is a Carleson measure.

We still need the following two
lemmas, which are exactly, respectively, \cite[Lemmas 4.13 and
5.7]{dsy}.

\begin{lemma}\label{de1}
Let $A\subset \mathbb{R}^{n+1}_+$ be a measurable set such
that $M(A)<\infty$, where $M$
is the same as in \eqref{asdf}. Assume that $A$ has Property I
with constants $\delta, \delta'\in (0, 1/10)$.
Then, for any $R\in(0,\infty)$, $M(A_R)<\infty$.
\end{lemma}

\begin{lemma}\label{as12312}
Let $q\in(0,\infty]$.
Then $F_{\infty,q}(\mathbb{R}^n,\mathbb{Z}_+)$ satisfies
Assumption \ref{pplp}.
\end{lemma}

Using Theorems \ref{thm-2-66} and \ref{thm-7-11} and
Lemmas \ref{asljk}, \ref{as123jl}, and \ref{as12312}, we
obtain the following conclusions.

\begin{theorem}\label{bvcaso}
Let $\alpha,s\in(0,\infty)$, $p\in(0,\infty)$, and
$q\in(0,\infty]$.
Assume that $r\in\mathbb N$ with $\alpha r>s+3$
and that
the regularity parameter $L\in\mathbb N$ of the Daubechies
wavelet system $\{\psi_\omega\}_{\omega\in\Omega}$ satisfies
that
$L>\max\{s,n(\max\{\frac1p,\frac1q,1\}-1)-s\}$ and $L\geq
r-1$. Let $\varphi$
be the same as in \eqref{fanofn}.
Then the following statements hold.
\begin{itemize}
\item[\rm(i)]
For any $f\in\Lambda_s$,
 \begin{align*}
 d \left(f, F^s_{p,q}\cap\Lambda_s
\right)_{\Lambda_s}&\sim
\inf\left\{\varepsilon\in(0,\infty):\
\left\|\left[\int_{0}^1 \mathbf{1}_{ D_{\alpha,r}(s, f,
\varepsilon)}(\cdot,y)\,\frac {dy}{y}
\right]^{\frac 1q}\right\|_{L^p}<\infty\right\}\\
&\quad+
\limsup_{k\in\mathbb{Z}^n,\;|k|\rightarrow\infty}
\left|\int_{\mathbb{R}^n}\varphi(x-k)f(x)\,dx\right|
\end{align*}
with positive equivalence constants independent of $f$.

\item[\rm(ii)]
For any $f\in\Lambda_s$,
 \begin{align*}
 d \left(f, F^s_{\infty,q}\right)_{\Lambda_s}&\sim
 \inf\left\{\varepsilon\in(0,\infty):\ \
M(D_{\alpha,r}(s,f, \varepsilon))<\infty\right\}
\end{align*}
with positive equivalence constants independent of $f$.
\end{itemize}
\end{theorem}

\begin{remark}
In (i) and (ii) of Theorem \ref{bvcaso}, if $s\in(0,1]$,
$\alpha=1$, and $q=2$,
then the conclusions are exactly
\cite[Theorems 4]{ss}, while
Theorem \ref{bvcaso} in other
cases is completely new.
\end{remark}

Using Corollaries \ref{ppqq} and
\ref{pppz2qq} and
Lemmas \ref{asljk}, \ref{as123jl}, and \ref{as12312},
we obtain the following conclusions.

\begin{theorem}\label{fasnof4}
Under the same assumption as in Theorem \ref{bvcaso},
then the following statements hold.
\begin{itemize}
\item[\rm(i)]
$f\in\overline{F^s_{p,q}\cap\Lambda_s
}^{\Lambda_s}$
if and only if
$f\in \Lambda_s$ and,
for any $\varepsilon\in(0,\infty)$,
 \begin{align*}
\left\|\left[\int_{0}^1 \mathbf{1}_{ D_{\alpha,r}(s, f,
\varepsilon)}(\cdot,y)\,\frac {dy}{y}
\right]^{\frac 1q}\right\|_{L^p}<\infty\ \mathrm{and}\
\limsup_{k\in\mathbb{Z}^n,\;|k|\rightarrow\infty}
\left|\int_{\mathbb{R}^n}\varphi(x-k)f(x)\,dx\right|=0.
\end{align*}
\item[\rm(ii)]
$f\in\overline{F^s_{\infty,q}
}^{\Lambda_s}$
if and only if
$f\in \Lambda_s$ and,
for any $\varepsilon\in(0,\infty)$,
$M(D_{\alpha,r}(s,f, \varepsilon))<\infty.$
\end{itemize}
\end{theorem}

\subsection{Besov-type spaces}

The Besov-type and Triebel--Lizorkin-type spaces were
intensively investigated in
\cite{ht23,lsuyy12,syy10,yy10,yy13,ysy10,ysy20} and, in
particular,
the Triebel--Lizorkin-type spaces
were introduced in \cite{yy08,yy10,ysy10} to
connect Triebel--Lizorkin spaces and $Q$ spaces.
Furthermore, they also have a close relation
with Besov--Morrey and Triebel--Lizorkin--Morrey
spaces introduced in \cite{ky94,tx},
which
are systematically studied
in \cite{HMS16,HS13,s08,s09,s10,st07,s011,s011a}.
We refer to \cite{syy10,yyz14,yzy15,yzy15-a}
for more studies on Besov-type and Triebel--Lizorkin-type spaces and
to \cite{hl,hlms,hms23,hst23,hst23-2,lz,lz12,mps}
for more variants and their applications.

For any $p$, $q\in\mathbb{R}$, let
$p\vee q:=\max\{p,q\}$ and
$p\wedge q:=\min\{p,q\}$;
furthermore, let
$p_+:=\max\{p,0\}$. Let $\mathscr{Q}$ be the set of
all dyadic cubes.
For any $Q\in\mathscr{Q}$,
let $j_Q:=-\mathop\mathrm{\,log\,}_2\ell(Q).$
Now, we recall the concept of Besov-type spaces.
Let $s\in\mathbb{R}$,
$p,q\in(0,\infty]$, and $\tau\in[0,\infty)$.
The \emph{space}
$(l^qL^p_\tau)_{\mathbb{Z}_+}$ is
defined to be the set of all
$\mathbf{f}:=\{f_j\}_{j\in\mathbb{Z}_+}\in
\mathscr{M}_{\mathbb{Z}_+}$ such that
\begin{align*}
\|\mathbf{f}\|_{(l^qL^p_\tau)_{\mathbb{Z}_+}}
:=\sup_{Q\in\mathscr{Q}}
\left\{\frac{1}{|Q|^\tau}\left[
\sum^\infty_{j=j_Q\vee 0}
\left\|f_j\right\|_{L^p(Q)}^q\right]^{\frac{1}{q}}
\right\}<\infty,
\end{align*}
where the usual modification is made
when $q=\infty$. Let $\{\phi_j\}_{j\in\mathbb{Z}_+}$
be the same as in \eqref{eq-phi1} and \eqref{eq-phik}. The \emph{Besov-type
space $B^{s,\tau}_{p,q}$}
is defined to be the set of all
$f\in \mathcal{S}'$
such that
\begin{align*}
\|f\|_{B^{s,\tau}_{p,q}}:=
\left\|\left\{2^{js}\phi_j*f\right\}_{j\in\mathbb{Z}_+}
\right\|_{(l^qL^p_\tau)_{\mathbb{Z}_+}}<\infty.
\end{align*}
From the definition of the space $(l^qL^p_\tau)_{\mathbb{Z}_+}$,
it is easy to deduce that
$(l^qL^p_\tau)_{\mathbb{Z}_+}$ is a quasi-normed
lattice of function sequences.

The following lemma is about the wavelet characterization of
Besov-type spaces (see, for example, \cite[Section 4.2]{ysy10}
or \cite[Theorem 6.3(i)]{lsuyy12}).

\begin{lemma}\label{as2}
Let $s\in(0,\infty)$,
$p,q\in(0,\infty]$, and $\tau\in[0,\infty)$.
Assume that
the regularity parameter $L\in\mathbb N$ of the Daubechies wavelet system
$\{\psi_\omega\}_{\omega\in\Omega}$ satisfies that
\begin{align}\label{4.8}
L> \left[s+n\left(\tau+\frac{1}{1\wedge p\wedge q}
-\frac{1}{p\vee 1}\right)+\frac{n}{p}\right]\vee
\left[-s+n\left(\tau+
\frac{1}{1\wedge p\wedge q}+\frac{1}{p}
-1\right)+2\frac{n}{p}\right].
\end{align}
Then $B^{s,\tau}_{p,q}\cap\Lambda_s
=\Lambda_X^{s}$,
where $\Lambda_X^{s}$ is the
Daubechies $s$-Lipschitz $X$-based space with $X:=(l^qL^p_\tau)_{\mathbb{Z}_+}$.
\end{lemma}

We still need the following
lemma, which is precisely \cite[Lemma 5.36]{dsy}.

\begin{lemma}\label{as1123}
Let $\tau\in[0,1/p)$, $p\in(0,\infty)$, and $q\in(0,\infty]$.
Then $(l^qL^p_\tau)_{\mathbb{Z}_+}$
satisfies Assumption \ref{a1}.
\end{lemma}

Using Theorem \ref{thm-2-66} and Lemmas \ref{as2} and \ref{as1123}, we obtain the following conclusions.

\begin{theorem}\label{fasnof}
Let $\tau\in[0,1/p)$, $\alpha,s\in(0,\infty)$, $p\in(0,\infty)$, and $q\in(0,\infty]$.
Assume that $r\in\mathbb N$ with $\alpha r>s+3$ and that
the regularity parameter $L\in\mathbb N$ of the Daubechies wavelet system $\{\psi_\omega\}_{\omega\in\Omega}$ satisfies that
\eqref{4.8}
and $L\geq r-1$. Let $\varphi$
be the same as in \eqref{fanofn}.
Let
$$\mathcal{A}:=\left\{\{a_k\}_{k\in\mathbb N}\subset\mathbb Z^n:\ \left\|\left\{
\sum_{k\in\mathbb N}
\mathbf{1}_{I_{0,a_k}}
\right\}_{j\in\mathbb{Z}_+}
\right\|_{(l^qL^p_\tau)_{\mathbb{Z}_+}}=\infty
\right\}.$$
Then,
for any $f\in\Lambda_s$,
 \begin{align*}
& d \left(f, B^{s,\tau}_{p,q}\cap\Lambda_s(
\mathbb{R}^n)\right)_{\Lambda_s}\\
&\quad\sim \inf\left\{\varepsilon\in(0,\infty):\
\sup_{Q\in\mathscr{Q}}
\frac{1}{|Q|^{\tau}}\left\{\sum_{j=j_Q\vee 0}^{\infty}
\left[\mu(D_{\alpha,r,j}(s,f, \varepsilon))\right]^{\frac qp}\right\}^{\frac1q}<\infty\right\}\\&\quad\quad+
\sup_{\{a_k\}_{k\in\mathbb N}\in\mathcal{A}}
\liminf_{k\rightarrow\infty}\left|\int_{\mathbb{R}^n}
\varphi(x-a_k)f(x)\,dx\right|
\end{align*}
with positive equivalence constants independent of $f$.
\end{theorem}

\begin{remark}
Whether Theorem \ref{fasnof} in the case
$p=\infty$ holds true remains unclear at present.
\end{remark}

Using
Corollary \ref{ppqq} and
Lemmas \ref{as2} and \ref{as1123},
we obtain the following conclusions.

\begin{theorem}\label{fasnof3}
Under the same assumption as in Theorem \ref{fasnof},
then $f\in\overline{B^{s,\tau}_{p,q}
\cap\Lambda_s
}^{\Lambda_s}$
if and only if
$f\in \Lambda_s$ and,
for any $\varepsilon\in(0,\infty)$,
 \begin{align*}
\sup_{Q\in\mathscr{Q}}
\frac{1}{|Q|^{\tau}}\left\{\sum_{j=j_Q\vee 0}^{\infty}
\left[\mu(D_{\alpha,r,j}(s,f, \varepsilon))\right]^{\frac qp}\right\}^{\frac1q}<\infty
\end{align*}
and
\begin{align*}
\sup_{\{a_k\}_{k\in\mathbb N}\in\mathcal{A}}
\liminf_{k\rightarrow\infty}\left|\int_{\mathbb{R}^n}
\varphi(x-a_k)f(x)\,dx\right|=0.
\end{align*}
\end{theorem}

\subsection{Triebel--Lizorkin-type spaces}

Next, we present the concept of Triebel--Lizorkin spaces
as follows; see, for example,
\cite{yy08,yy10,ysy10}.
 Let $\{\phi_j\}_{j\in\mathbb{Z}_+}$
be the same as in \eqref{eq-phi1} and \eqref{eq-phik}.
Let $s\in\mathbb{R}$, $p\in(0,\infty)$,
$q\in(0,\infty]$, and $\tau\in[0,\infty)$.
The \emph{space}
$(L^p_\tau l^q)_{\mathbb{Z}_+}$ is
defined to be the set of all
$\mathbf{f}:=\{f_j\}_{j\in\mathbb{Z}_+}\in
\mathscr{M}_{\mathbb{Z}_+}$ such that
\begin{align*}
\|\mathbf{f}\|_{(L^p_\tau l^q)_{\mathbb{Z}_+}}
:=\sup_{Q\in\mathscr{Q}}
\left\{\frac{1}{|Q|^\tau}
\left\|\left(\sum^\infty_{j=j_Q\vee 0}
\left|f_j\right|^q\right)^{\frac{1}{q}}
\right\|_{L^p(Q)}\right\}<\infty,
\end{align*}
where the usual modification is made when $q=\infty$.
The \emph{Triebel--Lizorkin-type
space $F^{s,\tau}_{p,q}$}
is defined to be the set of all
$f\in \mathcal{S}'$
such that
\begin{align*}
\|f\|_{F^{s,\tau}_{p,q}}:=
\left\|\left\{2^{js}\phi_j*f\right\}_{j\in\mathbb{Z}_+}
\right\|_{(L^p_\tau l^q)_{\mathbb{Z}_+}}<\infty.
\end{align*}
From these definitions, it is easy to infer that
$(L^p_\tau l^q)_{\mathbb{Z}_+}$
is a quasi-normed lattice of function sequences.

The following lemma is about the wavelet characterization of
Triebel--Lizorkin-type spaces (see, for example, \cite[Section 4.2]{ysy10} or \cite[Theorem 6.3(i)]{lsuyy12}).

\begin{lemma}\label{as3}
Let $s\in(0,\infty)$,
$p\in(0,\infty)$,
$q\in(0,\infty]$, and $\tau\in[0,\infty)$.
Assume that the regularity parameter $L\in\mathbb N$ of the Daubechies
wavelet system $\{\psi_\omega\}_{\omega\in\Omega}$ satisfies that
\begin{align}\label{4.9}
L&> \left[s+n\left(\tau+\frac{1}{1\wedge p\wedge q}
-\frac{1}{p\vee 1}\right)+\frac{n}{p\wedge q}\right]\nonumber\\
&\quad\vee \left[-s+n\left(\tau+
\frac{1}{1\wedge p\wedge q}+\frac{1}{p}
-1\right)+2\frac{n}{p\wedge q}\right].
\end{align}
Then $F^{s,\tau}_{p,q}\cap
\Lambda_s
=\Lambda_X^{s}$,
where $\Lambda_X^{s}$ is the
Daubechies $s$-Lipschitz $X$-based space
with $X:=(l^qL^p_\tau)_{\mathbb{Z}_+}$.
\end{lemma}

We still need the following
lemma, which is exactly \cite[Lemmas 5.45]{dsy}.

\begin{lemma}\label{as123j2}
Let $s,p\in(0,\infty)$, $\tau\in[0,1/p)$, and $q\in(0,\infty]$.
Then $(l^qL^p_\tau)_{\mathbb{Z}_+}$ satisfies Assumption \ref{a1}.
\end{lemma}

Using Theorem \ref{thm-2-66} and Lemmas \ref{as3} and \ref{as123j2}, we obtain the following conclusions.

\begin{theorem}\label{noanfs}
Let $\alpha,s\in(0,\infty)$, $\tau\in[0,1/p)$,
$p\in(0,\infty)$, and $q\in(0,\infty]$.
Assume that $r\in\mathbb N$ with $\alpha r>s+3$ and that
the regularity parameter $L\in\mathbb N$ of the Daubechies
wavelet system $\{\psi_\omega\}_{\omega\in\Omega}$ satisfies that
\eqref{4.9} and $L\geq r-1$. Let $\varphi$
be the same as in \eqref{fanofn}.
Let
$$\mathcal{A}:=\left\{\{a_k\}_{k\in\mathbb N}\subset\mathbb Z^n:\ \lim_{k\to\infty}|a_k|=\infty,\
\left\|\left\{
\sum_{k\in\mathbb N}\mathbf{1}_{I_{0,a_k}}
\right\}_{j\in\mathbb{Z}_+}\right
\|_{(l^qL^p_\tau)_{\mathbb{Z}_+}}=\infty\right\}.$$
Then, for any $f\in\Lambda_s$,
 \begin{align*}
& d \left(f, F^{s,\tau}_{p,q}\cap\Lambda_s
\right)_{\Lambda_s}\\&\quad\sim \inf\left\{\varepsilon\in(0,\infty):\
\sup_{Q\in\mathscr{Q}}
\frac{1}{|Q|^{\tau}}
\left\|\left[\int_{0}^{2^{-j_{Q} \vee 0}} \mathbf{1}_{ D_{\alpha,r}(s, f, \varepsilon)}(\cdot,y)\,\frac {dy}{y}
\right]^{\frac 1q}\right\|_{L^p(Q)}
<\infty\right\}\\&\quad\quad+
\sup_{\{a_k\}_{k\in\mathbb N}\in\mathcal{A}}
\liminf_{k\rightarrow\infty}\left|\int_{\mathbb{R}^n}
\varphi(x-a_k)f(x)\,dx\right|
\end{align*}
with positive equivalence constants independent of $f$.
\end{theorem}

Using Corollary \ref{ppqq} and Lemmas \ref{as3} and \ref{as123j2}, 
we obtain the following conclusions.

\begin{theorem}\label{noanfs2}
Under the same assumption as in Theorem \ref{noanfs},
then $f\in\overline{F^{s,\tau}_{p,q}
\cap\Lambda_s}^{\Lambda_s}$
if and only if
$f\in \Lambda_s$ and,
for any $\varepsilon\in(0,\infty)$,
\begin{align*}
\sup_{Q\in\mathscr{Q}}
\frac{1}{|Q|^{\tau}}
\left\|\left[\int_{0}^{2^{-j_{Q} \vee 0}} \mathbf{1}_{ D_{\alpha,r}(s, f, \varepsilon)}(\cdot,y)\,\frac {dy}{y}
\right]^{\frac 1q}\right\|_{L^p(Q)}<\infty
\end{align*}
and
\begin{align*}
\sup_{\{a_k\}_{k\in\mathbb N}\in\mathcal{A}}
\liminf_{k\rightarrow\infty}\left|\int_{\mathbb{R}^n}
\varphi(x-a_k)f(x)\,dx\right|=0.
\end{align*}
\end{theorem}

\bigskip

\smallskip

\noindent Feng Dai

\smallskip

\noindent Department of Mathematical and
Statistical Sciences, University of Alberta
Edmonton, Alberta T6G 2G1, Canada

\smallskip

\noindent {\it E-mail}: \texttt{fdai@ualberta.ca}

\bigskip

\noindent Eero Saksman (Corresponding author)

\smallskip

\noindent Department of Mathematics and Statistics, University of Helsinki,
P. O. Box 68, FIN-00014,
Helsinki, Finland

\smallskip

\noindent {\it E-mail}: \texttt{eero.saksman@helsinki.fi}

\bigskip

\noindent Dachun Yang, Wen Yuan and Yangyang Zhang

\smallskip

\noindent Laboratory of Mathematics and Complex Systems
(Ministry of Education of China),
School of Mathematical Sciences, Beijing Normal University,
Beijing 100875, The People's Republic of China

\smallskip

\noindent{\it E-mails:} \texttt{dcyang@bnu.edu.cn} (D. Yang)

\noindent\phantom{{\it E-mails:}} \texttt{wenyuan@bnu.edu.cn} (W. Yuan)

\noindent\phantom{{\it E-mails:}} \texttt{yangyzhang@bnu.edu.cn} (Y. Zhang)


\begin{thebibliography}{99}

\bibitem{bl}J. Bergh and J. L\"ofström,
Interpolation Spaces. An Introduction,
Grundlehren der Mathematischen Wissenschaften 223,
Springer-Verlag, Berlin-New York, 1976.

\vspace{-0.3cm}

\bibitem{b} J. Bourgain, Embedding $L^1$ in $L^1/H^1$,
Trans. Amer. Math. Soc. 278 (1983), 689--702.

\vspace{-0.3cm}

\bibitem{b5}T. A. Bui, Besov and Triebel--Lizorkin spaces for Schr\"odinger operators with inverse square potentials and applications, J. Differential Equations 269 (2020), 641--688.
\vspace{-0.3cm}

\bibitem{b6}T. A. Bui, Hermite pseudo-multipliers on new Besov and Triebel--Lizorkin spaces,
J. Approx. Theory 252 (2020), Paper No. 105348, 16 pp.

\vspace{-0.3cm}

\bibitem{b7}T. A. Bui, T. Q. Bui and X. T. Duong, Decay estimates on Besov and Triebel--Lizorkin spaces of
 the Stokes flows and the incompressible Navier--Stokes flows in
half-spaces, J. Differential Equations 340 (2022), 83--110.
\vspace{-0.3cm}

\bibitem{b8}T. A. Bui and X. T. Duong, Besov and Triebel--Lizorkin spaces associated to Hermite
operators, J. Fourier Anal. Appl. 21 (2015), 405--448.
\vspace{-0.3cm}

\bibitem{b9}T. A. Bui and X. T. Duong, Spectral multipliers of self-adjoint operators on Besov
and Triebel--Lizorkin spaces associated to operators, Int. Math. Res. Not. IMRN
2021 (2021), 18181--18224.

\vspace{-0.3cm}

\bibitem{bb}P. L. Butzer and H. Berens,
Semi-groups of Operators and Approximation,
Die Grundlehren der mathematischen Wissenschaften 145,
Springer-Verlag New York, Inc., New York, 1967.

\vspace{-0.3cm}

\bibitem{cdlsy} P. Chen, X.-T. Duong, J. Li, L. Song
and L. Yan, The Garnett--Jones theorem on BMO spaces
associated with operators and applications,
arXiv: 2304.08606.

\vspace{-0.3cm}

\bibitem{dw}F. Dai and K. Wang,
A note on the equivalences between the averages and the
$K$-functionals related to the Laplacian,
J. Approx. Theory 130 (2004), 38--45.
\vspace{-0.3cm}

\bibitem{dsy}
F. Dai, E. Saksman, D. Yang, W. Yuan and Y. Zhang, Difference and wavelet characterizations of distances from functions in Lipschitz spaces to their subspaces, Submitted or arXiv: 2505.16116.

\vspace{-0.3cm}
\bibitem{dsy2}
F. Dai, E. Saksman, D. Yang, W. Yuan and Y. Zhang, Characterizations of distances in Lipschitz spaces via families of convolution operators, Submitted or arXiv: ???.

\vspace{-0.3cm}

\bibitem{di}Z. Ditzian and K. G. Ivanov,
Strong converse inequalities,
J. Anal. Math. 61 (1993), 61--111.


\vspace{-0.3cm}

\bibitem{FSyy} H. G. Feichtinger, J. Sun, D. Yang and W. Yuan,
A framework of Besov--Triebel--Lizorkin type spaces via ball
quasi-Banach function sequence spaces II: Applications to specific
function spaces, Anal. Appl. (Singap.) (2025), DOI: 10.1142/S0219530525500228, 100 pp.

\vspace{-0.3cm}

\bibitem{gj78}
J. B. Garnett and P. W. Jones, The distance
in BMO to $L^\infty$, Ann. of Math. (2) 108 (1978), 373--393.


\vspace{-0.3cm}
\bibitem{Gr}L. Grafakos, Classical Fourier analysis, Third edition, Graduate Texts in Mathematics, 249. Springer, New York, 2014.

\vspace{-0.3cm}

\bibitem{hl}
D. D. Haroske and Z. Liu,
Generalized Besov-type and Triebel--Lizorkin-type spaces,
Studia Math. 273 (2023), 161--199.

\vspace{-0.3cm}

\bibitem{hlms}
D. D. Haroske, Z. Liu, S. D. Moura and L. Skrzypczak,
Embeddings of generalised Morrey smoothness spaces,
Acta Math. Sin. (Engl. Ser.)
41 (2025), 413--456.

\vspace{-0.3cm}

\bibitem{HMS16}
D. D. Haroske, S. Moura and L. Skrzypczak,
Smoothness Morrey spaces of regular distributions,
and some unboundedness property,
Nonlinear Anal. 139 (2016), 218--244.

\vspace{-0.3cm}

\bibitem{hms23}
D. D. Haroske, S. D. Moura and L. Skrzypczak,
On a bridge connecting Lebesgue and Morrey spaces in view
of their growth properties, Anal. Appl. (Singap.) 22 (2024),
751--790.

\vspace{-0.3cm}

\bibitem{HS13}
D. D. Haroske and L. Skrzypczak, Embeddings
of Besov--Morrey spaces on bounded domains,
Sudia Math. 218 (2013), 119--144.

\vspace{-0.3cm}

\bibitem{hst23}
D. D. Haroske, L. Skrzypczak and H. Triebel,
Nuclear Fourier transforms,
J. Fourier Anal. Appl. 29 (2023), Paper No. 38, 1--28.

\vspace{-0.3cm}


\bibitem{hst23-2} D. D. Haroske, L. Skrzypczak and H. Triebel,
Mapping properties of Fourier transforms, revisited,
Acta Math. Sin. (Engl. Ser.) 41 (2025), 231--254.

\vspace{-0.3cm}
\bibitem{ht23}
D. D. Haroske and H. Triebel, Morrey
smoothness spaces: a new approach,
Sci. China Math. 66 (2023), 1301--1358.

\vspace{-0.3cm}

\bibitem{h22}
M. Hovemann,
Triebel--Lizorkin--Morrey spaces and differences,
Math. Nachr. 295 (2022), 725--761.

\vspace{-0.3cm}

\bibitem{hs20}
M. Hovemann and W. Sickel,
Besov-type spaces and differences,
Eurasian Math. J. 11 (2020), 25--56.


\vspace{-0.3cm}

\bibitem{JN}
F. John and L. Nirenberg, On functions of bounded
mean oscillation, Comm. Pure Appl. Math. 14
(1961), 415--426.

\vspace{-0.3cm}

\bibitem{j}
P. W. Jones, Estimates for the corona problem,
J. Funct. Anal. 39 (1980), 162--181.

\vspace{-0.3cm}

\bibitem{j80} P. W. Jones, Factorization of
$A_p$ weights, Ann. of Math. (2), 111 (1980), 511--530.

\vspace{-0.3cm}

\bibitem{ky94}
H. Kozono and M. Yamazaki,
Semilinear heat equations and the Navier--Stokes
equation with distributions in new
function spaces as initial data,
Comm. Partial Differential Equations 19 (1994), 959--1014.

\vspace{-0.3cm}

\bibitem{lz} P. Li and Z. Zhai, Well-posedness and regularity of
generalized Navier--Stokes equations in some critical $Q$-spaces,
J. Funct. Anal. 259 (2010), 2457--2519.


\vspace{-0.3cm}

\bibitem{lz12} P. Li and Z. Zhai, Riesz transforms
on $Q$-type spaces with application to quasi-geostrophic
equation, Taiwanese J. Math. 16 (2012), 2107--2132.


\vspace{-0.3cm}

\bibitem{lsuyy12}
Y. Liang, Y. Sawano, T. Ullrich, D. Yang and W. Yuan,
New characterizations of Besov--Triebel--Lizorkin--Hausdorff
spaces including coorbits and wavelets,
J. Fourier Anal. Appl. 18 (2012), 1067--1111.

\vspace{-0.3cm}


\bibitem{Me92}
Y. Meyer, Wavelets and Operators, Translated from the 1990 French original by D. H. Salinger,
Cambridge Studies in Advanced Mathematics 37,
Cambridge University Press, Cambridge, 1992.


\vspace{-0.3cm}

\bibitem{mps}
V. Millot, M. Pegon and A. Schikorra,
Partial regularity for fractional harmonic maps into spheres,
Arch. Ration. Mech. Anal. 242 (2021), 747--825.

\vspace{-0.3cm}

\bibitem{ns}
A. Nicolau and O. Soler i Gibert, Approximation in the Zygmund class,
J. Lond. Math. Soc. (2)
101 (2020), 226--246.

\vspace{-0.3cm}

\bibitem{RBCR1} P. Ruiz, F. Baudoin, L. Chen, L. G. Rogers,
N. Shanmugalingam, A. Teplyaev, Besov class via heat semigroup
on Dirichlet spaces I: Sobolev type inequalities,
J. Funct. Anal. 278 (2020), Paper No. 108459, 48 pp.

\vspace{-0.3cm}

\bibitem{RBCR2} P. Ruiz, F. Baudoin, L. Chen, L. Rogers,
N. Shanmugalingam, A. Teplyaev, Besov class via heat semigroup
on Dirichlet spaces II: BV functions and Gaussian heat kernel
estimates, Calc. Var. Partial Differential Equations 59 (2020),
Paper No. 103, 32 pp.

\vspace{-0.3cm}

\bibitem{ss} E. Saksman and O. Soler i Gibert,
Approximation in the Zygmund and H\"older classes on $\mathbb{R}^n$,
Canad. J. Math. 74 (2022), 1745--1770.

\vspace{-0.3cm}

\bibitem{s08}
Y. Sawano,
Wavelet characterization of Besov--Morrey
and Triebel--Lizorkin--Morrey spaces,
Funct. Approx. Comment. Math. 38 (2008), 93--107.

\vspace{-0.3cm}

\bibitem{s09}
Y. Sawano,
A note on Besov--Morrey spaces and
Triebel--Lizorkin--Morrey spaces,
Acta Math. Sin. (Engl. Ser.) 25 (2009), 1223--1242.

\vspace{-0.3cm}

\bibitem{s10}
Y. Sawano,
Besov--Morrey spaces and
Triebel--Lizorkin--Morrey spaces on domains,
Math. Nachr. 283 (2010), 1456--1487.

\vspace{-0.3cm}

\bibitem{st07}
Y. Sawano and H. Tanaka,
Decompositions of Besov--Morrey spaces
and Triebel--Lizorkin--Morrey spaces,
Math. Z. 257 (2007), 871--905.

\vspace{-0.3cm}

\bibitem{syy10}
Y. Sawano, D. Yang and W. Yuan,
New applications of Besov-type and
Triebel--Lizorkin-type spaces,
J. Math. Anal. Appl. 363 (2010), 73--85.

\vspace{-0.3cm}

\bibitem{s011}
W. Sickel, Smoothness spaces related to Morrey spaces--a survey. I,
Eurasian Math. J. 3 (2012), 110--149.

\vspace{-0.3cm}

\bibitem{s011a}
W. Sickel, Smoothness spaces related to Morrey spaces--a survey. II,
Eurasian Math. J. 4 (2013), 82--124.

\vspace{-0.3cm}

\bibitem{PS}
P. R. Stinga, User’s guide to the fractional Laplacian
and the method of semigroups, arXiv: 1808.05159.

\vspace{-0.3cm}

\bibitem{syy} J. Sun, D. Yang and W. Yuan,
A framework of Besov--Triebel--Lizorkin type spaces via ball
quasi-Banach function sequence spaces I: Real-variable
characterizations,
Math. Ann. 390 (2024), 4283--4360.

\vspace{-0.3cm}

\bibitem{tx} L. Tang and J. Xu, Some properties of Morrey type Besov--Triebel spaces,
Math. Nachr. 278 (2005), 904--917.

\vspace{-0.3cm}

\bibitem{T83}H. Triebel, Theory of Function Spaces, Monographs
in Mathematics 78, Birkh\"auser Verlag, Basel, 1983.

\vspace{-0.3cm}

\bibitem{t92}H. Triebel,
Theory of Function Spaces. II,
Monographs in Mathematics 84,
Birkh\"auser Verlag, Basel, 1992.

\vspace{-0.3cm}

\bibitem{t06}
H. Triebel,
Theory of Function Spaces. III, Monographs in Mathematics 100,
Birkh\"auser Verlag, Basel, 2006.

\vspace{-0.3cm}

\bibitem{t08}
H. Triebel, Function Spaces and Wavelets on Domains, European Math. Soc.
Publishing House, Z\"urich, 2008.

\vspace{-0.3cm}

\bibitem{t10}
H. Triebel, Bases in Function Spaces, Sampling, Discrepancy, Numerical
Integration, European Math. Soc. Publishing House, Z\"urich, 2010.

\vspace{-0.3cm}

\bibitem{T20}H. Triebel, Theory of Function Spaces IV,
Monographs in Mathematics 107, Birkh\"auser/ Springer, Cham, 2020.

\vspace{-0.3cm}

\bibitem{yy08}
D. Yang and W. Yuan,
A new class of function spaces connecting
Triebel--Lizorkin spaces and $Q$ spaces,
J. Funct. Anal. 255 (2008), 2760--2809.

\vspace{-0.3cm}

\bibitem{yy10}
D. Yang and W. Yuan,
New Besov-type spaces and Triebel--Lizorkin-type
spaces including $Q$ spaces,
Math. Z. 265 (2010), 451--480.

\vspace{-0.3cm}

\bibitem{yy13}
D. Yang and W. Yuan,
Relations among Besov-type spaces,
Triebel--Lizorkin-type spaces and
generalized Carleson measure spaces,
Appl. Anal. 92 (2013), 549--561.

\vspace{-0.3cm}

\bibitem{yyz14}
D. Yang, W. Yuan and C. Zhuo,
Musielak--Orlicz Besov-type and
Triebel--Lizorkin-type spaces,
Rev. Mat. Complut. 27 (2014), 93--157.

\vspace{-0.3cm}

\bibitem{yzy15}
D. Yang, C. Zhuo and W. Yuan,
Besov-type spaces with variable smoothness
and integrability,
J. Funct. Anal. 269 (2015), 1840--1898.

\vspace{-0.3cm}

\bibitem{yzy15-a}
D. Yang, C. Zhuo and W. Yuan,
Triebel--Lizorkin type spaces with variable exponents,
Banach J. Math. Anal. 9 (2015), 146--202.

\vspace{-0.3cm}

\bibitem{ysy10}
W. Yuan, W. Sickel and D. Yang,
Morrey and Campanato Meet Besov, Lizorkin and Triebel,
Lecture Notes in Mathematics 2005,
Springer-Verlag, Berlin, 2010.

\vspace{-0.3cm}

\bibitem{ysy20}
W. Yuan, W. Sickel and D. Yang,
The Haar system in Besov-type spaces,
Studia Math. 253 (2020), 129--162.

\end{thebibliography}
\end{document}